\documentclass[10pt]{coulonpaper}

\usepackage{coulondrawing}
\usepackage{tikz-cd}
\tikzcdset{labels={font=\everymath\expandafter{\the\everymath\textstyle}}}
\usetikzlibrary[arrows.meta]
\usepackage{pinlabel}

\DeclareMathOperator*{\simop}{\sim}
\renewcommand{\phi}{\oldphi}

\title{Patterson-Sullivan theory for groups with a strongly contracting element}
\author{Rémi Coulon}
\date{\today}

\setbftrue
\intvaldfrenchtrue

\begin{document}

\maketitle

\begin{abstract}
	Using Patterson-Sullivan measures we investigate growth problems for groups acting on a metric space with a strongly contracting element.
	
\end{abstract}

\noindent
{\footnotesize 
\textbf{Primary MSC Classification.}   
20F65, 
57S30, 
37A35, 
37A15, 
37D40. 
}
\tableofcontents

\section{Introduction}

Let $G$ be a group acting properly by isometries on a proper geodesic space $(X, \distV)$.
In particular $G$ is countable. 
It will always be endowed with the discrete topology.
Its \emph{exponential growth rate} measures the size of its orbits and is defined as
\begin{equation*}
	\omega(G,X) = \limsup_{\ell \to \infty} \frac 1\ell \ln  \card{\set{g \in G}{\dist o{go} \leq \ell}}.
\end{equation*}
This number does depend on the space $X$.
Nevertheless, if the context is clear, we simply write $\omega_G$ instead of $\omega(G,X)$.
It is also the critical exponent of the \emph{Poincaré series} of $G$ defined by 
\begin{equation*}
	\mathcal P_G(s) = \sum_{g \in G} e^{-s\dist o{go}},
\end{equation*}
that is $\mathcal P_G(s)$ diverges (\resp converges) whenever $s < \omega_G$ (\resp $s > \omega_G$).
If $G$ is the fundamental group of a hyperbolic manifold $M$ acting on the universal cover $X = \tilde M$, then $\omega_G$ has numerous interpretations: it is the entropy of the geodesic flow, the Hausdorff dimension of the radial limit set of $G$, etc.
In this context, the exponential growth rate is a central object connecting geometry, group theory, dynamical systems, etc.

\paragraph{Growth spectrum.}
In this article we are interested in the (\emph{normal}) \emph{subgroup growth spectrum} of $G$, i.e. the set
\begin{equation*}
	\spec{G, X} =  \set{\omega(N,X)}{N \vartriangleleft G}.
\end{equation*}
Note that $\spec{G, X}$ is contained in $\intval 0{\omega_G}$.
In particular, $\omega_N = 0$ (\resp $\omega_N = \omega_G$) if $N$ is finite (\resp has finite index in $G$).
A natural question, which has received much attention, is to understand more precisely the extremal values of this set.
This problem is rather well understood if $G$ is a group acting properly, co-compactly by isometries on a Gromov hyperbolic metric space $X$.
For instance, any infinite normal subgroup  $N \vartriangleleft G$ satisfies
\begin{equation*}
	\frac 12 \omega(G,X) < \omega(N,X) \leq \omega(G,X).
\end{equation*}
see Matsuzaki-Yabuki-Jaerisch \cite{Matsuzaki:2020ue}.
Moreover, the second inequality is an equality if and only if $G/N$ is amenable.
Similarly, if $G$ has Kazhdan property (T), then $\omega_N$ cannot be arbitrarily close to $\omega_G$, unless $N$ has finite index in $G$.
See Coulon-Dougall-Schapira-Tapie \cite{Coulon:2018aa} and the references therein.

\paragraph{Contracting elements.}
In the past decades, there have been many efforts to investigate the marks of negative curvature in groups, beyond the context of hyperbolic spaces.
One notion which has emerged is the one of contracting element.
Roughly speaking a subset $Y \subset X$ is contracting, if any ball disjoint from $Y$ has a projection onto $Y$, whose diameter is uniformly bounded \cite{Bestvina:2009hh}.

\begin{rema*}
	In the literature, this property is sometimes called \emph{strong} contraction to distinguish it from a weaker version involving a certain system of non-geodesic paths, see for instance \cite{Arzhantseva:2015cl}.
	Since we will work with this single notion, we simply call it contraction.
\end{rema*}

An element $g \in G$ is \emph{contracting} if the orbit map  $\Z \to G$ sending $n$ to $g^nx$ is a quasi-isometric embedding with contracting image.
Contracting elements can be thought of as hyperbolic directions in the space $X$.
Here are a few examples of group actions with contracting elements.
\begin{itemize}
	\item If $X$ is a hyperbolic space (in the sense of Gromov) endowed with a proper action of $G$, then every loxodromic element in $G$ is contracting \cite{Gromov:1987tk}.
	This is typically the case if $G$ is the fundamental group of a manifold $M$ whose sectional curvature is negative and bounded away from zero, and $X$ the universal cover of $M$.
	Any metric quasi-isometric to this one will also work.
	\item Assume that $G$ is hyperbolic relative to $\{P_1, \dots, P_m\}$.
	Suppose that $G$ acts properly, co-compactly on $X$ (e.g. $X$ is the Cayley graph of $G$ with respect to a finite generating set of $G$).
	Any infinite order element of $G$ which is not conjugated in some $P_i$ is contracting  \cite{Sisto:2012um,Gerasimov:2016uc}.
	\item If $X$ is a CAT(0) space endowed with a proper, co-compact action of $G$, then any rank one element of $G$ is contracting \cite{Bestvina:2009hh}.
	Recall that the universal cover of any closed, compact manifold with non-positive sectional curvature is CAT(0) \cite{Bridson:1999ky}.
	\item Let $\Sigma$ be a closed, compact surface.
	Assume that $G$ is the mapping class group of $\Sigma$ and $X$ the Teichmüller space of $\Sigma$ endowed with the Teichmüller metric.
	Then any pseudo-Anosov element is contracting \cite{Minsky:1996wx}.
\end{itemize}

Groups with a contracting element are known to be acylindrically hyperbolic, see Sisto \cite{Sisto:2018uc}.
Acylindrical hyperbolicity is a powerful tool for studying the structure of a given group.
Nevertheless it is rather useless here as the exponential growth rate $\omega(G,X)$ heavily depends on the metric space $X$.
The typical spaces $X$ we are interested in are indeed not hyperbolic.

\paragraph{Main results.}
The goal of this article is to investigate the extremal values of the subgroup growth spectrum in the context of group actions admitting a contracting element.
Some of our results refine existing statements in the literature.
In particular, we answer most of the questions raised by Arzhantseva and Cashen in \cite{Arzhantseva:2020aa}.
Our main contribution though is the method that we use: we extend to this context the construction of Patterson-Sullivan measures (see below).

When it comes to counting problems, the behavior of the Poincaré series of $G$ at the critical exponent plays a major role.
This motivates the following definition.
The action of $G$ on $X$ is \emph{divergent} (\resp \emph{convergent}) if the Poincaré series $\mathcal P_G(s)$ of $G$ diverges (\resp converges) at $s = \omega_G$.
Our first statement deals with the bottom of the subgroup growth spectrum.

\begin{theo}[see \autoref{res: improved lower bound growth normal sbgp} and \autoref{res: strict half inequality}]
\label{intro: improved lower bound growth normal sbgp}
	Let $X$ be a proper geodesic metric space.
	Let $G$ be a group acting properly, by isometries on $X$ with a contracting element.
	Let $N$ be an infinite normal subgroup of $G$.
	Then
	\begin{equation*}
		\omega(N,X) + \frac 12 \omega(G/N, X/N) \geq \omega(G, X).
	\end{equation*}
	Assume in addition that $G$ is not virtually cyclic and the action of $G$ is divergent.
	Then 
	\begin{equation*}
		\omega(N,X) > \frac 12 \omega(G,X).
	\end{equation*}
\end{theo}

\begin{rema*}
	The first inequality was proved by Matsuzaki and Jaerisch when $G$ is a finitely generated free group acting on its Cayley graph with respect to a free basis \cite{Jaerisch:2017kg}.
	Their method involves fine estimates of the Cheeger constant and the spectral radius of the random walk in $G/N$.
	To the best of our knowledge this result is new, even if $G$ is a hyperbolic group.
	
	The second inequality is well known in the context of hyperbolic spaces, see Roblin \cite{Roblin:2005fn} and Matsuzaki-Yabuki-Jaerisch \cite{Matsuzaki:2020ue}.
	For groups acting with a contracting element, it was proved by Arzhantseva and Cashen under the stronger assumption that $G$ has pure exponential growth, that is when the map	
	\begin{equation*}
		\ell \mapsto \card{\set{g \in G}{\dist o{go} \leq \ell}} e^{-\omega_G \ell}
	\end{equation*}
	is bounded from above and away from zero \cite{Arzhantseva:2020aa}.
	Note that even if $X$ is Gromov hyperbolic, there are groups $G$ acting on $X$, which are divergent but do not have pure exponential growth.
\end{rema*}

The next two results focus on the top of the subgroup growth spectrum.
Let $Q$ be a discrete group.
The left action of $Q$ on itself induces an action of $Q$ on $\ell^\infty(Q)$.
The group $Q$ is \emph{amenable} if there exists a $Q$-invariant mean $\ell^\infty(Q) \to \R$.

\begin{theo}[see \autoref{res: amenability roblin}]
\label{intro: amenability roblin}
	Let $X$ be a proper, geodesic, metric space.
	Let $G$ be a group acting properly, by isometries on $X$ with a contracting element.
	Let $N$ be a normal subgroup of $G$.
	If $G/N$ is amenable, then $\omega(N,X) = \omega(G,X)$.
\end{theo}

\begin{rema*}
	This type of results has a long history.
	Assume that $G$ is the fundamental group of a compact hyperbolic manifold and $X$ the universal cover of $M$.
	Let $N$ be a normal subgroup of $G$.
	Brooks proved that $\omega_N = \omega_G$ if and only if $G/N$ is amenable \cite{Brooks:1981jp} -- Brooks' result is actually stated in terms of the bottom spectra of certain Laplace's operators, but they can be related to the growth rates of the groups via Sullivan's formula \cite{Sullivan:1987bt}.
	A similar statement was obtained independently by Grigorchuk and Cohen when $G$ is a free group acting on its Cayley graph $X$ with respect to a free basis \cite{Grigorchuk:1980wx,Cohen:1982gt}.
	The ``easy direction'' stated above was generalized by Roblin to the settings of CAT($-1$) spaces \cite{Roblin:2005fn}.
\end{rema*}

Recall that two subgroups $H_1$ and $H_2$ of $G$ are \emph{commensurable} if $H_1 \cap H_2$ has finite index in both $H_1$ and $H_2$.
A subgroup $H \subset G$ is \emph{commensurated}, if $H$ and $gHg^{-1}$ are commensurable, for every $g \in G$.
The class of commensurated subgroups contains all normal subgroups and finite index subgroups of $G$.
More generally any subgroup of $G$ that is commensurable with a normal subgroup of $G$ is commensurated.
But they are numerous other examples, see \autoref{rem: commensurated subroups}.

\begin{theo}[see \autoref{res: div normal sbgp}]
\label{intro: div normal sbgp part 1}
	Let $X$ be a proper, geodesic, metric space.
	Let $G$ be a group acting properly, by isometries on $X$ with a contracting element.
	Let $H$ be a commensurated subgroup of $G$.
	If the action of $H$ on $X$ is divergent, then $\omega(H,X) = \omega(G,X)$ and the action of $G$ on $X$ is divergent.
\end{theo}

\begin{rema*}
	To the best of our knowledge the statements in the literature only cover the case where $H$ is normal.
	With this stronger assumption, it was proved by Matsuzaki and Yabuki, if $G$ is a kleinian group, and generalized  by  Matsuzaki, Yabuki and Jaerisch when $X$ is Gromov hyperbolic \cite{Matsuzaki:2009vs, Matsuzaki:2020ue}.
\end{rema*}

\paragraph{Patterson-Sullivan theory.}
Assume that $G$ is the fundamental group of a closed riemannian manifold $M$ with negative sectional curvature.
In this context dynamical systems -- 
first and foremost the study of the geodesic flow on the unit tangent bundle of $M$ -- provide efficient tools to tackle counting problems.
For instance, using the dynamics of the geodesic flow, Margulis proved that the number $c(\ell)$ of simple closed geodesics on $M$ of length at most $\ell$ behaves like
\begin{equation*}
	c(\ell) \simop_{\ell \to \infty} \frac{e^{\omega_G \ell}}{\omega_G \ell}.
\end{equation*}
See \cite{Margulis:1969ve}.
Fix a base point $o \in X$.
Denote by $X = \tilde M$ the universal cover of $M$ and $\partial X$ its visual boundary.
In this topic, measures on the boundary play a prominent role.
Recall that a $G$-invariant, $\omega_G$-conformal density is a collection $\nu = (\nu_x)_{x \in X}$ of non-zero finite measures on $X \cup \partial X$, all in the same measure class, satisfying the following properties: for all $g \in G$, for all $x,y \in X$, we have
\begin{itemize}
	\item $g_\ast \nu_x = \nu_{gx}$ (invariance),
	\item $\displaystyle \frac{d\nu_x}{d\nu_y}(\xi) = e^{-\omega_G b_\xi(x,y)}$  $\nu$-almost everywhere (conformality),
\end{itemize}
where $b_\xi$ stands for the Buseman cocycle at $\xi \in \partial X$.
In particular $\omega_G$ can be interpreted as the dimension of the measure $\nu_o$.
Patterson's construction provides examples of such densities which are supported on $\partial X$.
These measures are designed so that the action of $G$ on $(\partial X, \nu_o)$ captures many properties of the geodesic flow on $M$.
The theory can be generalized for groups acting on a Gromov hyperbolic space, see for instance Coornaert \cite{Coornaert:1993uv} and Bader-Furman \cite{Bader:2017te}.
For such groups, Theorem~\ref{intro: improved lower bound growth normal sbgp}, \ref{intro: amenability roblin} and \ref{intro: div normal sbgp part 1} can be proved using invariant conformal densities.

In the past years, growth problems in groups with a contracting element have been investigated by various people, see for instance \cite{Yang:2014aa, Arzhantseva:2015cl,Dahmani:2019aa,Yang:2019wa,Yang:2020ub,Arzhantseva:2020aa,Li:2020aa}.
Since no Patterson-Sullivan theory existed in this context, each time the authors developed ad hoc methods.
Actually they often make a point of avoiding ``fairly sophisticated'' results about Patterson-Sullivan ``machinery''.
We adopt here an opposite point of view.
For us, these results witnessed the fact that a Patterson-Sullivan theory should exist.
Building this ``missing'' theory is the purpose of this work.
If the ambient space $X$ is CAT(0) this task has been achieved by Link \cite{Link:2018ue} extending the work of Knieper \cite{Knieper:1997us, Knieper:1998vj}.
Our approach does not require any CAT(0) assumption though.
Our goal is to stress that this construction is particularly robust and requires very little hypotheses, beside the existence of a contracting element.
Once the basic properties of invariant conformal densities have been established, they provide a unified framework for solving various growth problems.
We believed that these tools can be used for many other applications inspired by non-positive curvature.

\paragraph{Strategy.}
We would like to understand the behavior of certain densities supported on the ``boundary at infinity'' of $X$.
Thus, the first task is to build an appropriate compactification of $X$ to carry these measures.
There have been many attempts to build an analogue of the Gromov boundary for groups with a contracting element: the contracting and Morse boundaries \cite{Charney:2015dn,Murray:2019vd,Cordes:2015tr}, the sublinearly Morse boundary \cite{Qing:2022wg,Qing:2020wx}, etc.
However these boundaries are sometimes ``too small'' (for instance the Morse boundary cannot be used as a topological model of the Poisson boundary) and often not compact.
This can be a difficulty to build Patterson-Sullivan measures.
Instead, we choose to work with the horocompactification.
In short, it is the ``smallest'' compactification $\bar X$ of $X$ such that the map
\begin{equation*}
	\begin{array}{ccc}
		X \times X \times X & \to & \R \\
		(x,y,z) & \mapsto & \dist xz - \dist yz
	\end{array}
\end{equation*}
extends continuously to a map $X \times X \times \bar X \to \R$.
The horoboundary of $X$ is $\partial X = \bar X \setminus X$.
A point in the horoboundary is a cocycle $c \colon X \times X \to \R$, playing the role of a Buseman cocycle.
Hence this choice is natural to give a rigorous sense to conformal densities.
If $X$ is CAT(0), then the horoboundary coincides with the visual boundary.
In general, this boundary is slightly too large though for invariant conformal densities to behave as expected.
Let us illustrate this fact with an example.

\begin{exam*}
	Consider a group $G$ acting properly, co-compactly, by isometries on a CAT(0) space $X_0$.
	Build a new space $X = X_0 \times \intval 01$ endowed with the $L^1$-metric.
	Let $G$ act trivially on $\intval 01$ and consider the diagonal action of $G$ on $X$.
	This action is still proper and co-compact and $\omega(G, X) = \omega(G, X_0)$.
	The horoboundary of $X$ is homeomorphic to $\partial X = \partial X_0 \times \intval 01$.
	To carry the analogy with negatively curved manifold, we would like that if $\mu = (\mu_x)$ is a $G$-invariant, $\omega_G$-conformal density supported on $\partial X$ then the action of $G$ on $(\partial X, \mu_o)$ is ergodic.
	However, in this example we can choose a $G$-invariant, $\omega_G$-conformal density $\nu = (\nu_x)$ on $\partial X_0$ and form the average $\mu = (\nu^0 + \nu^1)/2$, where $\nu^i$ is a copy of $\nu$ supported on $\partial X_0 \times \{i\}$.
	Then the action of $G$ on $(\partial X, \mu_o)$ is not ergodic.
\end{exam*}

This issue already arises if $X$ is Gromov hyperbolic.
In this context, it can be fixed by passing to the reduced horoboundary.
Endow the horoboundary $\partial X$ with the equivalence relation $\sim$ defined as follows: two cocycles $c, c' \in \partial X$ are equivalent, if $\norm[\infty]{c-c'} < \infty$.
The reduced horoboundary is the quotient $\partial X / \sim$.
If $X$ is hyperbolic, then it coincides with the Gromov boundary.
Moreover the projection $\pi \colon \partial X \onto \partial X/\!\! \sim$ is very well understood, see Coornaert-Papadopoulos \cite{Coornaert:2001ff}.
Pushing forward in $\partial X/\!\! \sim$, the densities built in $\partial X$ provides well behaved measures.

However, in general the reduced horoboundary $\partial X /\!\!  \sim$ is a rather nasty topological space.
For instance, if $X = \R^2$ is endowed with the taxicab metric, then $\partial X /\!\!  \sim$ is not even Hausdorff.
To bypass this difficulty we adopt a measure theoretic point of view.
Denote by $\mathfrak R$ the $\sigma$-algebra that consists of all Borel sets which are saturated for the equivalence relation $\sim$.
We make an abuse of vocabulary and call the measurable space $(\partial X, \mathfrak R)$ the \emph{reduced horoboundary}.
When restricted to the reduced horoboundary, the invariant, conformal densities are well behaved.
For instance, we prove the following partial form of the Hopf-Tsuji-Sullivan dichotomy (we refer the reader to \autoref{sec: first application} for the definition of the radial limit set).

\begin{theo}[see \autoref{res: lower bound dimension density} and \autoref{res: ctg limit set with full measure}]
\label{intro: dichotomy}
	Let $X$ be a proper geodesic metric space and $o \in X$.
	Let $G$ be a group acting properly, by isometries on $X$ with a contracting element.
	Suppose that $G$ is not virtually cyclic.
	Let $\omega \in \R_+$.
	Let $\mu = (\mu_x)$ be the restriction to the reduced horoboundary $(\partial X, \mathfrak R)$ of  a $G$-invariant, $\omega$-conformal density.
	The following are equivalent.
	\begin{enumerate}
		\item The action of $G$ on $X$ is divergent (and thus $\omega = \omega_G$).
		\item $\mu_o$ gives positive measure to the radial limit set.
		\item $\mu_o$ gives full measure to the radial limit set.
	\end{enumerate}
\end{theo}

\begin{rema*}
	In a forthcoming work, see \cite{Coulon:2023aa}, we plan to complete the Hopf-Tsuji-Sullivan dichotomy by investigating the ergodicity of the geodesic flow in this context and its consequences for growth problems.
\end{rema*}

If the action of $G$ on $X$ is divergent, we prove that invariant, conformal densities are ergodic and essentially unique, when restricted to the reduced horoboundary.

\begin{theo}[see \autoref{res: quasi-conf + ergo}]
\label{intro: quasi-conf + ergo}
	Let $X$ be a proper, geodesic, metric space and $o \in X$.
	Let $G$ be a non virtually cyclic group acting properly, by isometries on $X$ with a contracting element.
	Assume that the action of $G$ on $X$ is divergent.
	Let $\mu = (\mu_x)$ be the restriction to the reduced horoboundary $(\partial X, \mathfrak R)$ of  a $G$-invariant, $\omega_G$-conformal density.
	Then
	\begin{enumerate}
		\item $\mu_o$ is ergodic;
		\item $\mu_o$ is non-atomic;
		\item $\mu$ is almost unique in the following sense:
		there is $C \in \R_+^*$, such that if $\mu' = (\mu'_x)$ is the restriction to the reduced horoboundary of another $G$-invariant, $\omega_G$-conformal density, then for every $x \in X$, we have $\mu'_x \leq C \mu_x$.
	\end{enumerate}
\end{theo}

Finally, we complete \autoref{intro: div normal sbgp part 1} as follows.

\begin{theo}[see \autoref{res: div normal sbgp}]
\label{intro: div normal sbgp part 2}
	Let $X$ be a proper, geodesic, metric space.
	Let $G$ be a group acting properly, by isometries on $X$ with a contracting element.
	Suppose that $G$ is not virtually cyclic.
	Let $H$ be a commensurated subgroup of $G$.
	If the action of $H$ on $X$ is divergent, then any $H$-invariant, $\omega_H$-conformal density is $G$-almost invariant when restricted to the reduced horoboundary $(\partial X, \mathfrak R)$.
\end{theo}

\begin{rema*}
	Other applications can be found in Sections~\ref{sec: first application} and \ref{sec: more applications}.
	In this article we focused on growth problems.
	Nevertheless, we believe that the tools we introduced can be used for other purposes, e.g. to generalize the ``no proper conjugation'' property of divergent subgroups exhibited by Matsuzaki, Yabuki and Jaerisch \cite{Matsuzaki:2020ue}.
\end{rema*}

\paragraph{Strongly positively recurrent actions.}
As we mentioned before, divergent actions play an important role in counting problems.
Any proper and co-compact action is divergent.
In particular, if $G$ acts properly on $X$ with a quasi-convex orbit, then its action is divergent.
This framework has been generalized independently by Schapira-Tapie \cite{Schapira:2021ti} and Yang \cite{Yang:2019wa} under the names \emph{strongly positively recurrent} action (SPR) and \emph{statistically convex co-compact} action (SCC) respectively -- the idea also implicitly appears in the work of Arzhantseva-Cashen-Tao \cite{Arzhantseva:2015cl}.
The notion has an independent dynamical origin as well, see for instance \cite{Gurevich:1998it,Sarig:2001cd}.
Roughly speaking the idea is to ask that the elements of $g \in G$ which ``violate'' the quasi-convexity of $G$ are statistically very rare.
It was proved by Yang that such actions are divergent.
In \autoref{sec: spr} we provide an alternative proof of this fact in the spirit of Schapira-Tapie \cite{Schapira:2021ti}.

\begin{rema*}
	Since obtaining the results in this article, we have learned that Wenyuan Yang independently investigated conformal measures in the same context \cite{Yang:2022wm}.
	The techniques used by Wenyuan Yang are slightly different.
	For instance, his proof of \autoref{intro: quasi-conf + ergo} (partial form of the Hopf-Tsuji-Sullivan dichotomy) relies on projection complexes introduced by Bestvina, Bromberg, and Fujiwara \cite{Bestvina:2015tv}.
	In contrast, we tried to use whenever possible low-tech arguments (both of geometric and measure theoretic nature).
\end{rema*}

\paragraph{Acknowledgment.}
The author is grateful to the \emph{Centre Henri Lebesgue} ANR-11-LABX-0020-01 for creating an attractive mathematical environment.
He acknowledges support from the Agence Nationale de la Recherche under Grant \emph{Dagger} ANR-16-CE40-0006-01 and \emph{GoFR} ANR-22-CE40-0004.
He thanks the Institut Henri Poincaré (UAR 839 CNRS-Sorbonne Université) for its hospitality and support (through LabEx CARMIN, ANR-10-LABX-59-01) during the trimester program \emph{Groups Acting on Fractals, Hyperbolicity and Self-Similarity} in Spring 2022.
The author warmly thanks Françoise Dal'bo, Ilya Gekhtman, Andrea Sambusetti, Barbara Schapira, Samuel Tapie, and Wenyuan Yang for related discussions.
He also thanks the referee for the helpful comments.

\section{Groups with a contracting element}
\label{sec: contracting}

\paragraph{Notations and vocabulary.}
In this article $(X,d)$ is a proper, metric space.
A \emph{geodesic} is a path $\gamma \colon I \to X$ (where $I \subset \R$ is an interval) such that 
\begin{equation*}
	\dist{\gamma(t)}{\gamma(t')} = \abs{t' - t},\quad \forall t,t' \in I.
\end{equation*}
From now on, we assume that $X$ is \emph{geodesic}, that is any two points are joined by a (non necessarily unique) geodesic. 
Note that we do require $X$ to be geodesically complete.

For every $x \in X$ and $r \in \R_+$, we denote by $B(x,r)$ the open ball of radius $r$ centered at $x$.
Let $Y$ be a closed subset of $X$.
Given $x \in X$, a point $y \in Y$ is a (\emph{nearest point}) \emph{projection} of $x$ on $Y$ if $\dist xy = d(x,Y)$.
The \emph{projection} of a subset $Z\subset X$ onto $Y$ is
\begin{equation*}
	\pi_Y(Z) = \set{y \in Y}{ y \ \text{is the projection of a point}\ z \in Z}.
\end{equation*}
Let $I \subset \R$ be a closed interval and  $\gamma \colon I \to X$ a continuous path intersecting $Y$.
The \emph{entry point} and \emph{exit point} of $\gamma$ in $Y$ are the points $\gamma(t)$ and $\gamma(t')$ where
\begin{equation*}
	t = \inf \set{s \in I}{\gamma(s) \in Y} 
	\quad \text{and} \quad
	t' = \sup \set{s \in I}{\gamma(s) \in Y} .
\end{equation*}
If $I$ is bounded, such points always exist (the subset $Y$ is closed).
Given $d\in \R_+$, we denote by $\mathcal N_d(Y)$ the \emph{$d$-neighborhood} of $Y$, that is the set of points $x \in X$ such that $d(x,Y) \leq d$.
The distance between two subsets $Y, Y'$ of $X$ is 
\begin{equation*}
	d(Y,Y') = \inf_{(y,y') \in Y \times Y'} \dist y{y'}.
\end{equation*}

\paragraph{Contracting set.}

\begin{defi}[Contracting set]
	Let $\alpha \in \R^*_+$.
	A closed subset $Y \subset X$ is \emph{$\alpha$-contracting} if for any geodesic $\gamma$ with $d(\gamma, Y)\geq \alpha$, we have $\diam \pi_Y(\gamma) \leq \alpha$.
	The set $Y$ is \emph{contracting} if $Y$ is $\alpha$-contracting for some $\alpha \in \R^*_+$.	
\end{defi}

The next statements are direct consequences of the definition.
Their proofs are left to the reader.

\begin{lemm}[Quasi-convexity]
\label{res: qc contracting set}
	Let $Y$ be an $\alpha$-contracting subset.
	If $\gamma$ is a geodesic joining two points of $\mathcal N_\alpha(Y)$, then $\gamma$ lies in the $5\alpha/2$-neighborhood of $Y$.
\end{lemm}

\begin{lemm}[Projections]
\label{res: proj contracting set}
	Let $Y$ be an $\alpha$-contracting subset.
	Let $x,y\in X$ and $\gamma$ be a geodesic from $x$ to $y$.
	Let $p$ and $q$ be respective projections of $x$ and $y$ onto $Y$.
	If $d(x,Y) < \alpha$ or $\dist pq > \alpha$, then the following holds:
	\begin{enumerate}
		\item $d(\gamma, Y) < \alpha$;
		\item the entry point (\resp exit point) of $\gamma$ in $\mathcal N_\alpha(Y)$ is $2\alpha$-closed to $p$ (\resp $q$);
		\item $\dist xy \geq \dist xp + \dist pq +  \dist qy - 8\alpha$.
	\end{enumerate}
\end{lemm}

\begin{rema}
\label{rem: Lipschitz proj}
	It follows from the above statement that the nearest point projection onto $Y$ is large-scale $1$-Lipschitz.
	Actually, refining the above argument one can prove that for every subset $Z \subset X$, we have
	\begin{equation*}
		\diam (\pi_Y(Z)) \leq \diam(Z) + 4\alpha.
	\end{equation*}
\end{rema}

\begin{lemm}
\label{res: contraction vs hausdorff dist}
	For every $\alpha, d \in \R_+$, there exists $\beta \in \R_+$ with the following property.
	Let $Y$ and $Z$ be two closed subsets of $X$.
	Assume that the Hausdorff distance between them is at most $d$.
	If $Y$ is $\alpha$-contracting, then $Z$ is $\beta$-contracting.
\end{lemm}

\paragraph{Contracting element.}
Consider now a group $G$ acting properly, by isometries on $X$.

\begin{defi}[Contracting element]
	Let $y \in X$.
	An element $g \in G$ is \emph{contracting}, for its action on $X$, if the orbit map  $\Z \to G$ sending $n$ to $g^ny$ is a quasi-isometric embedding with contracting image.
\end{defi}

Note that the definition does not depend on the point $y$ (see \autoref{res: contraction vs hausdorff dist}).
The next statement is a reformulation of Yang's Lemma~3.3 (and its proof) in \cite{Yang:2020ub}.

\begin{lemm}
\label{res: segment closed to a contracting element}
	Let $g \in G$ be a contracting element.
	For every $d \in \R_+$ and $z \in X$, there is $\alpha \in \R^*_+$ with the following property.
	Let $p,q \in \Z$ with $p \leq q$.
	Let $x,y \in X$ such that $\dist x{g^pz} \leq d$ and $\dist y{g^qz} \leq d$.
	Let $\gamma$ be a geodesic joining $x$ to $y$.
	Then any subpath of $\gamma$ is $\alpha$-contracting.
	Moreover for every integer $n \in \intvald pq$, the point $g^nz$ is $\alpha$-close to $\gamma$.
\end{lemm}

Let $g \in G$ be a contracting element and $A$ be the $\group g$-orbit of a point $y \in X$.
Define $E(g)$ as the set of elements $u \in G$ such that the Hausdorff distance between $A$ and $uA$ is finite.
It follows from the definition that $E(g)$ is a subgroup of $G$ that does not depend on $y$.
It is the maximal virtually cyclic subgroup of $G$ containing $\group g$.
Moreover $E(g)$ is almost-malnormal, that is $uE(g)u^{-1} \cap E(g)$ is finite, for every $u \in G \setminus E(g)$, see Yang \cite[Lemma~2.11]{Yang:2019wa}.

\begin{lemm}
\label{res: contracting element in commensurated subgroup}
	Assume that $G$ is not virtually cyclic and contains a contracting element.
	Let $H \subset G$ be a commensurated subgroup.
	If $H$ is infinite, then $H$ is not virtually cyclic and contains a contracting element.
\end{lemm}

\begin{proof}
	We only sketch the proof. 
	For more details we refer the reader to Arzhantseva-Cashen-Tao \cite[Section~3]{Arzhantseva:2015cl} where similar arguments are given.
	We first claim that for every contracting element $g \in G$, the group $H$ is not contained in $E(g)$.
	Assume on the contrary that $H \subset E(g)$.
	Since $G$ is not virtually cyclic, there is $u \in G \setminus E(g)$.
	By malnormality, $uHu^{-1} \cap H$ is finite and cannot have finite index in $H$, which contradicts the fact that $H$ is commensurated.
	
	We now fix once and for all a contracting element $g \in G$.	
	We denote by $A$ an orbit of $\group g$.
	It is $\alpha$-contracting for some $\alpha \in \R^*_+$.
	Moreover there is $C \in \R_+$ such that for every $u \in G \setminus E(g)$, we have $\diam(\pi_A(uA)) \leq C$, see Yang \cite[Lemma~2.11]{Yang:2019wa}.
	We choose $n \in \N$ such that $\dist o{g^no}$ is very large compare to both $\alpha$ and $C$.
	
	We claim that there is an element $h \in H\setminus E(g)$ such that both $g^nhg^{-n}$ and $g^{-n}hg^n$ belong to $H$.
	Since $H$ is commensurated, both $g^{-n}Hg^n \cap H$ and $g^nHg^{-n} \cap H$ have finite index in $H$.
	Thus
	\begin{equation*}
		H_0 = \left(g^{-n}Hg^n\right) \cap \left(g^nHg^{-n} \right) \cap H
	\end{equation*}
	has finite index in $H$.
	It follows that $H_0$ is not contained in $E(g)$.
	Indeed otherwise $E(g)$ should also contain $H$, contradicting our previous claim.
	Any element $h \in H_0 \setminus E(g)$ satisfies the conclusions of our second claim.
	
	Consider now the element 
	\begin{equation*}
		f = \left(g^nhg^{-n}\right)\left(g^{-n}h g^n\right)^{-1}
		= g^n \left(hg^{-2n} h^{-1}\right) g^n.
	\end{equation*}
	Note that $f$ belongs to $H$.
	We claim that $f$ is contracting.
	For simplicity we let $g_0 = g^n$ and $g_1= hg^{-2n}h^{-1}$ so that $f = g_0g_1g_0$.
	They respectively act ``by translation'' on the $\alpha$-contracting sets $A_0 = A$ and $A_1 = hA$.
	Fix a point $x \in \pi_A(hA)$.
	Since $\pi_A(hA)$ has diameter at most $C$, any geodesic $\gamma \colon \intval 0T \to X$ from $x$ to $fx$ fellow-travels for a long times with $A_0=A$, $g_0A_1=g_0hA$ and $g_0g_1A_0 = fA$ (see \autoref{fig: contracting1}).
	\begin{figure}[htbp]
		\vskip5mm
		\labellist
		\small\hair 2pt
		 \pinlabel {$x$} [ ] at 40 25
		 \pinlabel {$g_0x$} [ ] at 88 70
		 \pinlabel {} [ ] at 159 116
		 \pinlabel {} [ ] at 237 116
		 \pinlabel {$g_0g_1x$} [ ] at 310 80
		 \pinlabel {$g_0g_1g_0x = fx$} [ ] at 320 28
		 \pinlabel {$A_0 = A$} [ ] at 20 82
		 \pinlabel {$g_0A_1 = g_0hA$} [ ] at 205 164
		 \pinlabel {$g_0g_1A_0 = fA$} [ ] at 410 70
		\endlabellist
		\centering
		\includegraphics[width=0.8\linewidth]{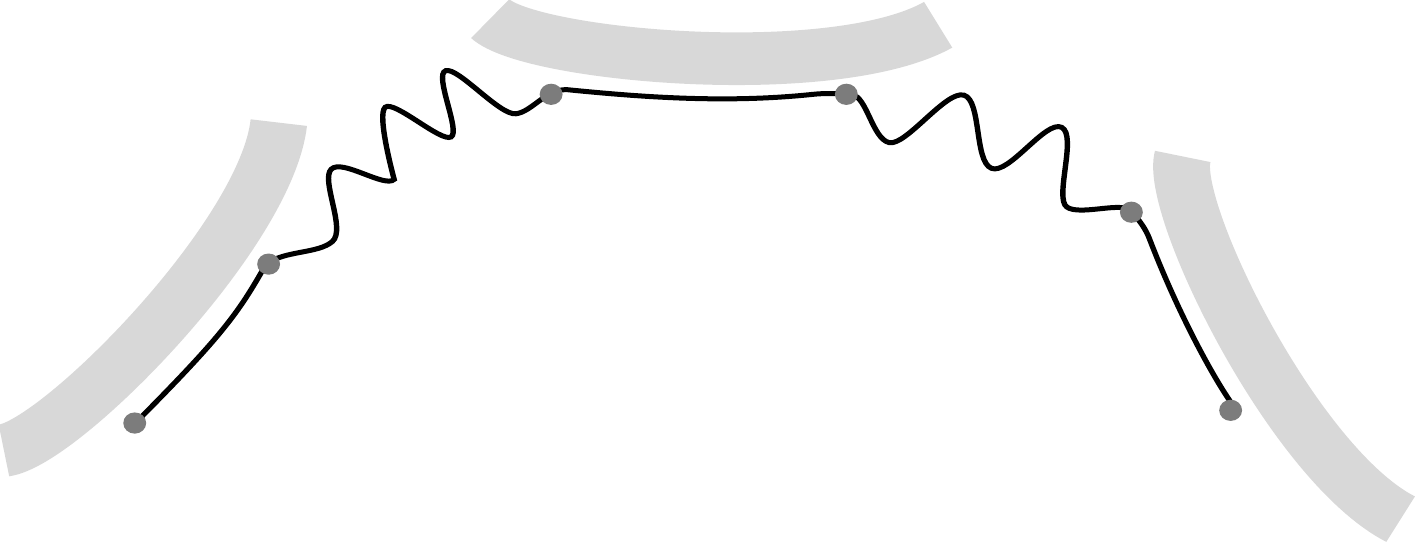}
		\caption{The geodesic from $x$ to $fx$. The gray shapes are ``axis'' of conjugates of $g$.}
		\label{fig: contracting1}
	\end{figure}
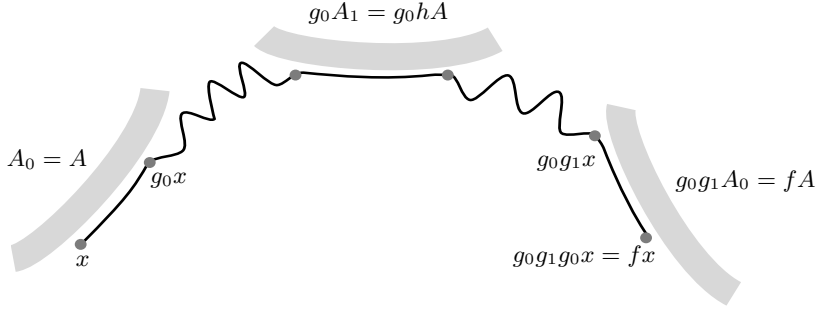
	Note that the construction has been designed so that the end of $\gamma$ and the beginning of $f \gamma$ both fellow-travel with $fA$ in the ``same direction''.
	Consequently, the concatenation of $\gamma$ and $f\gamma$ cannot backtrack much.
	We extend $\gamma$ to a bi-infinite $\group f$-invariant path, still denoted by $\gamma \colon \R \to X$, which is characterized as follows: $\gamma(t + kT) = f^k \gamma(t)$ for every $t \in \R$ and $k \in \Z$.
	One proves that $\gamma$ is a quasi-geodesic which  fellow-travels for a long time with $f^kA$, for every $k \in \Z$ (see \autoref{fig: contracting2}).
	This suffices to show that $f$ is contracting.
	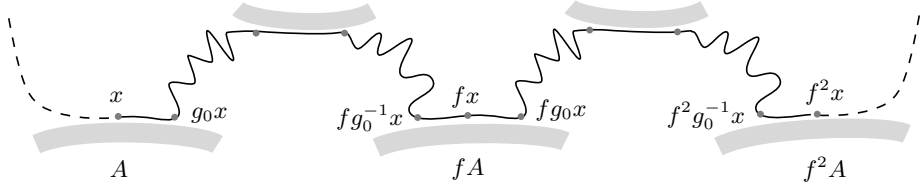
\begin{figure}[htb]
		\labellist
		\small\hair 2pt
		 \pinlabel {$x$} [ ] at 42 23
		 \pinlabel {$g_0x$} [ ] at 78 16
		 \pinlabel {$fg_0^{-1}x$} [ ] at 140 15
		 \pinlabel {$fx$} [ ] at 177 23
		 \pinlabel {$fg_0x$} [ ] at 214 18
		 \pinlabel {$f^2g_0^{-1}x$} [ ] at 270 16
		 \pinlabel {$f^2x$} [ ] at 316 25
		 \pinlabel {$A$} [ ] at 43 -5
		 \pinlabel {$fA$} [ ] at 178 -5
		 \pinlabel {$f^2A$} [ ] at 316 -5
		\endlabellist
		\centering
		\includegraphics[width=\linewidth]{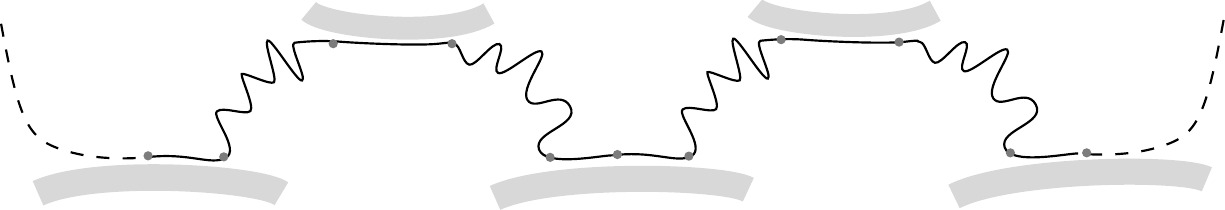}
		\vskip5mm
		\caption{The ``axis'' of $f$.}
		\label{fig: contracting2}
	\end{figure}

	We are left to prove that $H$ is not virtually cyclic.
	If $H$ was virtually cyclic, it would be contained in $E(f)$.
	This contradicts our first claim.
\end{proof}

\section{Compactification of $X$}

\subsection{Horocompactification}
Let $C(X)$ be the set of all real valued, continuous functions on $X$ endowed with the topology of uniform convergence on every compact subset.
We denote by $C^*(X)$ the quotient of $C(X)$ by the subspace consisting of all constant functions, and endowed with the quotient topology.
Given a base point $o \in X$, one can think of $C^*(X)$ as the set of all continuous functions that vanish at $o$.
Alternatively $C^*(X)$ is the set of continuous cocycles $c \colon X \times X \to \R$.
By \emph{cocycle} we mean that 
\begin{equation*}
	c(x,z) = c(x,y) + c(y,z), \quad \forall x,y,z \in X.
\end{equation*}
For example, given $z \in X$, we define the cocycle $b_z \colon X \times X \to \R$ by 
\begin{equation*}
	b_z(x,y) = \dist xz - \dist yz, \quad \forall x,y \in X.
\end{equation*}
Since $X$ is geodesic, the map
\begin{equation*}
	\begin{array}{rccc}
		\iota \colon & X & \to & C^*(X) \\
			& z & \mapsto & b_z
	\end{array}
\end{equation*}
is a homeomorphism from $X$ onto its image.

\begin{defi}[Horoboundary]
	The \emph{horocompactification} $\bar X$ of $X$ it the closure of $\iota(X)$ in $C^*(X)$.
	The \emph{horoboundary} of $X$ is the set $\partial X = \bar X \setminus \iota(X)$.
\end{defi}

From now on, we identify $X$ with its image under the map $\iota \colon X \to C^*(X)$.
By construction, every cocycle $c \in \bar X$ is $1$-Lipschitz, or equivalently $\abs{c(x,x')} \leq \dist x{x'}$, for every $x,x' \in X$.
It is a consequence of the Azerla-Ascoli theorem, that the horocompactification $\bar X$ is indeed a compact set.
We denote by $\mathfrak B$ the Borel $\sigma$-algebra on $\bar  X$.
In the remainder of the article we make an abuse of notations, and write $(\partial X, \mathfrak B)$ to denote the horoboundary endowed with the $\sigma$-algebra $\mathfrak B$ restricted to $\partial X$.

\begin{defi}
	Let $c \in \bar X$ and $\epsilon \in \R_+$.
	An \emph{$\epsilon$-quasi-gradient arc} for $c$ is a path $\gamma \colon I \to X$ parametrized by arc length such that 
	\begin{equation*}
		(t - s) - \epsilon \leq c(\gamma(s),\gamma(t)) \leq t - s, \quad \forall s,t \in I.
	\end{equation*}
	A \emph{gradient arc for $c$} is a $0$-quasi-gradient arc for $c$.	
	If $I = \R_+$, we call $\gamma$ a \emph{(quasi-)gradient ray}.
\end{defi}

\begin{rema}
\label{rem: gradient arc}
	The following observations follow from the definition and/or the triangle inequality.
	\begin{itemize}
		\item Since cocycles in $\bar X$ are $1$-Lipschitz, a gradient arc is always geodesic.		
		\item Conversely, let $x,y \in X$ and $\epsilon \in \R_+$ such that $c(x,y) \geq \dist xy - \epsilon$.
		Any geodesic from $x$ to $y$ is an $\epsilon$-quasi-gradient arc for $c$.
		\item Let $\gamma_1 \colon \intval {a_1}{b_1} \to X$ and $\gamma_2 \colon \intval {a_2}{b_2} \to X$ be two paths such that $\gamma_i$ is an $\epsilon_i$-quasi-gradient arc for $c$.
		If $\gamma_1(b_1) = \gamma_2(a_2)$ then the concatenation of $\gamma_1$ and $\gamma_2$ is an $(\epsilon_1 + \epsilon_2)$-quasi-gradient arc for $c$.
		
	\end{itemize}
\end{rema}

The existence of gradient rays is given by the next statement.

\begin{lemm}
\label{res: existence gradient ray}
	Let $c \in \partial X$.
	For every $x \in X$, there exists a gradient ray $\gamma \colon \R_+ \to X$ for $c$ such that $\gamma(0) = x$.
\end{lemm}

\begin{proof}
	Let $(z_n)$ be a sequence of points of $X$ converging to $c$.
	For every $n \in \N$, we let $b_n = \iota(z_n)$ and denote by $\gamma_n \colon \intval 0{\ell_n} \to X$ a geodesic from $x$ to $z_n$.
	Since $X$ is proper, $(\gamma_n)$ converges, up to passing to a subsequence, to a geodesic ray $\gamma \colon \R_+ \to X$ starting at $x$.
	As $(b_n)$ converges uniformly on every compact subset to $c$, we check that $\gamma$ is a gradient ray for $c$.
\end{proof}

Given $c \in \partial X$ and a gradient ray $\gamma \colon \R_+ \to X$ for $c$, we think of $\gamma$ as a geodesic from $\gamma(0)$ to $c$.
The next definition is designed to handle simultaneously the cocycles corresponding to points in $X$ and the ones in $\partial X$.

\begin{defi}
\label{res: complete gradient arc}
	Let $x \in X$ and $c \in \bar X$.
	A \emph{gradient arc from $x$ to $c$} is 
	\begin{itemize}
		\item any geodesic from $x$ to $z$, if $c = \iota(z)$ for some $z \in X$,
		\item any gradient ray for $c$ starting at $x$, if $c$ belongs to $\partial X$.
	\end{itemize}
\end{defi}

\subsection{Reduced horoboundary}
\label{sec: reduced horoboundary}

Given a cocycle $c \in C^*(X)$ we write $\norm[\infty]c$ for its uniform norm, i.e.
\begin{equation*}
	\norm[\infty]c = \sup_{x,x' \in X} \abs{c(x,x')}.
\end{equation*}
Note that $\norm[\infty]c$ can be infinite.
If $K \subset X$ is compact, then $\norm[K]c$ is the uniform norm of $c$ restricted to $K \times K$.
We endow $\bar X$ with a binary relation: two cocycles $c,c' \in \bar X$ are equivalent, and we write $c \sim c'$, if one of the following holds:
\begin{itemize}
	\item either $c$ and $c'$ lie in the image of $\iota \colon X \to C^*(X)$ and $c = c'$,
	\item or $c,c' \in \partial X$ and $\norm[\infty] {c-c'} < \infty$.
\end{itemize}
Given a subset $B \subset \bar X$, the \emph{saturation} of $B$, denoted by $B^+$, is the union of all equivalence classes intersecting $B$.
We say that $B$ is \emph{saturated} if it is a union of equivalence classes, or equivalently if $B^+ = B$.
Note that the collection of saturated subsets is closed under complement as well as (uncountable) union and intersection.
The \emph{reduced algebra}, denoted by $\mathfrak R$, is the sub-$\sigma$-algebra of $\mathfrak B$ which consists of all saturated Borel subsets.

\begin{defi}
\label{def: reduced horoboundary}
	The \emph{reduced horocompactification} and \emph{reduced horoboundary} of $X$ are respectively the measurable spaces $(\bar X, \mathfrak R)$ and $(\partial X, \mathfrak R)$.
\end{defi}

\begin{lemm}
\label{res: saturation of closed subset}
	If $F \subset \bar X$ is closed, then $F^+$ belongs to $\mathfrak R$.
\end{lemm}

\begin{proof}
	It suffices to prove that $F^+$ is a Borel subset.
	Without loss of generality we can assume that $F$ is contained in $\partial X$.
	Given $D \in \R_+$ and a compact subset $K \subset X$, we write
	\begin{equation*}
		F_{K,D} = \set{c \in \partial X}{ \exists b \in F, \ \norm[K]{c-b} \leq D}.
	\end{equation*}
	Since $F$ is compact, $F_{K, D}$ is closed.
	Using again the fact that $F$ is compact we observe that 
	\begin{equation*}
		F^+ = \bigcup_{D \in \R_+} \bigcap_{K \subset X} F_{K,D}
	\end{equation*}
	where $K$ runs over all compact subsets of $X$.
	Hence the result.
\end{proof}

\subsection{Boundary at infinity of a contracting set}

The goal of this section is to understand how cocycles at infinity interact with contracting subsets of $X$.

\begin{defi}
	Let $Y$ be a closed subset of $X$.
	Let $c \in \bar X$.
	A \emph{projection} of $c$ on $Y$ is a point $q \in Y$ such that for every $y \in Y$, we have $c(q,y) \leq 0$.
\end{defi}

Given $z \in X$, the projection of $b = \iota(z)$ on $Y$ coincides with the definition of the nearest point projection.

\begin{coro}
\label{res: proj cocycle pre lipschitz}
	Let $Y$ be a closed subset of $X$.
	Let $c \in \bar X$.
	Assume that $c$ admits a projection $q$ on $Y$ and denote by $\gamma \colon I \to \R_+$ a gradient arc for $c$ starting at $q$.
	Then for every $t \in I$, the point $q$ is a projection of $\gamma(t)$ on $Y$.
\end{coro}

\begin{proof}
	Let $t\in I$.
	Note that
	\begin{equation*}
		\dist q{\gamma(t)} 
		\leq c(q, \gamma(t))
		\leq c(q,y) + c(y,\gamma(t))
		\leq c(q,y) + \dist y{\gamma(t)}.
	\end{equation*}
	The first inequality holds since $\gamma$ is a gradient line for $c$ while the last one follows from the fact that $c$ is $1$-Lipschitz. 
	However $q$ being a projection of $c$ on $Y$, we have $c(q,y) \leq 0$.
	Consequently $\dist q{\gamma(t)} \leq \dist y{\gamma(t)}$, for every $y \in Y$, which completes the proof.
\end{proof}

If $c$ is a point in $\partial X$, a projection of $c$ on $Y$ may exist or not.
This leads to the following definition.

\begin{defi}
	Let $Y$ be a closed subset of $X$.
	The \emph{boundary at infinity} of $Y$, denoted by $\partial ^+ Y$, is the set of all cocycles $c \in \partial X$ for which there is no projection of $c$ on $Y$.
\end{defi}

We give several equivalent characterizations of the boundary at infinity of a contracting subset.

\begin{prop}
\label{res: equivalent def not boundary contracting}
	Let $\alpha \in \R_+$.
	Let $Y$ be an $\alpha$-contracting subset of $X$.	
	Let $(z_n)$ be a sequence of points in $X$ which converges to a cocycle $c \in \partial X$.
	For every $n \in \N$, denote by $q_n$ a projection of $z_n$ onto $Y$.
	Let $\gamma \colon \R_+ \to X$ a gradient ray for $c$.
	Define $T \in \R_+ \cup \{\infty\}$ by
	\begin{equation*}
		T = \sup \set{t \in \R_+}{d(\gamma(t),Y) \leq \alpha}
	\end{equation*}
	with the convention that $T = 0$, whenever $\gamma$ does not intersect $\mathcal N_\alpha(Y)$.
	The following are equivalent.
	\begin{enumerate}
		\item \label{enu: equivalent def not boundary contracting - bdy}
		$c \notin \partial^+Y$.
		\item \label{enu: equivalent def not boundary contracting - bounded for all}
		For every $x \in X$, the map $Y \to \R$ sending $y$ to $c(x,y)$ is bounded from above.
		\item \label{enu: equivalent def not boundary contracting - neighborhood}
		The ray $\gamma$ does not stay in a neighborhood of $Y$.
		\item \label{enu: equivalent def not boundary contracting - proj}
		The projection $\pi_Y(\gamma)$ is bounded.
		\item \label{enu: equivalent def not boundary contracting - improved proj}
		$T < \infty$.
		\item \label{enu: equivalent def not boundary contracting - bounded approx}
		The sequence $(q_n)$ is bounded.
	\end{enumerate}
	Moreover, in this situation,
	\begin{itemize}
		\item the diameter of the set $Q = \pi_Y\left(\restriction \gamma {[T, \infty)}\right)$ is at most $\alpha$.
		\item any accumulation point $q^*$ of $(q_n)$ is a projection of $c$ on $Y$ which lies in the $\alpha$-neighborhood of $Q$.
	\end{itemize}
\end{prop}

\begin{rema*}
	It follows from \ref{enu: equivalent def not boundary contracting - bounded for all} that $\partial^+Y$ is saturated (as the notation suggested).
\end{rema*}

\begin{proof}
	The equivalences \ref{enu: equivalent def not boundary contracting - neighborhood}$\iff$\ref{enu: equivalent def not boundary contracting - proj} and \ref{enu: equivalent def not boundary contracting - proj}$\iff$\ref{enu: equivalent def not boundary contracting - improved proj} are standard properties of contracting sets, which only use the fact that $\gamma$ is a geodesic.
	Note that $c(x,y) = c(x, x') + c(x',y)$, for every $x,x' \in X$ and $y \in Y$.
	The implication \ref{enu: equivalent def not boundary contracting - bdy} $\Rightarrow$ \ref{enu: equivalent def not boundary contracting - bounded for all} follows from this observation.
	The proof of \ref{enu: equivalent def not boundary contracting - bounded for all} $\Rightarrow$ \ref{enu: equivalent def not boundary contracting - neighborhood} is by contraposition.
	Suppose that there exists $d \in \R_+$ such that $\gamma$ lies $\mathcal N_d(Y)$.
	Let $x = \gamma(0)$.
	Let $t \in \R_+$.
	We denote by $q_t$ a projection of $\gamma(t)$ onto $Y$.
	Using the fact that $c$ is $1$-Lipschitz, we get
	\begin{equation*}
		c(x, q_t) 
		\geq c(\gamma(0), \gamma(t)) - \dist{\gamma(t)}{q_t}
		\geq t  - d.
	\end{equation*}
	This inequality holds for every $t \in \R_+$, hence the map $Y \to \R$, sending $y$ to $c(x,y)$ is not bounded from above.
	
	We now focus on \ref{enu: equivalent def not boundary contracting - improved proj}
 $\Rightarrow$ \ref{enu: equivalent def not boundary contracting - bounded approx}.
 	For every $n \in \N$, we let $b_n = \iota(z_n)$.
 	Assume that $T < \infty$.
	Let $Q$ be the projection onto $Y$ of $\gamma$ restricted to $[T, \infty)$.
	Since $Y$ is contracting, the diameter of $Q$ is at most $\alpha$.
	Let $q \in Q$ be a projection of $\gamma(T)$ on $Y$.
	We claim that there is $N \in \N$, such that for every $n \geq N$, the point $q_n$ stays at a distance at most $\alpha$ of $Q$.
	Assume on the contrary that it is not the case.
	Up to passing to a subsequence, $d(q_n, Q) > \alpha$, for every $n \in \N$.
	Using \autoref{res: proj contracting set}, we observe that for every $t \geq T$, for every $n \in \N$, 
	\begin{equation*}
		b_n(\gamma(t), q) \geq   \dist q{\gamma(t)} - 4\alpha.
	\end{equation*}
	After passing to the limit we get $c(\gamma(t), q) \geq  \dist q{\gamma(t)} - 4\alpha$, for every $t \geq T$.
	In particular, $t \mapsto c(\gamma(t), q)$ diverges to infinity as $t$ tends to infinity, which contradicts the fact that $\gamma$ is a gradient line for $c$.
	This completes the proof of our claim and thus implies \ref{enu: equivalent def not boundary contracting - bounded approx}.
	
	We finish the proof with \ref{enu: equivalent def not boundary contracting - bounded approx} $\Rightarrow$ \ref{enu: equivalent def not boundary contracting - bdy}. 
	Assume now that $(q_n)$ is bounded.
	Let $q^*$ be an accumulation point of $(q_n)$.
	As $Y$ is closed $q^*$ belongs to $Y$.
	Observe that for every $y \in Y$, for every $n \in \N$, we have 
	\begin{equation*}
		b_n(q^*,y) \leq b_n(q_n,y) + \dist {q^*}{q_n}  \leq  \dist {q^*}{q_n}.
	\end{equation*}
	By construction $b_n$ converges to $c$ on every compact subset, hence $c(q^*, y) \leq 0$ for every $y \in Y$.
	Thus $q^*$ is a projection of $c$ onto $Y$.
	In particular, $c \notin \partial^+Y$.
\end{proof}

The next statement is a variation on the previous one, sharpening the estimates.

\begin{lemm}
\label{res: proj cocycle lipschitz}
	Let $\alpha \in \R^*_+$.
	Let $Y$ be an $\alpha$-contracting set.
	Let $c \in \bar X \setminus \partial^+Y$ and $q$ a projection of $c$ on $Y$.
	Fix a sequence $(z_n)$ of points in $X$ converging to $c$.
	Denote by $q^*$ an accumulation point of $(q_n)$ where $q_n$ stands for a projection of $z_n$ on $Y$.
	Then $\dist q{q^*} \leq \alpha$.
\end{lemm}

\begin{proof}
	Consider a gradient line $\gamma\colon\R_+ \to X$ from $q$ to $c$.
	According to \autoref{res: proj cocycle pre lipschitz}, $q$ is a projection of any point $\gamma(t)$ on $Y$.
	Assume that contrary to our claim $\dist q{q^*} > \alpha$.
	Reasoning as in the proof of \autoref{res: equivalent def not boundary contracting}, we observe that  $c(\gamma(t), q) \geq  \dist q{\gamma(t)} - 4\alpha$, for every $t \in \R_+$, which contradicts the fact that $\gamma$ is a gradient line for $c$.
\end{proof}

The two statements extend \autoref{res: proj contracting set} for a gradient ray joining a point $x \in X$ to a cocycle $c \in \partial X$.
We distinguish two cases depending whether $c$ belongs to $\partial^+ Y$ or not.

\begin{coro}
\label{res: proj cocycle contracting set}
	Let $\alpha \in \R^*_+$.
	Let $Y$ be an $\alpha$-contracting set.
	Let $x \in X$ and  $c \in \bar X \setminus \partial^+Y$.
	Let $\gamma$ be a gradient arc from $x$ to $c$.
	Let $p$ and $q$ be respective projections of $x$ and $c$ on $Y$.
	If $d(x,Y)< \alpha$ or $\dist pq > 4\alpha$, then the following holds:
	\begin{itemize}
		\item $d(\gamma, Y) < \alpha$;
		\item  the entry point (\resp exit point) of $\gamma$ in $\mathcal N_\alpha(Y)$ is $2\alpha$-closed (\resp $5\alpha$-closed) to $p$ (\resp $q$);
		\item $c(x,q) \geq \dist xp + \dist pq - 14\alpha$.
	\end{itemize}
\end{coro}

\begin{proof}
	If $c$ is the cocycle associated to a point $z \in X$, then the statement is just a particular case of \autoref{res: proj contracting set}.
	Assume now that $c$ belongs to $\partial X \setminus \partial^+ Y$.
	Fix a sequence $(z_n)$ of points in $X$ converging to $c$.
	Denote by $q^*$ an accumulation point of $(q_n)$ where $q_n$ stands for a projection of $z_n$ on $Y$.
	Denote by $\gamma(T)$ the exit point of $\gamma$ from $\mathcal N_\alpha(Y)$, and $Q$ the projection onto $Y$ of $\gamma$ restricted to $[T, \infty)$.
	According to \autoref{res: equivalent def not boundary contracting}, such an exit point exists.
	Moreover, $Q$ has diameter at most $\alpha$, and $q^*$ is a projection of $c$ on $Y$ lying in the $\alpha$-neighborhood of $Q$.
	
	Suppose that $d(x,Y) < \alpha$ or $d(p, q) > 4\alpha$.
	According to \autoref{res: proj cocycle lipschitz}, $\dist p{q^*} > 3 \alpha$, hence, $d(p, Q) > \alpha$.
	It follows from the definition of contracting sets that  $d(\gamma, Y) < \alpha$.
	Let $p'$ (\resp $q'$) be a projection of the entry (\resp exit) point of $\gamma$ in $\mathcal N_\alpha(Y)$ -- note that if $d(x, Y) < \alpha$, then $x$ is the entry point of $\gamma$ in $\mathcal N_\alpha(Y)$ so we can choose $p' = p$.
	Using again the contraction of $Y$, we see that $\dist p{p'} \leq \alpha$.
	Moreover $q'$ belongs to $Q$, thus $\dist {q'}{q^*} \leq 2\alpha$.
	In addition,
	\begin{equation*}
		\dist x{\gamma(T)} 
		\geq \dist x{p'} + \dist {p'}{q'} + \dist {q'}{\gamma(T)} - 4\alpha.
	\end{equation*}
	The path $\gamma$ is a gradient line, thus $c(x, \gamma(T))= \dist x{\gamma(T)}$, hence
	\begin{equation*}
		c(x, \gamma(T))
		\geq \dist x{p'} + \dist {p'}{q'} + \dist {q'}{\gamma(T)} - 4\alpha.
	\end{equation*}
	Combined with the fact that $c$ is $1$-Lipschitz, we get
	\begin{equation}
	\label{eqn: proj cocycle contracting set}
		c(x,q') 
		\geq c(x, \gamma(T)) - \dist {q'}{\gamma(T)} 
		\geq \dist x{p'} + \dist {p'}{q'} - 4\alpha.
	\end{equation}
	We observed that $\dist p{p'} \leq \alpha$, while $\dist {q'}q \leq \dist {q'}{q^*} + \dist{q^*}q \leq 4\alpha$. 
	The conclusion follows from (\ref{eqn: proj cocycle contracting set}) and the triangle inequality.
\end{proof}

\begin{coro}
\label{res: proj cocycle contracting set bdy}
	Let $\alpha \in \R^*_+$.
	Let $Y$ be an $\alpha$-contracting set and $c \in \partial^+Y$.
	Let $x \in X$ and $p$ be a projection of $x$ onto $Y$.
	Let $\gamma \colon \R_+\to X$ be a gradient ray from $x$ to $c$.
	Then the following holds:
	\begin{itemize}
		\item $d(\gamma, Y) < \alpha$;
		\item  the entry point of $\gamma$ in $\mathcal N_\alpha(Y)$ is $2\alpha$-closed to $p$;
		\item $c(x,p) \geq \dist xp - 4\alpha$.
	\end{itemize}
\end{coro}

\begin{proof}
	According to \autoref{res: equivalent def not boundary contracting}~\ref{enu: equivalent def not boundary contracting - proj} the set $\pi_Y(\gamma)$ is unbounded.
	Thus there exists $s \in \R_+$ and a projection $q$ of $\gamma(s)$ onto $Y$ such that $\dist pq > \alpha$.
	It follows that $d(\gamma, Y) < \alpha$.
	Let $\gamma(t)$ be the entry point of $\gamma$ in $\mathcal N_\alpha(Y)$.
	Using again the contraction of $Y$, we get $\dist p{\gamma(t)} \leq 2\alpha$.
	Since $\gamma$ is a gradient line we have $c(x,\gamma(t)) = \dist x{\gamma(t)}$.
	Combined with the triangle inequality and the fact that cocycles are $1$-Lipschitz we get $c(x,p) \geq \dist xp - 4\alpha$.
\end{proof}

\begin{prop}
\label{res: existence transverse contracting}
	Assume that $G$ is not virtually cyclic.
	Let $g$ be a contracting element and $A$ an orbit of $\group g$.
	There exists $u \in G$ such that $\partial^+A \cap \partial ^+ (uA) = \emptyset$.
\end{prop}

\begin{proof}
	Consider an element $u \in G$ such that $\partial^+A\cap \partial^+(uA)$ is not empty.
	Let $c$ be a cocycle in this intersection and $\gamma \colon \R_+ \to X$ a gradient ray for $c$.
	By \autoref{res: equivalent def not boundary contracting}~\ref{enu: equivalent def not boundary contracting - neighborhood} there exist $d, T \in \R_+$ such that $\gamma$ restricted to $[T, \infty)$ is contained in $\mathcal N_d(A) \cap \mathcal N_d(uA)$.
	In particular, the diameter of this intersection is infinite.
	It follows that $u \in E(g)$, see Yang \cite[Lemma~2.12]{Yang:2019wa}.
	Recall that $E(g)$, unlike $G$, is virtually cyclic.
	Thus there exists $u \in G\setminus E(g)$.
	It follows from the above discussion that $\partial^+A \cap\partial ^+(uA) = \emptyset$.
\end{proof}

\section{Conformal densities}

As previously, $(X,d)$ is a proper, geodesic, metric space, while  $G$ is a group acting properly, by isometries on $X$.

\subsection{Definition and existence}

If $\mu$ is a finite measure on $\bar X$, we denote by $\norm \mu$ its total mass.

\begin{defi}[Density]
\label{def: density}
	Let $\omega \in \R_+$.
	Let $\mathfrak A$ be a $G$-invariant sub-$\sigma$-algebra of the Borel $\sigma$-algebra $\mathfrak B$.
	A \emph{density} on $(\bar X, \mathfrak A)$ is a collection $\nu = (\nu_x)$ of positive finite measures on $(\bar X, \mathfrak A)$ indexed by $X$ such that  $\nu_x \ll \nu_y$, for every $x,y \in X$, and normalized by $\norm{\nu_o} = 1$.
	Such a density is 
	\begin{enumerate}
		\item  \emph{$G$-invariant}, if $g_\ast \nu_x = \nu_{gx}$, for every $g \in G$ and $x \in X$.
		\item \emph{$\omega$-conformal}, if for every $x,y \in X$,
		\begin{equation*}
			\frac{d\nu_x}{d\nu_y} (c) = e^{-\omega c(x,y)}, \quad \nu_y\text{-a.e.}
		\end{equation*}
		\item \emph{$\omega$-quasi-conformal}, if there is $C \in \R_+^*$ such that for every $x,y \in X$,
		\begin{equation*}
			\frac 1Ce^{-\omega c(x,y)} \leq \frac{d\nu_x}{d\nu_y} (c) \leq C e^{-\omega c(x,y)} , \quad \nu_y\text{-a.e.}
		\end{equation*}
	\end{enumerate}
\end{defi}

\begin{voca*}
	Let $\nu = (\nu_x)$ be a density on $(\bar X, \mathfrak A)$.
	We make an abuse of vocabulary and say that a property on $(\bar X, \mathfrak A)$ holds \emph{$\nu$-almost everywhere} if it holds $\nu_x$-almost everywhere for some (hence every) $x \in X$.
\end{voca*}

\begin{rema}
\label{rem: rough comparison}
	Let $\nu = (\nu_x)$ be an $\omega$-conformal density on $(\bar X, \mathfrak A)$.
	Recall that every cocycle in $\bar X$ is $1$-Lipschitz.
	It follows that 
	\begin{equation}
	\label{eqn: harnack}
		\mu_x(A) \leq e^{\omega \dist xy} \mu_y(A), \quad \forall x,y \in X, \ \forall A \in \mathfrak A.
	\end{equation}
 	This observation will be useful many times later.
\end{rema}

In practice, we will consider only two $\sigma$-algebras on $\bar X$: the Borel $\sigma$-algebra $\mathfrak B$ and the reduced $\sigma$-algebra $\mathfrak R$ (see \autoref{sec: reduced horoboundary}).
If $\nu$ is a conformal density on $(\bar X, \mathfrak B)$, then its restriction to the reduced $\sigma$-algebra $\mathfrak R$ is not necessarily quasi-conformal.
We will see later that this pathology can be avoided if the action of $G$ on $X$ is divergent.

\paragraph{Topology.}
We denote by $\mathcal D(\omega)$ the set of all $\omega$-conformal densities on the horocompactification $(\bar X, \mathfrak B)$.
We endow $\mathcal D(\omega)$ with the following topology: a sequence $\nu^n = (\nu^n_x)$ of densities converges to $\nu = (\nu_x)$ if for every $x \in X$, the measure $\nu^n_x$ converges to $\nu_x$ for the weak-* topology.
Let $\mathcal P(\bar X)$ be the set of all Borel probability measures on $\bar X$ (endowed wit the weak-* topology).
An $\omega$-conformal density $\nu \in \mathcal D(\omega)$ is entirely determined by the measure $\nu_o$.
More precisely the map
\begin{equation}
\label{eqn: density only depends on base point}
	\begin{array}{ccc}
		\mathcal D(\omega) &\to & \mathcal P(\bar X) \\
		\nu & \mapsto & \nu_o
	\end{array}
\end{equation}
is a homeomorphism.
We denote by $\mathcal D(G,\omega)$ the convex closed subspace of $\mathcal D(\omega)$ consisting of all $G$-invariant, $\omega$-conformal densities on $(\bar X, \mathfrak B)$.
The densities $\nu = (\nu_x)$ in $\mathcal D(G, \omega)$ for which the action of $G$ on $(\bar X, \mathfrak B, \nu_o)$ is ergodic are exactly the extremal points of $\mathcal D(G, \omega)$.

\paragraph{Patterson's construction.}
We now prove the existence of invariant conformal densities supported on the horoboundary.
Actually we focus on a slightly more general settings that will be useful for our applications.
A map $\chi \colon G \to \R$ is a \emph{quasi-morphism} if there exists $C \in \R_+$ such that for every $g,g' \in G$, we have
\begin{equation}
\label{eqn: quasi-morphism}
	\abs{\chi(g) + \chi(g') - \chi(gg') } \leq C.
\end{equation}
To such a quasi-morphism $\chi$, we associate a \emph{twisted Poincaré series} defined as
\begin{equation*}
	\mathcal P_\chi(s) = \sum_{g \in G} e^{\chi(g)} e^{-s\dist o{go}},
\end{equation*}
and write $\omega_\chi$ for its critical exponent.
It follows from (\ref{eqn: quasi-morphism}) that 
\begin{equation*}
	e^{-2C} \mathcal P_\chi(s) \leq \mathcal P_{-\chi}(s) \leq e^{2C} \mathcal P_\chi(s), \quad \forall s \in \R_+.
\end{equation*}
Hence $\omega_{-\chi} = \omega_\chi$.
Note also that 
\begin{equation*}
	\frac12 \left[\mathcal P_\chi(s) + \mathcal P_{-\chi}(s)\right]
	\geq \sum_{g \in G} \cosh \left( \chi(g)\right) e^{-s\dist o{go}}
	\geq \mathcal P_G(s).
\end{equation*}
Thus $\omega_\chi \geq \omega_G$.

\begin{prop}
\label{res: existence twisted ps}
	Let $H$ be a subgroup of $G$.
	Let $\chi \colon G \to \R$ be a quasi-morphism such that $\chi(hg) = \chi(g)$, for all $h \in H$ and $g \in G$.
	There is an $H$-invariant, $\omega_\chi$-conformal density $\nu = (\nu_x)$ on $(\bar X, \mathfrak B)$ with the following properties: 
	\begin{itemize}
		\item $\nu$ is supported on $\partial X$;
		\item there is $C \in \R_+^*$ such that for every $g \in G$, and $x \in X$, we have
		\begin{equation*}
			\frac 1C \nu_x \leq e^{-\chi(g)} {g^{-1}}_\ast \nu_{gx} \leq C \nu_x.
		\end{equation*}
	\end{itemize}
\end{prop}

\begin{rema*}
	If $H=G$ and $\chi$ is the trivial morphism, then the proposition says that there exists a $G$-invariant, $\omega_G$-conformal density supported on $\partial X$.
\end{rema*}

\begin{proof}
	The proof follows Patterson's strategy \cite{Patterson:1976hp}.
	As Burger and Mozes already observed, this construction can be carried without difficulty in the horocompactification of $X$ \cite{Burger:1996kc}.
	We only review here its main steps.
	Note that the twisted Poincaré series $\mathcal P_\chi(s)$ may converge at the critical exponent $s = \omega_\chi$.
	Using Patterson's idea, one produces a ``slowly growing'' function $\theta \colon \R_+ \to \R_+$ with the following properties -- see Roblin \cite[Lemme~2.1.1]{Roblin:2005fn}.
	\begin{labelledenu}[P]
		\item \label{enu: patterson - slowing growing weight}
		For every $\varepsilon > 0$, there exists $t_0 \geq 0$, 
	such that for every $t \geq t_0$ and $u \geq 0$, we have $ \theta(t + u ) \leq e^{\varepsilon u} \theta(t)$.
		\item  \label{enu: patterson - weighted series}
		The weighted twisted Poincaré series, defined by 
		\begin{equation}
		\label{eqn: weighted patterson series}
			 \mathcal Q(s)
			=
			\sum_{g \in G} \theta(\dist o{g o})e^{\chi(g)}e^{-s\dist{g o}o}
		 \end{equation}
		 is divergent whenever $s \leq \omega_\chi$, and convergent otherwise. In particular, $\mathcal Q(s)$ diverges to infinity as $s$ approaches $\omega_\chi$ (from above).
	\end{labelledenu}
	For every $x \in X$ and $s > \omega_\chi$, we define a measure on $\bar X$ by
	\begin{equation}
	\label{res: patterson approx}
		\nu_x^s = \frac 1{\mathcal Q(s)}\sum_{g \in G} \theta(\dist x{go}) e^{\chi(g)}e^{-s\dist x{go}} {\rm Dirac}(go).
	\end{equation}
	Since $\bar X$ is compact, the space of probability measures on $\bar X$ is compact for the weak-* topology.
	Consequently there exists a sequence of real numbers $(s_n)$ converging to $\omega_\chi$ from above, and such that for every $x \in X$, the measure $\nu_x^{s_n}$ converges to a measure on $\bar X$, that we denote by $\nu_x$. 
	Note that $\nu ^s= (\nu^s_x)$ is $H$-invariant, for every $s > \omega_\chi$.
	Moreover, since $\chi$ is a quasi-morphism, there exists $C \in \R_+^*$ such that for every $s > \omega_\chi$, $g \in G$ and $x \in X$, we have
	\begin{equation*}
		\frac 1C \nu^s_x \leq e^{-\chi(g)} {g^{-1}}_\ast \nu^s_{gx} \leq C \nu^s_x.
	\end{equation*}
	Hence the same properties hold for $\nu$.
	The horocompactification is precisely designed so that  the map
	\begin{equation*}
		\begin{array}{ccc}
			X \times X \times X & \to & \R \\
			(x,y,z) & \mapsto & \dist xz - \dist yz
		\end{array}
	\end{equation*}
	extends continuously to a map $X \times X \times \bar X \to \R$.
	Taking advantage of this fact, one checks that $\nu$ is $\omega_\chi$-conformal.
	Since $\mathcal Q(s)$ diverges when $s$ approaches $\omega_\chi$, the density $\nu$ is supported on $\partial X$.
\end{proof}

\subsection{Group action on the space of density}

\paragraph{Group action.}
The action of $G$ on $X$ induces a right action of $G$ on the set of densities.
Let $\mathfrak A$ be a $G$-invariant sub-$\sigma$-algebra of $\mathfrak B$. 
Given a density $\nu = (\nu_x)$ on $(\bar X, \mathfrak A)$ and $g \in G$, we define a new density $\nu^g$ as follows
\begin{equation*}
	\nu^g_x = \frac 1{\norm{\nu_{go}}} {g^{-1}}_\ast \nu_{gx}, \quad \forall x \in X.
\end{equation*}
We make the following observations.
\begin{enumerate}
	\item If $\nu$ is $\omega$-conformal, then the same holds for $\nu^g$.
	\item If $\nu$ is $H$-invariant, for some subgroup $H \subset G$, then $\nu^g$ is $H^g$-invariant, where $H^g = g^{-1}Hg$.
\end{enumerate}
In particular, if $N$ is a normal subgroup of $G$, then the map 
\begin{equation*}
	\begin{array}{ccc}
		\mathcal D(N,\omega) \times G & \to & \mathcal D(N,\omega) \\
		(\nu, g) & \mapsto & \nu^g
	\end{array}
\end{equation*}
defines a right action of $G$ on $\mathcal D(N,\omega)$ which is trivial when restricted to $N$.

\paragraph{Fixed point properties.}
For our study, we need to distinguish several fixed point properties for the action of $G$ on the space of densities.
Recall that a density $\nu = (\nu_x)$ is
\begin{enumerate}
	\item \label{enu: fix point - inv}
	\emph{$G$-invariant}, if $g_\ast \nu_x = \nu_{gx}$ for every $g \in G$ and $x \in X$.
\end{enumerate}
We say that $\nu$ is
\begin{enumerate}
	\setcounter{enumi}{1}
	\item \label{enu: fix point - almost inv}
	\emph{$G$-almost invariant}, if there exists $C \in \R_+^*$ such that for every $g \in G$ and $x \in X$,
	\begin{equation*}
		\frac 1C \nu_{gx} \leq g_\ast \nu_x \leq C \nu_{gx},
	\end{equation*}
	\item  \label{enu: fix point - fix}
	\emph{fixed by $G$}, if $\nu^g = \nu$, for every $g \in G$,
	\item \label{enu: fix point - afix}
	\emph{almost-fixed by $G$}, if there exists $C \in \R_+^*$ such that for every $g \in G$ and $x \in X$,
	\begin{equation*}
		\frac 1C \nu_x\leq \nu_x^g \leq C \nu_x,
	\end{equation*}
	\item \label{enu: fix point - qinv}
	\emph{$G$-quasi-invariant}, if $g_\ast \nu_o \ll \nu_o$, for every $g \in G$.
\end{enumerate}
If we want to emphasize the constant $C$ in \ref{enu: fix point - afix}, we will say that $\nu$ is \emph{$C$-almost-fixed by $G$}.
These properties are related as follows.
\begin{center}
	\begin{tikzpicture}
		\matrix (m) [matrix of math nodes, row sep=0.5em, column sep=2.5em, text height=1.5ex, text depth=0.25ex] 
		{ 
				&  \rm \ref{enu: fix point - almost inv}& &  	\\
				\rm \ref{enu: fix point - inv} &  & \rm \ref{enu: fix point - afix}& \rm \ref{enu: fix point - qinv}. 	\\
				& \rm \ref{enu: fix point - fix} & & \\
		}; 
		\draw[-{Straight Barb[length=4pt,width=5pt]}, double] (m-2-1) -- (m-1-2) ;
		\draw[-{Straight Barb[length=4pt,width=5pt]}, double] (m-2-1) -- (m-3-2) ;
		\draw[-{Straight Barb[length=4pt,width=5pt]}, double] (m-1-2) -- (m-2-3) ;
		\draw[-{Straight Barb[length=4pt,width=5pt]}, double] (m-3-2) -- (m-2-3) ;
		\draw[-{Straight Barb[length=4pt,width=5pt]}, double] (m-2-3) -- (m-2-4) ;
	\end{tikzpicture}
\end{center}
The reverse implications do not hold in general.
The next statement highlights the role of quasi-morphisms in our study.

\begin{lemm}
\label{res: almost-fixed gives quasi-morphism}
	Let  $\nu = (\nu_x)$ be a density.
	Assume that $\nu$ is fixed (\resp almost-fixed) by $G$.
	Then the map $\chi \colon G \to \R$ sending $g$ to $\ln \norm{\nu_{go}}$ is a morphism (\resp quasi-morphism).
\end{lemm}

\begin{proof}
	Assume that $\nu$ is almost-fixed by $G$ (if $\nu$ is fixed by $G$ the proof works in the exact same way).
	There exists $C \in \R_+^*$ such that for every $g \in G$ and $x \in X$, we have 
	\begin{equation*}
		\frac 1C \nu_x\leq \nu_x^g \leq C \nu_x.
	\end{equation*}
	Let $g,g' \in G$.
	Comparing the total masses of the above measures for $x = g'o$, we get
	\begin{equation*}
		\frac 1C\norm{\nu_{g'o}}
		\leq \frac{\norm{\nu_{gg'o}}}{\norm{\nu_{go}}}
		\leq C \norm{\nu_{g'o}}. \qedhere
	\end{equation*}
\end{proof}

\begin{lemm}
\label{res: ergodic quasi-inv induces almost-fixed}
	Let $H$ be a subgroup of $G$.
	Let $\mu = (\mu_x)$ be an $H$-invariant, $\omega$-quasi-conformal density on the reduced horocompactification $(\bar X, \mathfrak R)$.
	Assume that for every $g \in G$, the action of $H^g \cap H$ on $(\bar X, \mathfrak R, \mu_o)$ is ergodic.
	If $\mu$ is $G$-quasi-invariant, then $\mu$ is almost-fixed by $G$.
\end{lemm}

\begin{proof}
	Since $\mu$ is $G$-quasi-invariant, we can define the following map
	\begin{equation*}
		\begin{array}{rccc}
			F \colon & G \times \bar X & \to & \R_+ \\
			& (g,c) & \mapsto & \displaystyle \frac{d \left({g^{-1}}_\ast \mu_{go}\right)}{d\mu_o}(c).
		\end{array}
	\end{equation*}
	
	\begin{clai}
		There exists $C \in \R_+^*$ such that for every $g_1, g_2 \in G$, we have
		\begin{equation*}
			\frac 1C\frac{F(g_1g_2,c)}{F(g_2,c) } \leq F(g_1,g_2c) \leq C \frac{F(g_1g_2,c)}{F(g_2,c) }, \quad \mu\text{-a.e.}
		\end{equation*}
	\end{clai}
	
	\noindent
	Let $g_1, g_2 \in G$.
	The computation gives
	\begin{equation}
	\label{eqn: ergodic quasi-inv induces almost-fixed}			
			F(g_1g_2,c) 
			=  \frac{d \left({g_1^{-1}}_\ast \mu_{g_1g_2o}\right)}{d\mu_{g_2o}}(g_2c) \ \frac{d \left({g_2^{-1}}_\ast \mu_{g_2o}\right)}{d\mu_o}(c), \quad \mu\text{-a.e.}	\end{equation}
	Since $\mu$ is $\omega$-quasi-conformal, there exists $C \in \R_+$ such that for every $x ,y\in X$, for every $g \in G$, we have
	\begin{equation*}
		\frac 1C \frac {d\mu_x}{d\mu_y}
		\leq \frac{d\mu_{gx}}{d\mu_{gy}} \circ g
		\leq C  \frac {d\mu_x}{d\mu_y}.
	\end{equation*}
	Hence
	\begin{equation*}
		\frac 1C \frac{d \left({g^{-1}}_\ast \mu_{gy}\right)}{d\mu_y}
		\leq \frac{d \left({g^{-1}}_\ast \mu_{gx}\right)}{d\mu_x}
		\leq C \frac{d \left({g^{-1}}_\ast \mu_{gy}\right)}{d\mu_y}.
	\end{equation*}
	It follows that the first factor in the right-hand side of (\ref{eqn: ergodic quasi-inv induces almost-fixed}) is $F(g_1,g_2c)$  -- up to a multiplicative error that does not depend on $g_1$ or $g_2$ -- while the second factor is exactly $F(g_2,c)$.
	This completes the proof of our claim.

	Let $g \in G$ and set $H_0 = H^g \cap H$.
	Let $h \in H_0$.
	According to our claim we have
	\begin{equation*}
		\frac 1C\frac{F(gh,c)}{F(h,c) } \leq F(g,hc) \leq C \frac{F(gh,c)}{F(h,c) }, \quad \mu\text{-a.e.}
	\end{equation*}
	Recall that $\mu$ is $H$-invariant, thus
	\begin{equation*}
		F(h,c) = 1, 
		\quad \text{and} \quad
		F(gh,c) = F(ghg^{-1}g,c) = F(g,c), \quad \mu\text{-a.e.}
	\end{equation*}
	Our previous inequalities becomes
	\begin{equation*}
		\frac 1C F(g,c) \leq F(g,hc) \leq C F(g,c), \quad \mu\text{-a.e.}
	\end{equation*}
	We now define an auxiliary function $F_g \colon \bar X \to \R_+$ by 
	\begin{equation*}
		F_g(c) = \inf_{h \in H_0} F(g,hc).
	\end{equation*}
	By construction $F_g$ is $H_0$-invariant.
	Since the action of $H_0$ on $(\bar X, \mathfrak R, \mu_o)$ is ergodic, $F_g$ is constant.
	From now on, we denote by $F_g$ its essential value.
	It follows from our previous observation that 
	\begin{equation*}
		F_g \leq F(g,c) \leq C F_g, 
		\quad \mu\text{-a.e.}	
	\end{equation*}
	Coming back to the definition of $F$ this means that	\begin{equation*}
		F_g \leq  \frac{d \left({g^{-1}}_\ast \mu_{go}\right)}{d\mu_o} \leq C F_g, \quad \mu\text{-a.e.}
	\end{equation*}
	Integrating these inequalities, we see that $F_g \leq \norm{\mu_{go}} \leq C F_g$ .
	Hence 
	\begin{equation*}
		\frac 1C \norm{\mu_{go}} \leq  \frac{d \left({g^{-1}}_\ast \mu_{go}\right)}{d\mu_o} \leq C \norm{\mu_{go}}.
	\end{equation*}
	Recall that $C$ does not depend on $g$.
	We have proved that there exists $C \in \R_+^*$ such that for every $g \in G$, 
	\begin{equation*}
		\frac 1C \mu_o \leq \mu^g_o \leq C \mu_o.
	\end{equation*}
	Using the quasi-conformality of $\mu$, we conclude that $\mu$ is almost-fixed by $G$.
\end{proof}

\begin{rema*}
	The same argument shows that if $\mu$ is $\omega$-conformal (instead of $\omega$-quasi-conformal) then $\mu$ is fixed by $G$.
\end{rema*}

\subsection{The Shadow Principle}

Given $x,y \in X$ and $c \in \bar X$, we define the following Gromov product 
\begin{equation*}
\label{eqn: gromov product}
	\gro xcy = \frac 12 \left[\dist  xy + c(y,x)\right].
\end{equation*}

\begin{rema*}
	If $c = \iota(z)$ for some $z \in X$, then the above formula coincides with the usual definition of the Gromov product.
\end{rema*}

Since cocycles in $\bar X$ are $1$-Lipschitz, we always have $0 \leq \gro xcy \leq \dist xy$.
We also observe that 
\begin{equation}
\label{eqn: gromov product - lip}
	\abs{ \gro xcy - \gro{x'}c{y'}} \leq \dist x{x'} + \dist y{y'}, \quad \forall x,x',y,y' \in X.
\end{equation}

\begin{defi}
	Let $x,y \in X$.
	Let $r \in \R_+$.
	The \emph{$r$-shadow} of $y$ seen from $x$, is the set 
	\begin{equation*}
		\mathcal O_x(y,r) = \set{ c \in \bar X}{ \gro xcy \leq r}.
	\end{equation*}
\end{defi}

By construction, $\mathcal O_x(y,r)$ is a closed subset of $\bar X$.
It follows from (\ref{eqn: gromov product - lip}) that for every $x,x',y,y' \in X$ and $r \in \R_+$,
\begin{equation}
\label{eqn: inclusion shadows triangle inequality}
	\mathcal O_x(y,r) \subset \mathcal O_{x'}(y',r'), \quad \text{where} \quad r' = r + \dist x{x'} + \dist y{y'}.
\end{equation}

\begin{rema*}
	A more intuitive definition of shadows could have been the following: a cocycle $c \in \bar X$ belongs to $\mathcal O_x(y,r)$ if some gradient arc from $x$ to $c$ passes at a distance at most $r$ from $y$.
	Nevertheless, unlike our approach, this definition is very sensitive to the change of point $x$.
\end{rema*}

Following Roblin with small variations we define the Shadow Principle \cite{Roblin:2005fn}.

\begin{defi}
\label{def: shadow principle}
	Let $\omega \in \R_+$ and $(\epsilon, r_0) \in \R_+^* \times \R_+$.
	Let $\nu = (\nu_x)$ be an $\omega$-conformal density on $(\bar X, \mathfrak B)$.
	We say that $(G, \nu)$ satisfies the \emph{Shadow Principle with parameters $(\epsilon, r_0)$} if for every $g \in G$ and $r \geq r_0$, we have
	\begin{equation}
	\label{eqn: shadow principle}
		\nu_o\left( \mathcal O_o(go,r)\right) 
		\geq \epsilon\norm{\nu_{go}}e^{-\omega \dist o{go}}.
	\end{equation}
	We say that $(G, \nu)$ satisfies the \emph{Shadow Principle}, if there are $(\epsilon, r_0) \in \R_+^* \times \R_+$ such that $(G, \nu)$ satisfies the Shadow Principle with parameters $(\epsilon, r_0)$.
\end{defi}

\begin{rema*}
	If the inequality (\ref{eqn: shadow principle}) holds for $r = r_0$, then it automatically holds for every $r \geq r_0$.
	We will see later that a similar upper bound is always satisfied without any additional assumption.
\end{rema*}

Our next task is to adapt Sullivan's celebrated Shadow Lemma (\autoref{res: shadow lemma}).
It states that $(G,\nu)$ satisfies the Shadow Principle whenever $\nu$ is an $\omega$-conformal density which is $N$-invariant for some infinite normal subgroup $N \vartriangleleft G$.

\begin{prop}
\label{res: pre-shadow lemma}
	Assume that $G$ is not virtually cyclic and contains a contracting element.
	Let $\mathcal D_0$ be a closed subset of $G$-quasi-invariant densities on $\bar X$.
	There exists $(\epsilon, r_0)  \in \R_+^*\times \R_+$ with the following property.
	For every $r \geq r_0$, for every density $\nu = (\nu_x)$ in $\mathcal D_0$, for every $z \in X$, we have
	\begin{equation*}
		\nu_o\left(\mathcal O_z(o,r)\right) \geq \epsilon.
	\end{equation*}
\end{prop}

\begin{proof}
	Assume that our claim fails.
	We can find a sequence $(r_n)$ diverging to infinity, a sequence $(z_n)$ of points in $X$, and a sequence $\nu^n = (\nu^n_x)$ of densities in $\mathcal D_0$ such that 
	\begin{equation*}
		\nu^n_o\left(\mathcal O_{z_n}(o,r_n)\right)
	\end{equation*}
	converges to zero.
	Since $\mathcal D_0$ is closed, up to passing to a subsequence, we may assume that $\nu^n$ converges to a density $\nu$ in $\mathcal D_0$.
	Note that for every $x,y \in X$ and $r \in \R_+$, we have $\mathcal O_x(y,r) = \bar X$, whenever $\dist xy \leq r$.
	Hence $\dist o{z_n}$ necessarily diverges to infinity.
	Up to passing again to a subsequence, we can assume that $z_n$ converges to $b \in \partial X$.
	
	By assumption $G$ is not virtually cyclic and contains a contracting element 
	According to \autoref{res: existence transverse contracting}, there exists a contracting element $h \in G$, such that $b \notin \partial ^+A$, where $A = \group h o$ is $\alpha$-contracting for some $\alpha \in \R_+^*$.
	For every $n \in \N$, we write $p_n$ for a projection of $z_n$ onto $A$.
	Up to passing to a subsequence $(p_n)$ converges to a point $p \in A$ which is a projection of $b$ onto $A$ (\autoref{res: equivalent def not boundary contracting}).
	We introduce the following closed subset of $\bar X$
	\begin{equation*}
		F = \set{c \in \bar X}{ c(p,y) \leq 4\alpha, \ \forall y \in Y}.
	\end{equation*}
	We are going to prove that $\nu_o(F) = 1$.
		
	Let $c \in \bar X \setminus F$.
	Suppose first that $c \notin \partial^+A$.
	Let $q$ be a projection of $c$ onto $A$.
	Since $c$ is $1$-Lipschitz, we have
	\begin{equation*}
		c(p,y) \leq \dist pq + c(q,y) \leq \dist pq, \quad \forall y \in A.
	\end{equation*}
	As $c$ does not belong to $F$, necessarily $\dist pq > 4\alpha$.
	In particular, there exists $N_0 \in \N$, such that for every $n \geq N_0$, we have $\dist {p_n}q > 4\alpha$.
	According to \autoref{res: proj cocycle contracting set}, $\gro {z_n}c{p_n} \leq 2\alpha$.
	Note that the conclusion still holds if $c \in \partial^+A$.
	This is indeed a consequence of \autoref{res: proj cocycle contracting set bdy}.
	Hence in all cases, we have
	\begin{equation*}
		\gro{z_n}co \leq \gro{z_n}c{p_n} + \dist o{p_n} \leq 2\alpha + \dist o{p_n}.
	\end{equation*}
	Recall that $(p_n)$ is bounded.
	Consequently, there is $N_1 \geq N_0$ such that for every $n \geq N_1$, the set $\bar X \setminus F$ is contained in $\mathcal O_{z_n}(o,r_n)$.
	In particular, $\nu^n_o (\bar X \setminus F)$ converges to zero.
	Since $\bar X \setminus F$ is an open subset of $\bar X$, we deduce that $\nu_o(\bar X \setminus F) = 0$, i.e. $\nu_o(F) = 1$.
	
	The group $\group h$ has unbounded orbits.
	Thus there is $g \in \group h$ such that $\dist{gp}p > 36\alpha$.
	We claim that $F \cap gF$ is empty.
	Let $c \in F$.
	Remember that $c$ does not belong to $\partial^+A$ by \autoref{res: equivalent def not boundary contracting}~\ref{enu: equivalent def not boundary contracting - bounded for all}.
	Choose a projection $q$ of $c$ onto $A$.
	It follows from \autoref{res: proj cocycle contracting set} that 
	\begin{equation*}
		4\alpha \geq c(p,q) \geq \dist pq - 14\alpha.
	\end{equation*}
	Hence $\dist pq \leq 18\alpha$.
	Suppose now that contrary to our claim $c$ also belongs to $gF$. 
	In particular, $g^{-1}q$ is a projection of $g^{-1}c$ onto $A$.
	Following the same argument as above, we get $\dist p{g^{-1}q} \leq 18\alpha$.
	Consequently $\dist p{gp} \leq 36\alpha$, a contradiction.
	Observe that
	\begin{equation*}
		\nu_o(gF) = {g^{-1}}_*\nu_o(F) = \norm{\nu_{go}} \nu^g_{g^{-1}o}(F).
	\end{equation*}
	As an element of $\mathcal D_0$, the density $\nu$ is $G$-quasi-invariant, hence $\nu_o \ll \nu^g_{g^{-1}o}$.
	We prove that $\nu_o(F) > 0$, thus $\nu_o(gF) > 0$.
	Consequently,
	\begin{equation*}
		\nu_o( F \cup gF)  = \nu_o( F) + \nu_o(g F) > 1.
	\end{equation*}
	This contradicts the fact that $\nu_o$ is a probability measure.
\end{proof}

\begin{coro}
\label{res: shadow lemma}
	Assume that $G$ is not virtually cyclic and has a contracting element.
	Let $N$ be an infinite normal subgroup of $G$.
	Let $\omega \in \R_+$.
	There exists $(\epsilon, r_0)  \in \R_+^* \times \R_+$ such that for every $r \geq r_0$, for every $N$-invariant, $\omega$-conformal density $\nu = (\nu_x)$, for every $g \in G$, we have
	\begin{equation}
	\label{eqn: shadow lemma}
		\epsilon \norm{\nu_{go}}e^{-\omega \dist o{go}}
		\leq \nu_o\left( \mathcal O_o(go,r)\right) 
		\leq e^{2\omega r}  \norm{\nu_{go}}e^{-\omega \dist o{go}}.
	\end{equation}
	In particular, $(G,\nu)$ satisfies the Shadow Principle with parameters $(\epsilon, r_0)$.
\end{coro}

\begin{proof}
	Let $\nu = (\nu_x)$ be an $\omega$-conformal density.
	A classical computation shows that for every $g \in G$ and $r \in \R_+$, we have
	\begin{equation*}
		\nu_o\left(\mathcal O_o(go,r)\right)
		= \norm{\nu_{go}}\int \mathbb 1_{\mathcal O_{g^{-1}o}(o, r)}(c) e^{-\omega c(g^{-1}o,o)} d\nu^g_o(c).
	\end{equation*}
	Shadows have been designed so that for every $c \in \mathcal O_{g^{-1}o}(o, r)$, we have 
	\begin{equation*}
		\dist o{go} -2 r \leq c(g^{-1}o,o) \leq \dist o{go}.
	\end{equation*}
	Consequently, 
	\begin{equation*}
		\nu^g_o\left(\mathcal O_{g^{-1}o}(o,r)\right)
		\leq \frac { e^{\omega\dist {go}o}}{\norm{\nu_{go}}} \nu_o\left(\mathcal O_o(go,r)\right)
		\leq e^{2\omega r}\nu^g_o\left(\mathcal O_{g^{-1}o}(o,r)\right).
	\end{equation*}
	The upper bound in (\ref{eqn: shadow lemma}) follows from the fact that $\nu^g_o$ is a probability measure.
	Let us focus on the lower bound.
	To that end, we assume now that $\nu$ is $N$-invariant.
	As $N$ is a normal subgroup of $G$, the density $\nu^g$ is $N$-invariant as well.
	Since $G$ contains a contracting element, and is not virtually cyclic, the same holds for $N$ (\autoref{res: contracting element in commensurated subgroup}).
	The result now follows from \autoref{res: pre-shadow lemma} applied with the group $N$ and the closed set $\mathcal D_0 = \mathcal D(N,\omega)$.
\end{proof}

\begin{rema}
\label{rem: shadow lemma}
	Note that the upper bound in (\ref{eqn: shadow lemma}) was proved without assuming any invariance for $\nu$.
	Moreover it works for any $r \in \R_+$.
\end{rema}

Here is another variation of the Shadow Lemma.

\begin{coro}
\label{res: shadow lemma almost-fixed}
	Assume that $G$ is not virtually cyclic and contains a contracting element.
	Let $\omega \in \R_+$ and $\nu = (\nu_x)$ be an $\omega$-conformal density.
	If $\nu$ is almost fixed by $G$, then $(G,\nu)$ satisfies the Shadow Principle.
\end{coro}

\begin{proof}
	Let $g \in G$ and $r \in \R_+$.
	Reasoning as in the proof of \autoref{res: shadow lemma}, we see that
	\begin{equation*}
		\nu_o\left(\mathcal O_o(go,r)\right) 
		\geq \norm{\nu_{go}} e^{-\omega \dist o{go}} \nu^g_o\left(\mathcal O_{g^{-1}o}(o,r)\right).
	\end{equation*}
	Since $\nu$ is almost-fixed by $G$, there is $\epsilon \in \R_+^*$, which does not depend on $g$ or $r$, such that $\nu^g_o \geq \epsilon \nu_o$.
	Consequently,
	\begin{equation*}
		\nu_o\left(\mathcal O_o(go,r)\right) 
		\geq \epsilon\norm{\nu_{go}} e^{-\omega \dist o{go}} \nu_o\left(\mathcal O_{g^{-1}o}(o,r)\right).
	\end{equation*}
	The result now follows from \autoref{res: pre-shadow lemma} applied with $\mathcal D_0 = \{ \nu\}$.
\end{proof}

\subsection{Contracting tails}

In order to take full advantage of the Shadow Lemma, we consider a particular kind of shadows, namely shadows of the form $\mathcal O_x(y,r)$ where $x$ is joined to $y$ by a ``geodesic'' whose tail is contracting.
The next definition quantifies the ``contraction strength'' of this tail.

\begin{defi}
\label{def: contracting tail}
	Let $\alpha \in \R_+^*$ and  $L \in \R_+$.
	Let $x,y \in X$.
	The pair $(x,y)$ has an \emph{$(\alpha, L)$-contracting tail}, if there exists an $\alpha$-contracting geodesic $\tau$ ending at $y$ and a projection $p$ of $x$ on $\tau$ satisfying $\dist py \geq L$.
	The path $\tau$ is called a \emph{(contracting) tail} of $(x,y)$ --- the contracting strength and the length of the tail should be clear from the context.
\end{defi}

\begin{nota}
\label{not: element with contracting tail}
	Given $\alpha \in \R^*_+$ and $L \in \R_+$, we denote by $\mathcal T(\alpha, L)$ the set of all elements $g \in G$ such that the pair $(o,go)$ has an $(\alpha, L)$-contracting tail.
\end{nota}

The next statement is essentially a reformulation of \autoref{res: proj cocycle contracting set}.
It states that if $(x,y)$ has a sufficiently long contracting tail, then the shadow of $y$ seen from $x$ behaves according to our intuition coming from hyperbolic geometry.

\begin{lemm}
\label{res: wysiwyg shadow}
	Let $\alpha \in \R_+^*$ and $r, L \in \R_+$ with $L > r + 13\alpha$.
	Let $x,y \in X$.
	Assume that $(x,y)$ has an $(\alpha, L)$-contracting tail, say $\tau$.
	Let $p$ be a projection of $x$ on $\tau$.
	Let $c\in \mathcal O_x(y,r)$.
	Let $\gamma$ be a gradient arc from $x$ to $c$.
	Let $q$ be a projection of $c$ onto $\tau$.
	Then the following holds.
	\begin{enumerate}
		\item \label{enu: wysiwyg shadow - dist to y}
		$\dist yq \leq r+7\alpha$.
		\item $d(\gamma, \tau) < \alpha$.
		\item The entry point of $\gamma$ in $\mathcal N_\alpha(\tau)$ is $2\alpha$-close to $p$.
		\item The exit point of $\gamma$ from $\mathcal N_\alpha(\tau)$ is $5\alpha$-close to $q$.
	\end{enumerate}
\end{lemm}

\begin{proof}
	Since $\tau$ is an contracting tail, there is a projection $p'$ of $x$ on $\tau$ such that $\dist {p'}y \geq L$.
	The computation shows that 
	\begin{equation*}
		\dist yq + \gro xc{p'} + \gro {p'}cq
		= \gro xcy + \gro xy{p'} + \gro ycq + \gro {p'}yq,
	\end{equation*}
	(it suffices to expand the definition of the Gromov products).
	On the one hand, we observed that $\gro xc{p'} \geq 0$ and $\gro {p'}cq \geq 0$.
	On the other hand, since $\tau$ is contracting, we have $\gro xy{p'} \leq 2\alpha$ (\autoref{res: proj contracting set}) and $\gro ycq \leq 5\alpha$ (\autoref{res: proj cocycle contracting set}).
	It follows that 	
	\begin{equation*}
		\dist yq
		\leq r + 7 \alpha + \gro {p'}yq.
	\end{equation*}
	We now discuss the relative positions of $p'$, $q$ and $y$ on $\tau$.
	If $p'$ lies between $q$ and $y$, then $\gro {p'}yq = \dist {p'}q$ while $\dist yq = \dist y{p'} + \dist {p'}q$.
	It forces $L\leq \dist y{p'} \leq r + 7 \alpha$, which contradicts our assumption.
	Thus $q$ lies between $p'$ and $y$ so that $\gro {p'}yq = 0$.
	Consequently $\dist yq \leq r + 7 \alpha$, which completes the proof of \ref{enu: wysiwyg shadow - dist to y}.

	By assumption, $p$ is another projection of $x$ on $\tau$, which is $\alpha$-contracting, thus $\dist p{p'} \leq 2 \alpha$.
	Consequently,
	\begin{equation*}
		\dist pq \geq \dist {p'}y -  \dist p{p'} - \dist yq \geq L - (r+9\alpha)  > 4\alpha.
	\end{equation*}
	According to \autoref{res: proj cocycle contracting set}, $d(\gamma,\tau) < \alpha$.
	Moreover the entry (\resp exit) point of $\gamma$ in $\mathcal N_\alpha(\tau)$ is  $2\alpha$-close to $p$ (\resp $5\alpha$-closed to $q$).
	Hence the result.
\end{proof}

\begin{lemm}
\label{res: intersection contracting shadows}
	Let $\alpha \in \R_+^*$ and $r, L \in \R_+$ with $L > r + 13\alpha$.
	Let $x, y_1, y_2 \in X$ such that $(x,y_i)$ has an $(\alpha, L)$-contracting tail, say $\tau_i$.
	Let $p_i$ be a projection of $x$ on $\tau_i$.
	If $\mathcal O_x(y_1, r) \cap \mathcal O_x(y_2,r)$ is non empty, then
	\begin{enumerate}
		\item $\dist {y_1}{y_2} \leq \abs{\dist x{y_1} - \dist x{y_2}} + 4r + 44\alpha$, and
		\item $\dist {p_1}{p_2} \leq \abs{\dist x{p_1} - \dist x{p_2}} +  8\alpha$.
	\end{enumerate}
\end{lemm}

\begin{proof}
	Let $c$ be a cocycle in the intersection $\mathcal O_x(y_1, r) \cap \mathcal O_x(y_2,r)$.
	Let $\gamma$ be a gradient arc from $x$ to $c$.
	According to \autoref{res: wysiwyg shadow} the points $y_1$ and $y_2$ lie in the $(r+11\alpha)$-neighborhood of $\gamma$, while $p_1$ and $p_2$ are $2\alpha$-closed to $\gamma$.
	The result follows.
\end{proof}

In order to state the next lemma, we need a notion of spheres in $G$ (for the metric induced by $X$).
This is the purpose of the following notation.
Given $\ell, a \in \R_+$, we let
\begin{align*}
	S(\ell, a) & = \set{g \in G}{\ell - a \leq \dist o{go} < \ell + a}, 
\end{align*}

\begin{lemm}
\label{res: contracting extension}
	Assume that $G$ has a contracting element.
	There is $\alpha \in \R^*_+$ such that for every $L \in \R_+$, for every $g \in G$, there exist $u \in S(L,\alpha)$ and a geodesic $\tau \colon \intval 0T \to X$ from $o$ to $guo$ with the following properties
	\begin{enumerate}
		\item The path $\gamma$ restricted to $\intval{T-L}T$ is $\alpha$-contracting.
		In particular, $gu \in \mathcal T(\alpha, L)$.
		\item The point $go$ is $\alpha$-close to $\tau(T-L)$.
	\end{enumerate}
\end{lemm}

\begin{proof}
	Let $h \in G$ be a contracting element.	
	Denote by $A$ the $\group h$-orbit of $o$.
	It is $\beta$-contracting for some $\beta \in \R^*_+$.
	Let $\alpha_0 \in \R^*_+$ be the parameter given by \autoref{res: segment closed to a contracting element} applied with the element $h$, the point $z = o$ and $d = 2\beta$.
	Up to increasing the value of $\alpha_0$ we can assume that $\alpha_0 \geq  2 \beta + \dist o{ho}$.
	
	Let $L \in \R_+$.
	Since $h$ is contracting, the orbit map $\Z \to X$ sending $n$ to $h^no$ is a quasi-isometric embedding.
	Thus there is $N \in \N$, such that for every $n \geq N$, we have $\dist o{h^no} \geq L+ 2\beta$.
	We choose $N$ minimal with the above property.
	In particular,
	\begin{equation*}
		L + 2\beta 
		\leq \dist o{h^No} \leq \dist o{h^{N-1}o} + \dist o{ho}
		< L + 2\beta + \dist o{ho}.
	\end{equation*}
	Hence $h^N$ and $h^{-N}$ belong to $S(L, \alpha_0)$.
	
	Let $g \in G$.
	There is $k \in \Z$, such that $h^ko$ is a projection of $g^{-1}o$ onto $A$.
	If $k \leq 0$ (\resp $k \geq 0$), we choose $u = h^N$ (\resp $u = h^{-N}$).
	We now prove that $u$ satisfies the announced properties.
	We suppose that $k \geq 0$.
	The other case works in the exact same way.
	Let $\gamma \colon \intval 0T \to X$ be a geodesic from $g^{-1}o$ to $h^{-N}o$.
	According to \autoref{res: proj contracting set}, $h^ko$ is $2\beta$-close to the entry point $\gamma(t)$ of $\gamma$ in $\mathcal N_\beta(A)$.
	By our choice of $N$, we have
	\begin{equation*}
		T - t \geq \dist{\gamma(t)}{\gamma(T)} \geq \dist {h^ko}{h^{-N}o} - 2 \beta \geq \dist o{h^{N+k}o} - 2\beta \geq L.
	\end{equation*}
	In particular, $t \leq T-L\leq T$.
	It follows from \autoref{res: segment closed to a contracting element} and our choice of $\alpha_0$ that 
	\begin{itemize}
		\item $\gamma$ restricted to $\intval {T-L}T$ is  $\alpha_0$-contracting;
		\item there is $s \in \intval {T-L}T$ such that $\dist {\gamma(s)}o \leq \alpha_0$.
	\end{itemize}
	Using the triangle inequality we observe that 
	\begin{equation*}
		\abs{\dist{\gamma(s)}{h^{-N}o} - \dist o{h^{-N}o}} \leq \dist{\gamma(s)}o \leq \alpha_0.
	\end{equation*}
	On the one hand $\dist {\gamma(s)}{h^{-N}o} = T - s$.
	On the other hand $h^{-N} \in S(L, \alpha_0)$.
	Thus $\abs{(T-L) - s} \leq 2\alpha_0$.
	Consequently the triangle inequality yields
	\begin{equation*}
		\dist o{\gamma(T-L)}
		\leq \dist o{\gamma(s)} + \dist{\gamma(s)}{\gamma(T-L)} \leq 3\alpha_0.
	\end{equation*}
	Observe now that the path $\tau = g\gamma$ satisfies the conclusion of the lemma with the parameter $\alpha = 3\alpha_0$.
\end{proof}

Given $\alpha, r, L, \ell \in \R_+$, we consider the following set
\begin{equation*}
	A_\ell(\alpha,r, L) = \bigcup_{ g \in S(\ell, r) \cap \mathcal T(\alpha, L) } \mathcal O_o(go, r).
\end{equation*}
Observe that for a fixed $\ell \in \R$, the set $A_\ell(\alpha,r,L)$ is a non-decreasing function of $r$ (\resp $\alpha$) and a non-increasing function of $L$.

\begin{prop}
\label{res: pre upper bounds}
	Assume that $G$ contains a contracting element.
	There is $\alpha \in \R^*_+$, such that for every $\omega, a \in \R_+$ and $(\epsilon, r_0) \in \R_+^* \times \R_+$, there exist $r_1, C \in \R_+^*$, with the following property.
	Let $\nu = (\nu_x)$ be an $\omega$-conformal density.
	If $(G,\nu)$ satisfies the Shadow Principle with parameters $(\epsilon, r_0)$, then for every $r \geq r_1$, $L > r + 13\alpha$, and $\ell \in \R_+$,
	\begin{equation*}
		\sum_{g \in S(\ell, a)}  \norm{\nu_{go}}  e^{-\omega \dist o{go}}
		\leq C e^{2\omega L}\nu_o \left( A_{\ell+L}(\alpha, r,L) \right).
	\end{equation*}
\end{prop}

\begin{proof}
	We choose for $\alpha$ the parameter given by \autoref{res: contracting extension}.
	Let $\omega, a \in \R_+$ and $(\epsilon, r_0) \in \R_+^* \times \R_+$.
	The action of $G$ on $X$ is proper, so there is $M \in \N$ such that
	\begin{equation*}
		\card{\set{g \in G}{\dist o{go} \leq 2a + 12\alpha}} \leq M.
	\end{equation*}
	In addition, we set
	\begin{equation*}
		r_1 = \max\{r_0, a + 3\alpha\}.
	\end{equation*}
	Let $\nu = (\nu_x)$ be an $\omega$-conformal density such that $(G, \nu)$ satisfies the Shadow Principle with parameters $(\epsilon, r_0)$.
	Let $r \geq r_1$,  $L > r + 13\alpha$, and $\ell \in \R_+$.
	According to \autoref{res: contracting extension}, for every $g \in S(\ell, a)$, there is $u_g \in S(L,\alpha)$ such that $gu_g$ belongs to $\mathcal T(\alpha, L)$.
	Moreover there is a geodesic $\tau_g \colon \intval 0{T_g} \to X$ joining $o$ to $gu_go$, whose restriction to $\intval {T_g-L}{T_g}$ is $\alpha$-contracting, and such that $go$ is $\alpha$-close to $z_g = \tau_g(T_g-L)$.
	In particular, 
	\begin{equation*}
		\abs{\dist o{go}  + \dist o{u_go} - \dist o{gu_go}} 
		\leq 2\alpha.
	\end{equation*}
	Since $g \in S(\ell, a)$ and $u_g \in S(L,\alpha)$, we get $gu_g \in S(\ell + L, a + 3\alpha)$.
	Hence $\mathcal O_o(gu_go,r)$ is contained in $A_{\ell+L}(\alpha, r, L)$.
	
	Consider now $g,g' \in S(\ell, a)$ such that $\mathcal O_o(gu_go,r) \cap \mathcal O_o(g'u_{g'}o,r)$ is non empty.
	By construction, $go$ and $g'o$ are $\alpha$-close to $z_g$ and $z_{g'}$.
	Note that $z_g = \tau_g(T_g -L)$ is the (unique) projection of $o$ on $\tau_g$ restricted to $\intval{T_g -L}{T_g}$, which is $\alpha$-contracting.	
	A similar statement holds for $\tau_{g'}$.
	Using \autoref{res: intersection contracting shadows} and the triangle inequality, we observe that 
	\begin{align*}
		\dist {go}{g'o}
		\leq \dist {z_g}{z_{g'}} + 2\alpha
		& \leq \abs{ \dist o{z_g} - \dist o{z_{g'}}} + 10\alpha \\
		& \leq \abs{ \dist o{go} - \dist o{g'o}} + 12\alpha \\
		& \leq 2a + 12\alpha.
	\end{align*}
	It follows from our choice of $M$ that any cocycle $c \in \bar X$ belongs to a most $M$ shadows of the form $\mathcal O_o(gu_go,r)$, where $g \in S(\ell,a)$.
	Consequently 
	\begin{equation}
	\label{eqn: pre upper bounds}
		\sum_{g \in S(\ell, a)} \nu_o\left(\mathcal O_o(gu_go,r) \right)
		\leq M\nu_o\left(A_{\ell+L}(\alpha, r, L)\right).
	\end{equation}
	Recall that $(G, \nu)$ satisfies the Shadow Principle with parameters $(\epsilon, r_0)$.
	Hence for every $g \in S(\ell, a)$, we have
	\begin{equation*}
		 \nu_o\left(\mathcal O_o(gu_go,r)\right)
		 \geq \epsilon \norm{\nu_{gu_go}}e^{-\omega\dist o{gu_go}}
		 \geq \epsilon e^{-2\omega (L+\alpha)}\norm{\nu_{go}}e^{-\omega\dist o{go}}.
	\end{equation*}
	The second inequality follows from (\ref{eqn: harnack}).
	Consequently (\ref{eqn: pre upper bounds}) becomes
	\begin{equation*}
		\sum_{g \in S(\ell, a)}  \norm{\nu_{go}}  e^{-\omega \dist o{go}}
		\leq  \left(\frac{M e^{2\omega\alpha}}\epsilon\right) e^{2\omega L}\nu_o\left(A_{\ell+L}(\alpha, r, L)\right). \qedhere
	\end{equation*}
\end{proof}

\begin{coro}
\label{res: pre upper bounds bis}
	Assume that $G$ contains a contracting element.
	For every $\omega, a \in \R_+$ and $(\epsilon, r_0) \in \R_+^* \times \R_+$, there exists $C \in \R_+^*$, with the following property.
	Let $\nu = (\nu_x)$ be an $\omega$-conformal density.
	Assume that $(G, \nu)$ satisfies the Shadow Principle with parameters $(\epsilon, r_0)$.
	For every $\ell \in \R_+$,
	\begin{equation*}
		\sum_{g \in S(\ell, a)}  \norm{\nu_{go}}  e^{-\omega \dist o{go}}
		\leq C \nu_o \left( \bar X \setminus B(o, \ell) \right).
	\end{equation*}
\end{coro}

\begin{proof}
	Denote by $\alpha, r_1, C$ the parameters given by \autoref{res: pre upper bounds} applied with $\omega$, $a$, and $(\epsilon,r_0)$.
	We choose $r \geq r_1$ and $L > \max \{ 2r, r + 15\alpha\}$.
	Let $\nu = (\nu_x)$ be an $\omega$-conformal density such that $(G, \nu)$ satisfies the Shadow Principle with parameters $(\epsilon, r_0)$.
	By \autoref{res: pre upper bounds}, we have
	\begin{equation*}
		\sum_{g \in S(\ell, a)}  \norm{\nu_{go}}  e^{-\omega \dist o{go}}
		\leq Ce^{2\omega L}\nu_o \left( A_{\ell+L}( \alpha, r, L) \right), \quad \forall \ell \in \R_+.
	\end{equation*}
	Consider now $x,y,z \in X$.
	By the very definition of shadows, if $z \in \mathcal O_x(y,r)$, then 
	\begin{equation*}
		\dist xy - \dist xz \leq \gro xzy \leq r.
	\end{equation*}
	Hence for every $\ell \in \R_+$, the set $A_{\ell+L}(\alpha, r, L)$ lies in $\bar X \setminus B(o, \ell + L - 2r)$.
	The result follows from the fact that $L \geq 2r$.
\end{proof}

\subsection{First applications}
\label{sec: first application}

For our first applications, we assume that $G$ is a group acting properly, by isometries on $X$ with a contracting element.
In addition, we suppose that $G$ is not virtually cyclic.

\begin{prop}
	There exists $C \in \R_+$ such that for every $\ell \in \R_+$,
	\begin{equation*}
		\card{ \set{g \in G}{ \dist o{go} \leq \ell}} \leq C e^{\omega_G\ell}.
	\end{equation*}
\end{prop}

\begin{rema*}
	An alternative proof of this fact can be found in Yang \cite{Yang:2019wa}.
\end{rema*}

\begin{proof}
	According to \autoref{res: existence twisted ps} there exists a $G$-invariant, $\omega_G$-conformal density $\nu = (\nu_x)$.
	By \autoref{res: shadow lemma}, the pair $(G, \nu)$ satisfies the Shadow Principle for some parameters $(\epsilon, r_0) \in \R_+^* \times \R_+$.
	Fix $a \in \R_+^*$.
	Applying \autoref{res: pre upper bounds bis}, there exists $C \in \R_+$ such that for every $\ell \in \R_+$, 
	\begin{equation*}
		\sum_{g \in S(\ell, a)} \norm{\nu_{go}}  e^{-\omega_G \dist o{go}}
		\leq C \nu_o \left( \bar X \setminus B(o, \ell) \right) \leq C.
	\end{equation*}
	However, since $\nu$ is $G$-invariant, $\norm{\nu_{go}} = \norm{\nu_o} = 1$, for every $g \in G$.
	Hence we get
	\begin{equation*}
		\card{S(\ell, a)} \leq Ce^{\omega_G a}e^{\omega_G \ell}, \quad \forall \ell \in \R_+.
	\end{equation*}
	The result follows by summing this inequality over $\ell \in a \N$.
\end{proof}

Before stating our next application, we define the radial limit set for the action of $G$ on $X$.

\begin{defi}
\label{def: radial limit set}
	Let $r \in \R_+$ and $ \in X$.
	The set $\Lambda_{\rm rad}(G, x,r)$ consists of all cocycles $c \in \partial X$ with the following property:  for every $T \geq 0$, there exists $g \in G$ with $\dist x{go} \geq T$ such that $c \in \mathcal O_x(go,r)$.
	The \emph{radial limit set} of $G$ is the union 
	\begin{equation*}
		\Lambda_{\rm rad}(G) 
		= \bigcup_{ r \geq 0} 
		G \Lambda_{\rm rad}(G, o, r).
	\end{equation*}
\end{defi}

\begin{rema}
\label{rem: radial limit set}
	Note that $\Lambda_{\rm rad}(G, x,r)$ is a non-decreasing function of $r$.
	The radial limit set $\Lambda_{\rm rad}(G)$ is $G$-invariant.
	It follows from (\ref{eqn: inclusion shadows triangle inequality}) that 
	\begin{equation*}
		\Lambda_{\rm rad}(G) = 
		\bigcup_{r \geq 0} 
		\Lambda_{\rm rad}(G, o, r).
	\end{equation*}
\end{rema}

\begin{lemm}
\label{res: radial limit set saturated}
	The radial limit set is saturated.
\end{lemm}

\begin{proof}
	Let $x,y \in X$ and $r \in \R_+$.
	Let $c, c' \in \partial X$ such that $c \sim c'$.
	Note that if $c$ belongs to $\mathcal O_x(y,r)$ then $c'$ belongs to $\mathcal O_x(y,r')$ where $r' = r + \norm[\infty]{c- c'}$.
	The result follows from this observation.
\end{proof}

\begin{prop}
\label{res: series roblin}
	Let $\omega \in \R_+$ and $\nu = (\nu_x)$ be an $\omega$-conformal density.
	If $(G, \nu)$ satisfies the Shadow Principle then the series 
	\begin{equation}
	\label{eqn: series roblin}
		\sum_{g \in G} \norm{\nu_{go}}e^{-s\dist o{go}} 
	\end{equation}
	converges whenever $s > \omega$.
	If, in addition, $\nu_o$ gives positive measure the radial limit set $\Lambda_{\rm rad}(G)$, then the series diverges at $s = \omega$.
	In particular its critical exponent is exactly $\omega$.
\end{prop}

\begin{proof}
	Fix $a \in \R_+^*$.
	According to \autoref{res: pre upper bounds bis} there is $C \in \R_+^*$ such that for every $\ell \in \R_+$,
	\begin{equation*}
		\sum_{g \in S(\ell, a)}\norm{\nu_{go}}  e^{-\omega \dist o{go}} 
		\leq C \nu_o \left( \bar X \setminus B(o, \ell) \right)
		\leq C.
	\end{equation*}
	Hence the series (\ref{eqn: series roblin}) converges whenever $s > \omega$.
	The second part of the proposition is proved by contraposition.
	Suppose that the series (\ref{eqn: series roblin}) converges at $s = \omega$.
	Let $r \in \R_+$.
	For simplicity we write  $\Lambda = \Lambda_{\rm rad}(G, o, r)$.
	Observe that for every $T \in \R_+$, 
	\begin{equation*}
		\Lambda \subset  \bigcup_{\substack{g \in G \\ \dist o{go} \geq T}} \mathcal O_o(go, r).
	\end{equation*}
	It follows from \autoref{rem: shadow lemma}, that 
	\begin{equation*}
		\nu_o(\Lambda)
		\leq \sum_{\substack{g \in G \\ \dist o{go} \geq T}} \nu_o\left(\mathcal O_o(go, r)\right)
		\leq e^{2\omega r} \sum_{\substack{g \in G \\ \dist o{go} \geq T}} \norm{\nu_{go}}e^{-\omega \dist o{go}}.
	\end{equation*}
	The right-hand side of the inequality is the remainder of the series we are interested in.
	Since this series converges at $s = \omega$, we get $\nu_o(\Lambda) = 0$.
	We observed in \autoref{rem: radial limit set} that
	\begin{equation*}
		\Lambda_{\rm rad}(G) =  \bigcup_{r \geq 0} \Lambda_{\rm rad}(G, o, r).
	\end{equation*}
	Hence $\nu_o(\Lambda_{\rm rad}(G)) = 0$.
\end{proof}

\begin{rema}
\label{rem:  series roblin}
	Note that the proof of the second assertion (the series diverges at $s = \omega$, whenever $\nu_o$ gives positive measure to the radial limit set) only uses the upper estimate from the Shadow lemma.
	Hence it holds even if $G$ is virtually cyclic or does not contain a contracting element.
\end{rema}

\begin{coro}
\label{res: lower bound dimension density}
	Let $\omega \in \R_+$.
	Let $\nu = (\nu_x)$ be a $G$-invariant, $\omega$-conformal density.
	Then $\omega \geq \omega_G$.
	Moreover if $\nu_o$ gives positive measure to the radial limit set $\Lambda_{\rm rad} (G)$, then $\omega = \omega_G$ and the action of $G$ on $X$ is divergent.
\end{coro}

\begin{rema*}
	If the action of $G$ on $X$ is proper an co-compact, one checks that the radial limit set is actually $\partial X$.
	Hence if $\nu = (\nu_x)$ is a $G$-invariant, $\omega$-conformal density supported on $\partial X$, then $\omega = \omega_G$.
\end{rema*}

\begin{proof}
	By \autoref{res: shadow lemma}, the pair $(G, \nu)$ satisfies the Shadow Principle.
	Since the density $\nu$ is $G$-invariant, the series (\ref{eqn: series roblin}) which appears in \autoref{res: series roblin} is exactly the Poincaré series of $G$.
	The result follows.
\end{proof}

\begin{coro}
\label{res: dimension almost-fixed}
	Let $\omega \in \R_+$.
	Let $\nu = (\nu_x)$ be an $\omega$-conformal density and $\mu = (\mu_x)$ its restriction to the reduced horocompactification $(\bar X, \mathfrak R)$.
	Assume that $(G, \nu)$ satisfies the Shadow Principle.
	If $\mu$ is almost-fixed by $G$, then $\omega \geq \omega_G$.
\end{coro}

\begin{proof}
	According to \autoref{res: almost-fixed gives quasi-morphism} the map $\chi \colon G \to \R_+$ sending $g$ to $\ln\norm{\nu_{go}}$ is a quasi-morphism.
	Note that the exponent of the series
	\begin{equation*}
		\sum_{g \in G} \norm{\nu_{go}} e^{-s\dist o{go}} 
		= \sum_{g \in G} e^{\chi(g)} e^{-s\dist o{go}} 
	\end{equation*}
	is exactly $\omega_\chi$.
	We observed earlier that $\omega_\chi \geq \omega_G$.
	The result follows from \autoref{res: series roblin}.
\end{proof}

\begin{coro}
\label{res: amenability roblin}
	Let $N$ be a normal subgroup of $G$ such that $G/N$ is amenable.
	Then $\omega_N = \omega_G$.
\end{coro}

\begin{proof}
	The proof follows Roblin's argument \cite{Roblin:2005fn}.
	Assume first that $N$ is finite.
	This forces $G$ to be itself amenable.
	However since $G$ contains a contracting element, $G$ has to be virtually cyclic (otherwise $G$ would contain a free subgroup). 
	Hence $\omega_G = 0$ and the result holds.
	Suppose now that $N$ is infinite.
	For simplicity we let $Q = G/N$.
	Denote by $\nu = (\nu_x)$ an $N$-invariant, $\omega_N$-conformal density on $\bar X$.
	We are going to ``average'' the $G$-orbit of $\nu$ to produce another $N$-invariant, $\omega_N$-conformal density.
	The construction goes as follows.
	Choose a $Q$-invariant mean $M \colon \ell^\infty(Q) \to \R$.
	Let $x \in X$ and $f \in C(\bar X)$.
	Consider the function
	\begin{equation*}
		\begin{array}{rccc}
			\psi_{x,f} \colon & G & \to & \R \\
			& u & \mapsto & \displaystyle \int f d\nu_x^u
		\end{array}
	\end{equation*}
	According to (\ref{eqn: harnack}), we have $\norm {\nu_x^u} \leq e^{\omega_N \dist ox}$, for every $u \in G$.
	Consequently, $\psi_{x,f}$ is a bounded function.
	Since $\nu$ is $N$-invariant, we also observe that $\psi_{x,f}$ induces a map bounded map $Q \to \R$, that we still denote by $\psi_{x,f}$.
	Using Riesz representation theorem we define a new measure $\mu_x$ by imposing
	\begin{equation*}
		\int f d\mu_x = M \left(\psi_{x,f}\right),\quad \forall f\in C(\bar X).
	\end{equation*}
	One checks that $\mu = (\mu_x)$ is still $\omega_N$-conformal. 
	Since $N$ is a normal subgroup, it is also $N$-invariant.
	In particular, it satisfies the Shadow Principle (\autoref{res: shadow lemma}).
	According to \autoref{res: series roblin} the series
	\begin{equation}
	\label{eqn: amenability roblin}
		\sum_{g \in G} \norm{\mu_{go}}e^{-s\dist o{go}} 
	\end{equation}
	converges whenever $s > \omega_N$. Consider now the map 
	\begin{equation*}
		\begin{array}{rccc}
			\chi \colon & G & \to & \R \\
			& g & \mapsto & M \left(\ln \psi_{go,\mathbb 1}\right).
		\end{array}
	\end{equation*}
	Note that 
	\begin{equation*}
		\psi_{g_1g_2o,\mathbb 1}(u) = \psi_{g_2o,\mathbb 1}(ug_1) \psi_{g_1o,\mathbb 1}(u), \quad \forall u,g_1,g_2 \in G.
	\end{equation*}
	Since $M$ is $Q$-invariant, $\chi$ is a homomorphism (factoring through the projection $G \onto Q$).
	The exponential is a convex function, hence the Jensen inequality yields
	\begin{equation*}
		e^{\chi(g)} \leq \norm {\mu_{go}}, \quad \forall g \in G.
	\end{equation*}
	Consequently 
	\begin{equation*}
		\sum_{g \in G} \norm{\mu_{go}}e^{-s\dist o{go}}
		\geq \sum_{g \in G} e^{\chi(g)}e^{-s\dist o{go}}
	\end{equation*}
	The critical exponent of the sum on the right hand side is $\omega_\chi$.
	Combined with the above discussion it yields $\omega_N \geq \omega_\chi$.
	We observed previously that for a homomorphism $\chi$, we have $\omega_\chi \geq \omega_G$, thus $\omega_N \geq \omega_G$.
	The other inequality follows from the fact that $N \subset G$.
\end{proof}

\begin{rema}
	In the proof, we have used the fact that $N$ is normal in a crucial way to ensure that $\mu$ is $N$-invariant and therefore satisfies the Shadow Principle.
	It would be nice to generalize this strategy to the case of a \emph{co-amenable} subgroup $H$, i.e. when there is a $G$-invariant mean on the coset space $H\backslash G$.
	We know from the Patterson-Sullivan construction that $H$ fixes a point $\nu$ in $\mathcal D(\omega_H)$.
	Since $G$ acts on this compact space, it would be tempting to use the fixed point characterization of co-amenability due to Eymard \cite[Exposé~1, §2]{Eymard:1972ug} and conclude that $G$ should fix a point $\mu$ in $\mathcal D(\omega_H)$.
	Even though $\mu$ may not be $H$-invariant that would be enough to ensure that it satisfies the Shadow Principle.
	One has to be careful though that the action of $G$ on $\mathcal D(\omega_H)$ is not affine.
	Hence the fixed point criterion does not apply here.
	In this context, one can easily prove that the ``average'' density $\mu$ we have built is $G$-quasi-invariant. 
	Unfortunately, this information is not sufficient to prove that it satisfies the Shadow Principle.
\end{rema}

\begin{coro}
\label{res: improved lower bound growth normal sbgp}
	Let $N$ be an infinite, normal subgroup of $G$.
	Then
	\begin{equation*}
		\omega(N,X) + \frac 12 \omega(G/N, X/N) \geq \omega(G, X).
	\end{equation*}
\end{coro}

\begin{proof}
	Let $Q = G/N$.
	Denote by $\pi \colon G \onto Q$ and $\zeta \colon X \to X/N$ the canonical projections.
	For simplicity we write $\omega_N$ and $\omega_G$ for the growth rates of $N$ and $G$ acting on $X$, while $\omega_Q$ stands for the growth rate of $Q$ acting on $X/N$.
	We denote by $\mathcal H$ the Hilbert space $\mathcal H = \ell^2(Q)$.
	Given $s,t \in \R_+$ we consider the following maps
	\begin{equation*}
		\begin{array}{rccccrccc}
			\phi_s \colon & Q & \to & \R 
			& \quad \text{and} \quad \
			& \psi_t \colon & Q & \to & \R \\
			& q & \mapsto &\displaystyle  \sum_{g \in \pi^{-1}(q)} e^{-s\dist o{go}}
			&&
			& q & \mapsto & \displaystyle e^{-t\dist {\zeta(o)}{q\zeta(o)}}
		\end{array}
	\end{equation*}
	One checks that $\psi_t \in \mathcal H$ (\resp $\psi_t \notin \mathcal H$) whenever $2t >  \omega_Q$ (\resp $2t <  \omega_Q$).
	Similarly $\phi_s \notin \mathcal H$, whenever $s < \omega_N$.
	We now prove that the converse essentially holds true.
	
	\begin{clai}
		If $s > \omega_N$, then $\phi_s \in \mathcal H$.
	\end{clai}
	
	Let $s > \omega_N$.
	Consider the $N$-invariant, $s$-conformal density $\nu^s = (\nu^s_x)$ on $\bar X$ defined by
	\begin{equation*}
		\nu^s_x = \frac 1 {\mathcal P_N(s)} \sum_{h \in N} e^{-s \dist x{ho}} {\rm Dirac}(ho).
	\end{equation*}
	The computation shows that 
	\begin{align*}
		\norm{\phi_s}^2 
		= \sum_{\substack{g_1,g_2 \in G\\ \pi(g_1)=\pi(g_2)}} e^{-s \left[\dist o{g_1o} + \dist o{g_2o}\right]}
		& = \sum_{g \in G}e^{-s \dist o{go}} \left(\sum_{h \in N} e^{-s\dist{go}{ho}}\right) \\
		& = \mathcal P_N(s)  \sum_{g \in G} \norm{\nu^s_{go}} e^{-s \dist o{go}} .
	\end{align*}
	Fix $a \in \R_+^*$.
	According to \autoref{res: shadow lemma}, $\nu^s$ satisfies the Shadow Principle.
	By \autoref{res: pre upper bounds bis} there exists $C \in\R_+^*$, such that for every $\ell \in \R_+$,
	\begin{equation*}
		\sum_{g \in S(\ell, a)} \norm{\nu^s_{go}}  e^{-s \dist o{go}} 
		\leq C\nu^s_o\left(\bar X \setminus B(o, \ell)\right).
	\end{equation*}
	Unfolding the definition of $\nu^s$, we get
	\begin{equation*}
		\mathcal P_N(s)\sum_{g \in S(\ell, a)}  \norm{\nu^s_{go}}  e^{-s \dist o{go}}
		\leq C \sum_{\substack{h \in N \\ \dist o{ho} \geq \ell}} e^{-s \dist o{ho}}.
	\end{equation*}
	It follows that 
	\begin{equation*}
		\norm{\phi_s}^2
		\leq C \sum_{k \in \N} \sum_{\substack{h \in N \\ \dist o{ho} \geq ka}} e^{-s \dist o{ho}}
		\leq C \sum_{h \in N} \sum_{\substack{k \in \N \\ ka \leq \dist o{ho}}}  e^{-s \dist o{ho}}.
	\end{equation*}
	Up to replacing $C$ by a larger constant, we have
	\begin{equation}
	\label{eqn: improved lower bound growth normal sbgp}
		\norm{\phi_s}^2
		\leq C \sum_{h \in N} \left[1 + \dist o{ho}\right] e^{-s \dist o{ho}}.
	\end{equation}
	Note that the series 
	\begin{equation*}
		-\sum_{h \in N} \dist o{ho} e^{-s\dist o{ho}}
	\end{equation*}
	is the derivative of the Poincaré series of $N$, hence it converges.
	Consequently the right-hand side in (\ref{eqn: improved lower bound growth normal sbgp}) converges.
	Thus $\norm{\phi_s}$ is finite, which completes the proof of our claim.

	Let $s, t \in \R_+$ with $s > \omega_N$ and $2t > \omega_Q$.
	The scalar product of $\phi_s$ and $\psi_t$ can be computed as follows:
	\begin{align*}
		\left( \phi_s, \psi_t\right)
		& = \sum_{q \in Q} e^{-t \dist {\zeta(o)}{q\zeta(o)}} \left(\sum_{g \in \pi^{-1}(q)} e^{-s\dist o{go}}\right) \\
		& = \sum_{g \in G} e^{-s \dist o{go}}e^{-t \dist{\zeta(o)}{\pi(g)\zeta(o)}}.
	\end{align*}
	The projection $\zeta \colon X \onto X/N$ is $1$-Lipschitz.
	Combined with the Cauchy-Schwartz inequality, it gives
	\begin{equation*}
		\mathcal P_G(s+t)
		\leq \left( \phi_s, \psi_t\right)
		\leq \norm{\phi_s}\norm{\psi_t} < \infty.
	\end{equation*}
	Consequently $s + t \geq \omega_G$.
	This inequality holds for every $s, t \in \R_+$ with $s > \omega_N$ and $2t > \omega_Q$, whence the result.
\end{proof}

\begin{rema*}
 	We keep the notations of the proof.
	\autoref{res: improved lower bound growth normal sbgp} is sharp.
	\begin{itemize}
		\item 
		Indeed, assume that $X$ is a Cayley graph of $G$.
		If $Q$ has subexponential growth, then $\omega_Q = 0$ and $Q$ is amenable, so that $\omega_N = \omega_G$ by \autoref{res: amenability roblin}.
		\item At the other end of the spectrum, assume that $G = \free r$ is the free group of rank $r$ acting on its Cayley graph $X$ with respect to a free basis.
		Let $g \in G \setminus\{1\}$.
		For every $k \in \N$, denote by $N_k$ the normal closure of $g^k$.
		Then
		\begin{equation*}
			\lim_{k \to \infty}\omega_{N_k} = \frac 12 \omega_G
			\quad \text{and} \quad
			\lim_{k \to \infty}{\omega_{G/N_k}} = \omega_G.
		\end{equation*}
		The first limit is due to Grigorchuk \cite{Grigorchuk:1977vw}, see also Champetier \cite{Champetier:1993gs}.
		The second limit was proved by Shukhov \cite{Shukhov:1999ex}, see also Coulon \cite{Coulon:2013ab}.
	\end{itemize}
\end{rema*}

\section{Divergent actions}
\label{sec: div action}

\subsection{Contracting tails, continued}

We continue here our study of shadows of elements $g \in \mathcal T(\alpha, L)$ having a contracting tail.
Before we recall the proof of the following classical result.

\begin{lemm}
\label{res: neighbored geodesics}
	Let $C \in \R_+$.
	Let $x,y \in X$ and  $\gamma \colon \intval ab \to X$ be a continuous path from $x$ to $y$.
	Let $\nu \colon I \to X$ be a geodesic.
	Denote by $p = \nu(c)$ and $q = \nu(d)$ respective projections of $x$ and $y$ on $\nu$ and suppose that $c \leq d$.
	Assume in addition that $\gamma$ lies in the $C$-neighborhood of $\nu$.
	Then $\nu$ restricted to $\intval cd$ lies in the $2C$-neighborhood of $\gamma$.
\end{lemm}

\begin{proof}
	Fix a point $r = \nu(t_0)$ on $\nu$ where $t_0 \in \intval cd$.
	Consider the set $J$ consisting of all $s \in \intval ab$ such that $\gamma(s)$ admits a projection on $\nu$ of the form $\nu(t)$ with $t \leq t_0$.
 	Note that $J$ is non empty as it contains $a$.
	Now set $s_0 = \sup J$.
	Let $\epsilon > 0$.
	Since $\gamma$ is continuous, there is $\eta > 0$ such that for every $s \in \intval ab$, we have $\dist {\gamma(s)}{\gamma(s_0)} \leq \epsilon$ provided $\abs{s-s_0} \leq \eta$.
	By definition of $J$, there is $s_1 \in (s_0-\eta, s_0]$ such that $\gamma(s_1)$ admits a projection on $\nu$ of the form $\nu(t_1)$ with $t_1 \leq t_0$.
	On the other hand, we claim that there is $s_2 \in [s_0, s_0 + \eta)$ such that $\gamma(s_2)$ admits a projection on $\nu$ of the form $\nu(t_2)$ with $t_2 \geq t_0$.
	Indeed if $s_0 = b$, we can simply take $s_2 = b$.
	Otherwise it follows from the definition of $J$.
	By assumption $\dist{\nu(t_i)}{\gamma(s_i)} \leq C$.
	Hence the triangle inequality yields
	\begin{equation*}
		\dist {\nu(t_1)}{\nu(t_2)} 
		\leq \dist {\gamma(s_1)}{\gamma(s_2)} + 2C
		\leq 2C + 2\epsilon.
	\end{equation*}
	and then 
	\begin{equation*}
		d(r, \gamma) 
		\leq \frac 12 \dist {\nu(t_1)}{\nu(t_2)} + \max \left\{ \dist{\nu(t_1)}{\gamma(s_1)}, \dist{\nu(t_2)}{\gamma(s_2)}\right\}
		\leq 2C + \epsilon.
	\end{equation*}
	This inequality holds for every $\epsilon > 0$, whence the result.
\end{proof}

\begin{lemm}
\label{res: pre-vitali}
	Let $\alpha \in \R_+^*$ and $r, L \in \R_+$ with $L > r + 16\alpha$.
	Let $x,y_1, y_2 \in X$ with $\dist x{y_1} \leq \dist x{y_2}$.
	Assume that the pairs $(x,y_1)$ and $(x,y_2)$ have an $(\alpha, L)$-contracting tail.
	If $\mathcal O_x(y_1,r) \cap \mathcal O_x(y_2,r)$ is not empty, then $\mathcal O_x(y_2,r)$ is contained in $\mathcal O_x(y_1,r + 42\alpha)$.
\end{lemm}

\begin{proof}
	By assumption there is an $\alpha$-contracting geodesic $\tau_i$ ending at $y_i$, and a projection $p_i$ of $x$ on $\tau_i$ satisfying $\dist{p_i}{y_i} \geq L$.
	In addition, we write $z_i$ for the point on $\tau_i$ satisfying $\dist {y_i}{z_i} = r + 13\alpha$.
	We split the proof in several claims.
	
	\begin{clai}
		Let $c \in \mathcal O_x(y_i,r)$ and $\gamma \colon I \to X$ a gradient line from $x$ to $c$.
		Then $z_i$ lies in the $5\alpha$-neighborhood of $\gamma$.
		In particular $c \in \mathcal O(z_i, 5\alpha)$.
	\end{clai}
	
	According to \autoref{res: wysiwyg shadow}, $\gamma$ intersects $\mathcal N_\alpha(\tau_i)$.
	Denote by $\gamma(s_i)$ and $\gamma(t_i)$ the entry and exit point of $\gamma$ in $\mathcal N_\alpha(\tau_i)$.
	We know, again from \autoref{res: wysiwyg shadow}, that $\dist{\gamma(s_i)}{p_i} \leq 2 \alpha$ and $\dist{\gamma(t_i)}{y_i} \leq r + 12\alpha$.
	In particular, if $m_i$ and $n_i$ stands for projections of $\gamma(s_i)$ and $\gamma(t_i)$ on $\tau_i$ respectively then $m_i$, $z_i$, $n_i$ and $y_i$ are aligned in this order along $\tau_i$.
	The path $\gamma$ restricted to $\intval{s_i}{t_i}$ lies in the $5\alpha/2$-neighborhood of $\tau$ (\autoref{res: qc contracting set}).
 	It follows that $z_i$ is $5\alpha$-close to $\gamma$ (\autoref{res: neighbored geodesics}).

	\begin{clai}
		We have $\gro x{z_2}{z_1} \leq 24\alpha$.	
	\end{clai}
	
	According to our assumption $\mathcal O_x(y_1,r) \cap \mathcal O_x(y_2,r)$ is not empty.
	Let $\gamma \colon I \to X$ be a gradient arc from $x$ to a cocycle $c$ in this intersection.
	It follows form our previous claim that there is $u_i \in \R_+$ such that $\gamma(u_i)$ is $5\alpha$-close to $z_i$.
	Hence the triangle inequality yields
	\begin{equation}
	\label{eqn: pre-vitali}
		\gro x{z_2}{z_1} \leq \gro x{\gamma(u_2)}{\gamma(u_1)} + 10\alpha.
	\end{equation}
	Observe that
	\begin{equation*}
		\dist  x{z_1} + \dist{z_1}{y_1} - 4 \alpha
		\leq \dist x{y_1}
		\leq \dist x{y_2}
		\leq \dist x{z_2} + \dist {z_2}{y_2}.
	\end{equation*}
	Indeed, the first inequality follows from \autoref{res: proj contracting set} applied to $\tau_1$, the second one is part of our assumptions, while the last one is just the triangle inequality.
	By construction $\dist {y_1}{z_1} = \dist {y_2}{z_2}$.
	Combined with the triangle inequality it yields
	\begin{equation*}
		\dist x{\gamma(u_1)} \leq \dist x{\gamma(u_2)} + 14\alpha.
	\end{equation*}
	Therefore $\gro x{\gamma(u_2)}{\gamma(u_1)} \leq 14\alpha$, which combined with (\ref{eqn: pre-vitali}) completes the proof of our second claim.

	\begin{clai}
		 $\mathcal O_x(y_2,r)$ is contained in $\mathcal O_x(y_1,r + 42\alpha)$.
	\end{clai}
	
	Consider $c \in \mathcal O_x(y_2,r)$.
	According to our first claim $\gro xc{z_2} \leq 5\alpha$.
	Since cocyles are $1$-Lipschitz, we get in combination with the previous claim
	\begin{equation*}
		\gro xc{z_1} \leq \gro xc{z_2} + \gro x{z_2}{z_1} \leq 29 \alpha.
	\end{equation*}
	However $\dist {y_1}{z_1} = r + 13\alpha$.
	It follows from (\ref{eqn: gromov product - lip}) that $\gro xc{y_1} \leq r + 42\alpha$.
\end{proof}

The next statement will be used later to estimate the measures of various saturated sets using a Vitali type argument.

\begin{lemm}
\label{res: vitali}
	Let $\alpha \in \R_+^*$ and $r, L \in \R_+$ with $L > r + 16\alpha$.
	Let $S$ be a subset of $\mathcal T(\alpha, L)$.
	There is a subset $S^* \subset S$ with the following properties
	\begin{enumerate}
		\item \label{enu: vitali - disjoint}
		The collection $\left(\mathcal O_o(go,r)\right)_{g \in S^*}$ is pairwise disjoint.
		\item \label{enu: vitali - cover}
		$\displaystyle \bigcup_{g \in S} \mathcal O_o(go,r) \subset  \bigcup_{g \in S^*} \mathcal O_o(go,r+42\alpha)$.
	\end{enumerate}
\end{lemm}

\begin{proof}
	Since the action of $G$ on $X$ is proper, we can index the elements $g_0, g_1, g_2 \dots$ of $S$ such that $\dist o{g_io} \leq \dist o{g_{i+1}o}$, for every $i \in \N$.
	We build by induction a sequence $(i_n)$ as follows.
	Set $i_0 = 0$.
	Let $n \in \N$ for which $i_n$ has been defined.
	We search for the minimal index $j > i_n$ such that
	\begin{equation*}
		\left( \bigcup_{k = 0}^n \mathcal O_o(g_{i_k}o,r)\right) \cap  \mathcal O_o(g_jo,r)  = \emptyset.
	\end{equation*}
	If such an index exists, we let $i_{n+1} = j$.
	Otherwise we let $i_{n+1} = i_n$, in other words, the sequence $(i_n)$ eventually stabilizes.
	Finally we set 
	\begin{equation*}
		S^* = \set{g_{i_n}}{n \in \N}.
	\end{equation*}
	Note that \ref{enu: vitali - disjoint} directly follows from the construction.
	Let $g \in S$.
	If $g$ does not belong to $S^*$, it means that there is $h \in S^*$, with $\dist o{ho} \leq \dist o{go}$ such that $\mathcal O_o(go,r) \cap \mathcal O_o(ho,r)$ is non-empty.
	Hence by \autoref{res: pre-vitali}, the shadow $\mathcal O_o(go,r)$ is contained in $\mathcal O_o(ho,r+42\alpha)$.
	This completes the proof of \ref{enu: vitali - cover}.
\end{proof}

\begin{lemm}
\label{res: cocycle locally determined}
	Let $\alpha \in \R_+^*$ and $r, L \in \R_+$ with $L > r + 13\alpha$.
	Let $x,y \in X$ such that $(x,y)$ has an $(\alpha, L)$-contracting tail.
	Let $K$ be a closed ball of radius $R$ centered at $x$.
	If $\dist xy > R + r +  13\alpha$, then for every $c,c' \in \mathcal O_x(y,r)$, we have $\norm[K]{c-c'} \leq 20\alpha$.
\end{lemm}

\begin{proof}
	Denote by $\tau$ a contracting tail of $(x,y)$ and let $p$ be a projection of $x$ on $\tau$ such that $\dist py \geq L$.
	Fix $x' \in K$ and denote by $p'$ a projection of $x'$ onto $\tau$.
	Let $q$ be a projection of $c$ on $\tau$, so that $\dist yq \leq r + 7\alpha$  (\autoref{res: wysiwyg shadow}).
	We claim that $\dist {p'}q > 4\alpha$.
	If $\dist p{p'} \leq \alpha$, then this is just a consequence of the triangle inequality.
	Thus we can assume that $\dist p{p'} > \alpha$.
	According to \autoref{res: proj contracting set} we have $\dist x{p'} \leq \dist x{x'} + 2 \alpha$.
	Hence
	\begin{align*}
		\dist {p'}q 
		\geq \dist xy - \dist x{p'} - \dist qy 
		& \geq \dist xy - \dist x{x'} - \dist qy -2 \alpha \\
		& \geq \dist xy - (R + r + 9\alpha) 
		> 4 \alpha.
	\end{align*}
	It follows from \autoref{res: proj cocycle contracting set} that 
	\begin{equation*}
		\dist {x'}q - 10\alpha \leq c(x',q) \leq \dist {x'}q.
	\end{equation*}
	This estimate holds for every $x' \in K$.
	Consider now $x_1,x_2 \in K$.
	Since $c$ is a cocycle, $c(x_1,x_2) = c(x_1,q) + c(q,x_2)$, thus
	\begin{equation*}
		\abs{c(x_1,x_2) - \left[\dist {x_1}q - \dist {x_2}q\right]}\leq 10\alpha.
	\end{equation*}
	The same argument holds for $c'$, thus $\norm[K]{c-c'} \leq 20\alpha$.
\end{proof}

\subsection{The contracting limit set}

\paragraph{Contracting limit set.}
We now introduce a variation of the the radial limit set that keeps track of the elements fellow-traveling with some contracting geodesics.
Let $\alpha,r,L \in \R_+$ and $x \in X$.
The set $\Lambda_{\rm ctg}(G, x, \alpha, r, L)$ consists of all cocycles $c \in \partial X$ with the following property:  for every $T \geq 0$, there exists $g \in G$ such that
\begin{itemize}
	\item $\dist x{go} \geq T$,
	\item $(x,go)$ has an $(\alpha, L)$-contracting tail,
	\item $c \in \mathcal O_x(go,r)$.
\end{itemize}
We also let
\begin{equation*}
	\Lambda_{\rm ctg} (G, x, \alpha,r) = \bigcap_{ L \in \R_+} \Lambda_{\rm ctg} (G, x, \alpha, r, L). 
\end{equation*}
\begin{rema*}
	Observe that the set $\Lambda_{\rm ctg}(G, x, \alpha, r, L)$ is a non-decreasing (\resp non-increasing) function of $\alpha$ and $r$, (\resp $L$).
\end{rema*}

\begin{defi}
\label{def: contracting limit set}
	The \emph{contracting limit set} of $G$ is 
	\begin{equation*}
		\Lambda_{\rm ctg}(G) = 
		\bigcup_{\alpha, r \in \R_+} G \Lambda_{\rm ctg}(G, o,\alpha, r).
	\end{equation*}
\end{defi}

It follows from the definition that the contracting limit set is $G$-invariant and contained in the radial limit set.

\begin{prop}
\label{res: contracting limit set invariant}
	Let $\alpha \in \R^*_+$.
	For every $r \in \R_+$ and $L > r + 13\alpha$ we have
	\begin{equation*}
		G \Lambda_{\rm ctg}(G, o,\alpha, r, L) \subset \Lambda_{\rm ctg}(G, o, \alpha,r+12\alpha, L- \alpha).
	\end{equation*}
	In particular, $G \Lambda_{\rm ctg}(G, o,\alpha, r) \subset \Lambda_{\rm ctg}(G, o, \alpha,r+12\alpha)$.
	Moreover, the contracting limit set can also be described as
	\begin{equation*}
		\Lambda_{\rm ctg}(G) = 
		\bigcup_{\alpha, r \in \R_+}
		\Lambda_{\rm ctg}(G, o, \alpha, r).
	\end{equation*}
\end{prop}

\begin{proof}
	Let $h \in G$ and $c$ be a cocycle in 
	\begin{equation*}
		h\Lambda_{\rm ctg}(G, o,\alpha, r, L) = \Lambda_{\rm ctg}(G, ho,\alpha, r, L).
	\end{equation*}
	For simplicity we let $b = h^{-1}c$.
	For every $n \in \N$, we can find an element $g_n \in \mathcal T(\alpha, L)$ such that $\dist o{g_no} \geq n$ and $b \in \mathcal O_o(g_no, r)$.
	Denote by $\tau_n$ the contracting tail for the pair $(o,g_no)$ and $p_n$ a projection of $o$ on $\tau_n$ satisfying $\dist {p_n}{g_no} \geq L$.
	Up to passing to a subsequence, we can assume that $\dist o{g_no} \geq \dist o{h^{-1}o} + L + \alpha$, for every $n \in \N$.
	
	Let $n \in \N$.
	We claim that $hg_n$ belongs to $\mathcal T(\alpha, L - \alpha)$.
	Let $p'_n$ be a projection of $h^{-1}o$ on $\tau_n$.
	It suffices to prove that $\dist {p'_n}{g_no} \geq L-\alpha$.
	Indeed, after translating the figure by $h$, it tells us that $h\tau_n$ is a contracting tail for the pair $(o, hg_no)$.
	We distinguish two cases.
	Assume first that $\dist {p_n}{p'_n} > \alpha$.
	It follows from \autoref{res: proj contracting set} that $\dist o{p'_n} \leq \dist o{h^{-1}o} + 2\alpha$.
	The triangle inequality yields
	\begin{equation*}
		\dist{p'_n}{g_no} \geq \dist o{g_no} - \dist o{p'_n} 
		\geq \dist o{g_no} - \dist o{h^{-1}o} - 2\alpha 
		\geq L - \alpha
	\end{equation*}
	Assume now that $\dist{p_n}{p'_n} \leq \alpha$.
	The triangle inequality yields
	\begin{equation*}
		\dist{p'_n}{g_no} \geq \dist {p_n}{g_no} - \dist {p_n}{p'_n} 
		\geq L - \alpha,
	\end{equation*}
	which completes the proof of our claim.
	
	We now prove that $c \in \mathcal O_o(hg_no, r +18\alpha)$.
	Let $q_n$ be a projection of $b$ onto $\tau_n$.
	According to \autoref{res: wysiwyg shadow}, we have $\dist {q_n}{g_no} \leq r + 7\alpha$.
	Combining the above discussion with the triangle inequality we have
	\begin{equation*}
		\dist {p'_n}{q_n} 
		\geq \dist {p'_n}{g_no} - \dist {q_n}{g_no} 
		\geq L - (r + 8\alpha) > 4\alpha.
	\end{equation*}
	On the one hand, by \autoref{res: proj cocycle contracting set}, we have
	\begin{equation*}
		b(h^{-1}o, q_n) \geq \dist {h^{-1}o}{q_n} - 10\alpha,
	\end{equation*}
	i.e. $\gro{h^{-1}o}b{q_n} \leq 5\alpha$.
	Consequently 
	\begin{equation*}
		\gro {h^{-1}o}b{g_no} \leq \gro {h^{-1}o}b{q_n} + \dist {q_n}{g_no}  \leq r + 12\alpha.
	\end{equation*}
	Recall that $b = h^{-1}c$.
	The above inequality implies that $c$ belongs to the shadow $\mathcal O_o(hg_no, r + 12\alpha)$, which completes the proof of our claim.
	Note that $\dist o{hg_no}$ diverges to infinity, hence $c \in  \Lambda_{\rm ctg}(G, o, \alpha, r+12\alpha,L-\alpha)$.
\end{proof}

The points in the contracting limit set have a specific behavior with respect to the equivalence relation $\sim$  used to define the reduced horoboundary (see \autoref{sec: reduced horoboundary}).
Proceeding as for the radial limit set (see \autoref{res: radial limit set saturated}) one checks that the contracting limit set is saturated.
The next statement precises this fact.

\begin{prop}
\label{res: equiv rel restricted to contracting}
	Let $\alpha \in \R_+^*$ and $r, L \in \R_+$ with $L > r + 29\alpha$.
	Let $c,c' \in \partial X$ such that $c \sim c'$.
	Assume that $c$ belongs to $\Lambda_{\rm ctg}(G,o, \alpha, r, L)$.
	Then $\norm[\infty] {c-c'} \leq 20\alpha$ and $c'$ belongs to $\Lambda_{\rm ctg}(G,o, \alpha, r + 16\alpha, L)$.
	In particular, the saturation of $\Lambda_{\rm ctg}(G,o, \alpha, r)$ is contained in $\Lambda_{\rm ctg}(G,o, \alpha, r + 16\alpha)$.
\end{prop}

\begin{proof}
	Since $c$ belongs to $\Lambda_{\rm ctg}(G,o,  \alpha, r, L)$, there exists a sequence of elements $g_n \in \mathcal T(\alpha, L)$ such that $\dist o{g_no}$ diverges to infinity and $c \in \mathcal O_o(g_no,r)$ for every $n \in \N$.
	For each $n \in \N$,  the pair $(o,g_no)$ has an $(\alpha, L)$-contracting tail, say $\tau_n$.
	Let $q_n$ (\resp $q'_n$) be a projection of $c$ (\resp $c'$) onto $\tau_n$.
	We break the proof in several steps.
	
	\begin{clai}
		For every $n \in \N$, we have $\dist {q_n}{q'_n} \leq 4\alpha$.
	\end{clai}
	
	Fix $n \in \N$.
	Let $\gamma \colon \R_+ \to X$ be a gradient ray for $c$ starting at $q_n$.
	Let $t \in \R_+$.
	According to \autoref{res: proj cocycle pre lipschitz}, $q_n$ is a projection of $\gamma(t)$ on $\tau_n$.
	Assume that contrary to our claim $\dist {q'_n}{q_n} > 4\alpha$.
	It follows from \autoref{res: proj cocycle contracting set} applied with $c'$ that 
	\begin{equation*}
		c'(\gamma(t), q'_n) \geq \dist {q'_n}{\gamma(t)} - 10\alpha \geq - 10\alpha.
	\end{equation*}
	On the other hand, since $\gamma$ is a gradient ray for $c$, we know that
	\begin{equation*}
		c(\gamma(t), q'_n)  = c(\gamma(t), q_n) + c(q_n,q'_n)  = - t + c(q_n,q'_n).
	\end{equation*}
	These two estimates hold for every $t \geq T$, thus contradicting the fact that $\norm[\infty]{c - c'} < \infty$.
	It completes the proof of our first claim.
	
	\begin{clai}
	\label{clai: equiv rel restricted to contracting - common shadows}
		For every $n \in \N$, the cocycle $c'$ belongs to $\mathcal O_o(g_no, r+ 16\alpha)$.
		In particular, $c' \in \Lambda_{\rm ctg}(G,o, \alpha, r+16\alpha, L)$.
	\end{clai}

	Let $n \in \N$.
	Since $(o, g_no)$ has a contracting tail, there is a projection $p_n$ of $o$ on $\tau_n$ such that $\dist{p_n}{g_no} \geq L$.
	According to \autoref{res: wysiwyg shadow}, $\dist {q_n}{g_no} \leq r + 7\alpha$.
	Our previous claim, combined with the triangle inequality, gives
	\begin{equation*}
		\dist {p_n}{q'_n} 
		\geq \dist {p_n}{g_no} - \dist{q'_n}{q_n} - \dist{q_n}{g_no} 
		\geq L - (r + 11\alpha)   > 4\alpha.
	\end{equation*}
	Hence by \autoref{res: proj cocycle contracting set}, we have $\gro o{c'}{q'_n} \leq 5\alpha$, which combined with  (\ref{eqn: gromov product - lip}) yields
	\begin{equation*}
		\gro o{c'}{g_no} \leq \gro o{c'}{q'_n} + \dist {q'_n}{q_n} + \dist {q_n}{g_no}  \leq r + 16\alpha.
	\end{equation*}
	This completes the proof of our second claim.
	
	\begin{clai}
		$\norm[\infty]{c -c'} \leq 20\alpha$.
	\end{clai}
	
	Consider a closed ball $K$ of radius $R$ centered at $o$.
	If $n \in \N$ is sufficiently large, then $\dist o{g_no} > R + r + 29\alpha$.
	According to our previous claim $c$ and $c'$ both belong to $\mathcal O_o(g_no, r+ 16\alpha)$.
	It follows from \autoref{res: cocycle locally determined} that $\norm[K]{c-c'} \leq 20\alpha$.
	This estimate does not depend on $K$, whence the result.
\end{proof}

In a CAT(-1) settings, shadows are known to provide a basis of open neighborhoods for the points in the boundary at infinity.
This is no more the case in our context.
However we can still approximate \emph{saturated} subsets of the contracting limit set using shadows.
This is the purpose of the next lemma.

\begin{coro}
\label{res: approximation saturated contracting by shadows}
	Let $\alpha \in \R_+^*$ and $r, L \in \R_+$ with $L > r + 13\alpha$.
	Let $B$ be a saturated subset of $\Lambda_{\rm ctg}(G,o, \alpha, r, L)$ and $V$ an open subset of $\bar X$ containing $B$.
	Let $b \in B$.
	There exists $T \in \R_+$ such that for every $g \in \mathcal T(\alpha, L)$ with $\dist o{go} \geq T$, if $b$ belongs to $\mathcal O_o(go,r)$ then $\mathcal O_o(go,r) \subset V$.
\end{coro}

\begin{proof}
	Assume on the contrary that our statement fails.
	We can find a sequence of elements $g_n \in \mathcal T(\alpha, L)$ such that $\dist o{g_no}$ diverges to infinity, $b$ belongs to $\mathcal O_o(g_no, r)$ and $\mathcal O_o(g_no, r) \setminus V$ is non-empty.
	For every $n \in \N$, we write $c_n$ for a cocycle in $\mathcal O_o(g_no, r) \setminus V$.
	Up to passing to a subsequence, we can assume that $c_n$ converges to $c \in \bar X$.
	As $V$ is open, $c \notin V$.
	We claim that $\norm[\infty]{c - b} < \infty$.
	Let $K$ be a closed ball of radius $R$ centered at $o$.
	As $c_n$ converges to $c$, there is $N \in \N$, such that for every $n \geq N$, we have $\norm[K]{c_n - c} \leq 1$.
	By construction $b$ and $c_n$ both belong to $\mathcal O_o(g_no,r)$.
	According to \autoref{res: cocycle locally determined}, if $n$ is sufficiently large $\norm[K]{b - c_n} \leq 20\alpha$ so that $\norm[K]{b - c} \leq 20\alpha +1$.
	This inequality holds for every compact subset $K \subset X$,  which completes the proof of our claim.
	Since $B$ is saturated, $c \in B$.
	It contradicts the fact that $B \subset V$.
\end{proof}

\subsection{Measure of the contracting limit set}
\label{sec: measure contracting}

From now on, we assume that  $G$ is not virtually cyclic and contains a contracting element.
According to \autoref{res: shadow lemma}, there are $(\epsilon, r_0) \in \R_+^* \times \R_+$ such that any $G$-invariant, $\omega_G$-conformal density $\nu = (\nu_x)$ satisfies the Shadow Principle with parameters $(\epsilon, r_0)$.
The goal of this section is to prove that if the action of $G$ on $X$ is divergent, then any such density gives full measure to the contracting limit set.
The proof is an application of the Kochen-Stone theorem, which generalizes the second Borel-Cantelli Lemma.

\begin{prop}[Kochen-Stone \cite{Kochen:1964vy}]
\label{res: kochen-stone}
	Let $(\Omega,\mu)$ be a probability space.
	Let $(B_n)$ be a sequence of subsets of $\Omega$ such that 
	\begin{equation*}
		\sum_{n \in \N} \mu(B_n) = \infty.
	\end{equation*}
	Assume that there exists $C \in \R_+^*$ such that for every $N \in \N$,
	\begin{equation*}
		\sum_{n_1 = 0}^N \sum_{n_2 = 0}^N \mu\left(B_{n_1} \cap B_{n_2} \right) \leq C \left( \sum_{n = 0}^N \mu\left(B_n\right) \right)^2.
	\end{equation*}
 	Then 
	\begin{equation*}
		\mu\left( \bigcap_{N \in \N} \bigcup_{n \geq N} B_n\right) \geq \frac 1C.
	\end{equation*}
\end{prop}

Recall that for every $\alpha, r, \ell,  L \in \R_+$, the set $A_\ell(\alpha, r,L)$ is defined by
\begin{equation*}
	A_\ell(\alpha, r,L) = \bigcup_{g \in S(\ell, r) \cap  \mathcal T(\alpha, L)} \mathcal O_o(go, r).
\end{equation*}
These are the sets with which we will apply the Kochen-Stone theorem.
The aim of the next lemmas is to make sure that the hypotheses of \autoref{res: kochen-stone} are satisfied.

\begin{lemm}
\label{res: borel-cantelli - lower hyp}
	Let $a \in \R_+^*$.
	There are $\alpha, r_1, C_1\in \R_+^*$, such that for every $r \geq r_1$ and $L > r + 13\alpha$, the following holds.
	Let $\nu = (\nu_x)$ be a $G$-invariant, $\omega_G$-conformal density.
	For every integer $N \in \N$, 
	\begin{equation*}
		\sum_{\substack{ g \in G \\ \dist o{go} \leq Na - L}}e^{-\omega_G \dist o{go}}
		\leq C_1e^{2\omega_GL}
		\sum_{n = 0}^N \nu_o(B_n),
	\end{equation*}
	where $B_n = A_{na}(\alpha,r,L)$.
\end{lemm}

\begin{proof}
	Recall that $(\epsilon, r_0)$ are the Shadow Principle parameters which have been fixed once and for all at the beginning of \autoref{sec: measure contracting}.
	We denote by $\alpha, r_1, C \in \R_+$ the parameters given by \autoref{res: pre upper bounds} applied with $\omega_G$, $a$, and $(\epsilon, r_0)$.
	Fix $r \geq r_1$ and $L > r + 13\alpha$.
	Let $\nu = (\nu_x)$ be a $G$-invariant, $\omega_G$-conformal density.
	Recall that $(G, \nu)$ satisfies the Shadow Principle with parameters $(\epsilon, r_0)$.
	It follows from \autoref{res: pre upper bounds} that for every $\ell \in \R_+$,
	\begin{equation*}
		\sum_{g \in S(\ell, a)} e^{-\omega _G\dist o{go}}
		\leq C e^{2\omega L}\nu_o \left( A_{\ell+L}(\alpha, r,L) \right).
	\end{equation*}
	Summing this identity we get for every $N \in \N$, 
	\begin{equation*}
		\sum_{n = 0}^N \sum_{g \in S(na, a)} e^{-\omega _G\dist o{go}}
		\leq C e^{2\omega L} \sum_{n = 0}^N  \nu_o \left( A_{na+L}(\alpha, r,L) \right).
	\end{equation*}
	Note that $A_{na+L}(\alpha, r,L)$ is covered by $A_{(m-1)a}(\alpha, r,L)$ and $A_{ma}(\alpha, r,L)$ where $m = n + \lceil L/a \rceil$.
	Hence 
	\begin{equation*}
		\sum_{\substack{ g \in G \\ \dist o{go} \leq (N+1)a}} e^{-\omega _G\dist o{go}}
		\leq 2C e^{2\omega L} \sum_{n = 0}^{N +\lceil L/a \rceil} \nu_o \left( B_n \right),
	\end{equation*}
	whence the result.
\end{proof}

\begin{lemm}
\label{res: borel-cantelli - upper hyp}
	Let $a, \alpha, r \in \R_+^*$.
	There are $b, C_2 \in \R_+$ such that for every $L > r + 13\alpha$, the following holds.
	Let $\nu = (\nu_x)$ be a $G$-invariant, $\omega_G$-conformal density.
	For every $N \in \N$, we have
	\begin{equation*}
		\sum_{n_1 = 0}^N\sum_{n_2 = 0}^N \nu_o(B _{n_1}\cap B_{n_2}) 
		\leq C_2 \left( \sum_{\substack{ g \in G \\ \dist o{go} \leq Na+ b}}e^{-\omega_G \dist o{go}}\right)^2,
	\end{equation*}
	where $B_n = A_{na}(\alpha,r,L)$.
\end{lemm}

\begin{proof}
	Let $L > r + 13\alpha$.	
	Let $\nu = (\nu_x)$ be a $G$-invariant, $\omega_G$-conformal density.
	Let $N \in \N$.
	Observe first that 
	\begin{equation}
	\label{eqn: borel-cantelli - upper hyp}
		\begin{split}
			\sum_{n_1 = 0}^N\sum_{n_2 = 0}^N \nu_o(B _{n_1}\cap B_{n_2}) 
			& \leq 2 \sum_{n_1 = 0}^N\sum_{n_2 = n_1}^N \nu_o(B _{n_1}\cap B_{n_2}) \\
			& \leq 2  \sum_{n_1 = 0}^N\sum_{n_2=0}^{N-n_1} \nu_o(B _{n_1}\cap B_{n_1+n_2}).
		\end{split}
	\end{equation}
	Consider now $n_1, n_2 \in \N$ with $0 \leq n_1 \leq n_1 + n_2 \leq N$.
	By definition, we have
	\begin{equation*}
		B_{n_1} \cap B_{n_1+n_2}
		= \bigcup_{(g_1,g_2) \in U}
		\mathcal O_o(g_1o,r) \cap \mathcal O_o(g_1g_2o,r),
	\end{equation*}
	where $U$ is the set of pairs $(g_1,g_2) \in G$ with the following properties:
	\begin{labelledenu}[U]
		\item \label{enu: borel-cantelli - upper hyp - tail}
		$g_1, g_1g_2 \in \mathcal T(\alpha, L)$,
		\item \label{enu: borel-cantelli - upper hyp - dist}
		$g_1 \in S(n_1a,r)$ and $g_1g_2 \in S(n_1a + n_2a,r)$.
	\end{labelledenu}
	Let $(g_1,g_2) \in U$ for which $\mathcal O_o(g_1o,r) \cap \mathcal O_o(g_1g_2o,r)$ is non-empty.
	According to \autoref{res: intersection contracting shadows}, we have
	\begin{equation*}
		\dist o{g_2o} \leq \dist{g_1o}{g_1g_2o} 
		\leq \abs{\dist o{g_1g_2o} - \dist o{g_1o}} + 4r + 44\alpha.
	\end{equation*}
	In particular, 
	\begin{equation*}
		\dist o{g_1g_2o} \geq \dist o{g_1o} + \dist o{g_2o} - 4r - 44\alpha.
	\end{equation*}
	Moreover, combined with  \ref{enu: borel-cantelli - upper hyp - dist}, it shows that $g_2 \in S(n_2a, 6r + 44\alpha)$.
	Using \autoref{rem: shadow lemma}, we estimate the measure of each shadow:
	\begin{align*}
		\nu_o\left( O_o(g_1o,r) \cap \mathcal O_o(g_1g_2o,r)\right)
		& \leq \nu_o\left( \mathcal O_o(g_1g_2o,r)\right) \\
		& \leq e^{2\omega_Gr}e^{-\omega_G\dist o{g_1g_2o}} \\
		& \leq e^{\omega_G(6r+44\alpha)} e^{-\omega_G\left[\dist o{g_1o} + \dist o{g_2o}\right]}.
	\end{align*}
	Consequently 
	\begin{equation*}
		\nu_o\left(B_{n_1} \cap B_{n_1+n_2}\right)
		\leq e^{\omega_G(6r+44\alpha)} \sum_{\substack {g_1 \in S(n_1a,r), \\  g_2 \in S(n_2a, 6r + 44\alpha)}} e^{-\omega_G\left[\dist o{g_1o} + \dist o{g_2o}\right]}.
	\end{equation*}
	Note that for every $d \in \R_+$, an element $g \in G$ belongs to at most $\lceil 2d/a\rceil$ spheres of the form $S(na,d)$ when $n$ runs over $\N$.
	Summing the previous inequality over $n_1$ and $n_2$, and using (\ref{eqn: borel-cantelli - upper hyp}) we get 
	\begin{align*}
		\sum_{n_1 = 0}^N\sum_{n_2 = 0}^N \nu_o(B _{n_1}\cap B_{n_2}) 
		& \leq C_2\sum_{\substack {g_1,g_2 \in G \\ \dist o{g_1o},   \dist o{g_2o} \leq Na +b}} e^{-\omega_G\left[\dist o{g_1o} + \dist o{g_2o}\right]} \\
	& \leq C_2 \left( \sum_{\substack{ g \in G \\ \dist o{go} \leq Na + b}}e^{-\omega_G \dist o{go}}\right)^2,
	\end{align*}
	where 
	\begin{equation*}
		b = a + 6r + 44 \alpha,
		\quad \text{and} \quad
		C_2 = 2e^{\omega_G(6r+44\alpha)} \left\lceil  \frac {12r + 88\alpha}a\right\rceil ^2
	\end{equation*}
	only depends on $a$, $\alpha$ and $r$.
\end{proof}

\begin{prop}
\label{res: borel-cantelli}
	Assume that the action of $G$ on $X$ is divergent.
	There exists $\alpha, r \in \R_+$ with the following property.
	Let $\nu = (\nu_x)$ be a $G$-invariant, $\omega_G$-conformal density.
	For every $L > r + 13\alpha$, we have
	\begin{equation*}
		\nu_0 \left( \Lambda_{\rm ctg}(G, o, \alpha, r, L) \right) > 0.
	\end{equation*}
\end{prop}

\begin{proof}
	Fix $a \in \R_+^*$.
	Let $\alpha, r_1, C_1$ be the parameters given by \autoref{res: borel-cantelli - lower hyp}.
	Fix $r \geq r_1$.
	Let $b, C_2$ be the parameters given by \autoref{res: borel-cantelli - upper hyp} applied with $a$, $\alpha$ and $r$.
	Choose $L > r+13\alpha$ and set $a' = \max\{a, b, L\}$.
	We write $C$ for the constant given by \autoref{res: pre upper bounds bis} applied with $\omega_G$, $a'$ and $(\epsilon, r_0)$.
	
	Let $\nu = (\nu_x)$ be a $G$-invariant, $\omega_G$-conformal density.
	For simplicity; for every $n \in \N$,  we let $B_n  = A_{na}(r,\alpha, L)$.
	Since the action of $G$ is divergent, \autoref{res: borel-cantelli - lower hyp} tells us that 
	\begin{equation*}
		\sum_{n \in \N} \nu_o(B_n) = \infty.
	\end{equation*}
	Recall that $(G,\nu)$ satisfies the Shadow Principle with parameters $(\epsilon, r_0)$.
	By \autoref{res: pre upper bounds bis}
	\begin{equation*}
		\sum_{g \in S(\ell, a')} e^{-\omega_G\dist o{go}} \leq C, \quad \forall \ell \in \R_+.
	\end{equation*}
	Since the Poincaré series of $G$ diverges at $s = \omega_G$, we deduce from  Lemmas~\ref{res: borel-cantelli - lower hyp} and  \ref{res: borel-cantelli - upper hyp} that there exists $C' \in \R_+^*$, such that for every sufficiently large $N \in \N$,
	\begin{equation*}
		\sum_{n_1 = 0}^N \sum_{n_2 = 0}^N \nu_o\left(B_{n_1} \cap B_{n_2} \right) 
		\leq C' \left( \sum_{n = 0}^N \nu_o\left(B_n\right) \right)^2.
	\end{equation*}
	Applying \autoref{res: kochen-stone}, we observe that 
	\begin{equation*}
		\nu_0 \left( \Lambda_{\rm ctg}(G, o, \alpha, r, L) \right) 
		= \nu_o\left( \bigcap_{N \in \N} \bigcup_{n \geq N} B_n\right) \geq \frac 1{C'} .\qedhere
	\end{equation*}
\end{proof}

\begin{coro}
\label{res: ctg limit set with full measure}
	Assume that the action of $G$ on $X$ is divergent.
	There exist $\alpha, r \in \R_+^*$ with the following property.
	If $\nu = (\nu_x)$ is a $G$-invariant, $\omega_G$-conformal density, then $\nu_o$ gives full measure to 
	\begin{equation*}
		\Lambda_{\rm ctg} (G, o, \alpha, r) = \bigcap_{ L \in \R_+} \Lambda_{\rm ctg} (G, o, \alpha, r, L). 
	\end{equation*}
	In particular, $\nu_o$ gives full measure to the contracting limit set $\Lambda_{\rm ctg} (G)$ and thus to the radial limit set $\Lambda_{\rm rad}(G)$.
\end{coro}

\begin{rema}
	The corollary is a reminiscence of the fact that contracting elements are ``generic'' in $G$, see Yang \cite{Yang:2020ub}.
\end{rema}

\begin{proof}
	Let $\alpha, r$ be the parameters given by \autoref{res: borel-cantelli}.
	Let $\nu = (\nu_x)$ be a $G$-invariant, $\omega_G$-conformal density.
	Let $L > r + 13\alpha$.
	For simplicity we let 
	\begin{equation*}
		B = G \Lambda_{\rm ctg} (G, o, \alpha, r, L).
	\end{equation*}
	We claim that $\nu_o$ gives full measure to $B$.
	Assume on the contrary that the set $A = \bar X \setminus B$ has positive measure.
	We define a new density $\nu^*  = (\nu^*_x)$ by 
	\begin{equation*}
		\nu^*_x = \frac 1{\nu_o(A)}  \mathbb 1_A \nu_x.
	\end{equation*}
	Note that $\nu^*$ is $\omega_G$-conformal.
	By construction $B$, and thus $A$, is $G$-invariant.
	Hence $\nu^*$ is also $G$-invariant.
	It follows from \autoref{res: borel-cantelli}, that the $\nu^*_o$ gives positive measure to $B$, a contradiction.
	According to \autoref{res: contracting limit set invariant}, the set $B$ is contained in $\Lambda_{\rm ctg} (G, o, \alpha, r + 12\alpha, L- \alpha)$.
	Hence the latter has full measure as well.
	This facts holds for every $L > r + 13\alpha$.
	Thus
	\begin{equation*}
		\nu_0 \left(\bigcap_{L \in \R_+} \Lambda_{\rm ctg}(G, o,\alpha, r+12\alpha, L) \right) = 1,
	\end{equation*}
	i.e. 
	\begin{equation*}
		\nu_0 \left( \Lambda_{\rm ctg}(G, o, \alpha, r + 12\alpha) \right) = 1. \qedhere
	\end{equation*}
\end{proof}

\subsection{Passing to the reduced horoboundary}

We now study the restriction to the reduced horoboundary of invariant conformal densities.
We still assume that $G$ is not virtually cyclic and contains a contracting element.
The Shadow Principle parameters $(\epsilon, r_0) \in \R_+^* \times \R_+$ are as in the previous section.

\begin{prop}
\label{res: pre-unicity ps for div group}
	Assume that the action of $G$ on $X$ is divergent.
	There is $C \in \R_+^*$ with the following property.
	Let $\nu = (\nu_x)$ and $\nu' = (\nu'_x)$ be two $G$-invariant, $\omega_G$-conformal densities.
	Denote by $\mu = (\mu_x)$ and $\mu' = (\mu'_x)$ their respective restrictions to the reduced horocompactification $(\bar X, \mathfrak R)$.
	Then $\mu_o \leq C\mu'_o$.
\end{prop}

\begin{proof}
	Let $\alpha, r \in \R^*_+$ be the parameters given by \autoref{res: ctg limit set with full measure}.
	Without loss of generality we can assume that $r \geq r_0$.
	Set $r' = r + 16\alpha$.
	For simplicity we write $\Lambda$ and $\Lambda'$ for the set  $\Lambda_{\rm ctg} (G, o, \alpha, r)$ and $\Lambda_{\rm ctg} (G, o, \alpha, r')$ respectively.
	Let $\nu = (\nu_x)$ and $\nu' = (\nu'_x)$ be as in the statement.
	Let $B \subset \partial X$ be a saturated Borel subset.
	We want to compare $\mu_o(B) = \nu_o(B)$ and $\mu'_o(B) = \nu'_o(B)$.
	According to \autoref{res: ctg limit set with full measure}, both $\nu_o$ and $\nu'_o$ give full measures to $\Lambda$.
	In addition, the saturated set $(B \cap \Lambda)^+$ is contained in $\Lambda'$, see \autoref{res: equiv rel restricted to contracting}
	(note that we do not claim that $(B \cap \Lambda)^+$ is measurable; we will use it only as an auxiliary tool to describe sets and will never compute its measure).
	Let $V$ be an open subset of $\bar X$ containing $B$.
	Choose $L  > r' + 16\alpha$.
	Using \autoref{res: approximation saturated contracting by shadows} with $(B \cap \Lambda)^+$, we build a subset $S \subset \mathcal T(\alpha, L)$ such that 
	\begin{equation*}
		B \cap \Lambda \subset (B \cap \Lambda)^+ \subset \bigcup_{g \in S}\mathcal O_o(go,r') \subset V.
	\end{equation*}
	According to \autoref{res: vitali}, there is a subset $S^*$ of $S$ such that 
	\begin{itemize}
		\item the collection $\left( \mathcal O_o(go,r')\right)_{g \in S^*}$ is pairwise disjoint, and
		\item $B \cap \Lambda$ is covered by $\left( \mathcal O_o(go,r'+42\alpha)\right)_{g \in S^*}$
	\end{itemize}
	Recall that $(G,\nu')$ satisfies the Shadow Principle.
	Since $\Lambda$ has full measure, we get
	\begin{align*}
		\nu_o(B)
		\leq \nu_o(B \cap \Lambda) 
		\leq \sum_{g \in S^*} \nu_o(\mathcal O_o(go,r'+42\alpha))
		& \leq e^{2\omega_G(r'+42\alpha)} \sum_{g \in S^*} e^{-\omega_G \dist o{go}} \\
		& \leq C \sum_{g \in S^*} \nu'_o(\mathcal O_o(go,r)),
	\end{align*}
	where $C = e^{2\omega_G(r'+42\alpha)}/\epsilon$ does not depend on $\nu$ and $\nu'$.
	Hence $\nu_o(B) \leq C \nu'_o(V)$.
	This inequality holds for every open subset $V$ containing $B$, thus $\nu_o(B) \leq C\nu'_o(B)$.
\end{proof}

\begin{prop}
\label{res: quasi-conf + ergo}
	Let $\nu = (\nu_x)$ be a $G$-invariant, $\omega_G$-conformal density and $\mu = (\mu_x)$ its restriction to the reduced horocompactification $(\bar X, \mathfrak R)$.
	Assume that the action of $G$ on $X$ is divergent.
	Then
	\begin{enumerate}
		\item \label{enu: quasi-conf + ergo - ergo}
		 $\mu_o$ is supported on the contracting limit set
		\item \label{enu: quasi-conf + ergo - ergo}
		$\mu_o$ is ergodic.
		\item \label{enu: quasi-conf + ergo - non-atomic}
		$\mu_o$ is non-atomic, in particular, no equivalence class for $\sim$ has positive measure.
		\item \label{enu: quasi-conf + ergo - conf}
		$\mu$ is a $G$-invariant, $\omega_G$-quasi-conformal density.
		\item \label{enu: quasi-conf + ergo - uniqueness}
		$\mu$ is almost-unique, that is there is $C \in \R_+^*$ such that  if $\mu' = (\mu'_x)$ is the restriction to the reduced horocompactification of another $G$-invariant, $\omega_G$-conformal density, then for every $x \in X$, we have $\mu'_x \leq C \mu_x$.
	\end{enumerate}
\end{prop}

\begin{proof}
	Let $\alpha, r \in \R^*_+$ be the parameters given by \autoref{res: ctg limit set with full measure} and $C \in \R_+$  the one given by \autoref{res: pre-unicity ps for div group}.
	Let $\nu = (\nu_x)$ be a $G$-invariant, $\omega_G$-conformal density and $\mu = (\mu_x)$ its restriction to the reduced horocompactifiction $(\bar X, \mathfrak R)$.
	It follows from \autoref{res: ctg limit set with full measure} that $\mu_o$ gives full measure to the contracting limit set.
	Let us prove the ergodicity of $\mu_o$.
	Let $B$ be a $G$-invariant saturated Borel subset such that $\mu_o(B) > 0$.
	Consider the density $\nu^* = (\nu^*_x)$ defined by
	\begin{equation*}
		\nu^*_x = \frac 1{\nu_o(B)} \mathbb 1_B \nu_x.
	\end{equation*}
	Since $B$ is $G$-invariant, $\nu^*$ is a $G$-invariant, $\omega_G$-conformal density.
	Denote by $\mu^* = (\mu^*_x)$ its restriction to the reduced horocompactifiction.
	It follows from \autoref{res: pre-unicity ps for div group} that $\mu_o \leq C \mu^*_o$.
	Consequently $\mu_o(\bar X \setminus B) = 0$, that is $\mu_o(B) = 1$.
	
	We now focus on non-atomicity.
	Assume on the contrary that $B \in \mathfrak R$ is an atom of $\mu_o$.
	For every $g \in G$, we write $\mathcal O^+_o(go,r)$ for the saturation of the shadow $\mathcal O_o(go,r)$.
	Note that $\mathcal O^+_o(go,r)$ is mesurable (\autoref{res: saturation of closed subset}).
	Consequently the measure of $B \cap \mathcal O^+_o(go,r)$ is either zero or equals $\mu_o(B)$.
	By \autoref{res: ctg limit set with full measure} $\mu_o$ only charges the set $\Lambda = \Lambda_{\rm ctg}(G, o, \alpha,r)$, hence $\Lambda_{\rm rad}(G,o,r)$.
	Consequently, for every $n \in \N$, there is $g_n \in G$, with $\dist o{g_no} \geq n$ such that
	\begin{equation*}
		\mu_o \left(\mathcal O^+_o(g_no,r) \right) 
		\geq \mu_o \left( B \cap \mathcal O^+_o(g_no,r) \right) 
		\geq \mu_o(B) 
	\end{equation*}
	By \autoref{res: equiv rel restricted to contracting}, $\mathcal O^+_o(g_no,r) \cap \Lambda$ is contained in $\mathcal O_o(g_no,r+20\alpha)$.
	Since $\nu_o$ gives full measure to $\Lambda$, we get
	\begin{equation*}
		0 
		< \mu_o(B) 
		\leq \nu_o \left(\mathcal O^+_o(g_no,r) \right) 
		\leq \nu_o \left(\mathcal O^+_o(g_no,r) \cap \Lambda\right)
		\leq \nu_o \left(\mathcal O_o(g_no,r+20\alpha)\right).
	\end{equation*}
	Recall that $\dist o{g_no}$ diverges to infinity.
	Hence the above inequality contradicts the Shadow Lemma (\autoref{res: shadow lemma}).
	
	Let us prove \ref{enu: quasi-conf + ergo - conf}.
	It follows from the construction that $\mu$ is $G$-invariant, $\norm{\mu_o} = 1$, and $\mu_x \ll \mu_y$, for every $x,y \in X$.
	Hence we are only left to prove that $\mu$ is quasi-conformal.
	Let $x,y \in X$.
	We define two auxiliary maps as follows.
	\begin{equation*}
		\begin{array}{ccccccc}
			\bar X & \to & \R & \quad \text{and} \quad\ & \bar X & \to & \R \\
			c & \mapsto &\displaystyle \inf_{c' \sim c} c'(x,y) && c & \mapsto &\displaystyle \sup_{c' \sim c} c'(x,y) 
		\end{array}
	\end{equation*}
	We denote them by $c \mapsto \beta^-_c(x,y)$ and $c \mapsto \beta^+_c(x,y)$ respectively.
	As $\bar X$ is separable, one checks that these maps are $\mathfrak R$-measurable.
	Let $B$ be a saturated Borel subset.
 	Using the conformality of $\nu$ we have
	\begin{equation*} 
		\nu_x(B) 
		\leq \int \mathbb 1_B(c) e^{-\omega_Gc(x,y)} d\nu_y(c) 
		\leq \int \mathbb 1_B(c) e^{-\omega_G\beta^-_c(x,y)} d\nu_y(c). 
	\end{equation*}
	Since $B$ is saturated and $c \mapsto \beta^-_c(x,y)$ is $\mathfrak R$-measurable, we get
	\begin{equation*} 
		\mu_x(B) 
		\leq \int \mathbb 1_B(c) e^{-\omega_G\beta^-_c(x,y)} d\mu_y(c).
	\end{equation*}
	This inequality holds for every $B \in \mathfrak R$.
	Hence 
	\begin{equation*}
		 \frac {d\mu_x}{d\mu_y} (c) \leq e^{-\omega_G \beta^-_c(x,y)}, \quad \mu\text{-a.e.}
	\end{equation*}
	In the same way, we obtain a lower bound for the Radon-Nikodym derivative with $\beta^+_c(x,y)$ in place of $\beta^-_c(x,y)$.
	By \autoref{res: equiv rel restricted to contracting}, for $\mu$-almost every $c \in \bar X$, we have
	\begin{equation*}
		c(x,y) -20\alpha \leq \beta^-_c(x,y) \leq \beta^+_c(x,y) \leq c(x,y) + 20\alpha.
	\end{equation*}
	Hence $\mu$ is quasi-conformal.
	Point~\ref{enu: quasi-conf + ergo - uniqueness} now follows from \autoref{res: pre-unicity ps for div group} and the quasi-conformality.
\end{proof}

\subsection{More applications}
\label{sec: more applications}

\begin{prop}
\label{res: strict half inequality}
	Assume that the action of $G$ on $X$ is divergent.
	For every infinite normal subgroup of $G$ we have 
	\begin{equation*}
		\omega_N > \frac 12 \omega_G.
	\end{equation*}
\end{prop}

\begin{proof}
	Let $Q = G/N$ and $\omega_Q$ be the growth rate of $Q$ on $X/N$.
	According to \autoref{res: improved lower bound growth normal sbgp} we have
	\begin{equation*}
		\omega_N + \frac 12 \omega_Q \geq \omega_G.
	\end{equation*}
	Since the map $X \to X/N$ is $1$-Lipschitz, $\omega_Q \leq \omega_G$.
	Hence
	\begin{equation*}
		\omega_N \geq \frac 12 \omega_G.
	\end{equation*}
	Suppose now that, contrary to our claim, $\omega_G = 2 \omega_N$.
	We choose 
	\begin{itemize}
		\item a $G$-invariant, $\omega_G$-conformal density $\nu = (\nu_x)$ and
		\item an $N$-invariant, $\omega_N$-conformal density $\nu'  = (\nu'_x)$ such that the action of $N$ on $(\bar X, \mathfrak B, \nu'_o)$ is ergodic.
	\end{itemize}
	We write $\mu$ and $\mu'$ for their respective restrictions to the reduced horocompactification $(\bar X, \mathfrak R)$.
	In particular, the action of $N$ on $(\bar X, \mathfrak R, \mu'_o)$ is ergodic.
	We claim that $\mu_0$ is absolutely continuous with respect to $\mu'_0$.
	According to \autoref{res: shadow lemma}, $(G, \nu')$ satisfies the Shadow Lemma for some parameters $(\epsilon, r_0) \in \R_+^* \times \R_+$.
	By \autoref{res: ctg limit set with full measure}, there exists $\alpha,r \in \R^*_+$ such that $\nu_o$ gives full measure to $\Lambda = \Lambda_{\rm ctg} (G, o, \alpha,r)$.	
	Without loss of generality, we can assume that $r \geq r_0$.
	For simplicity we set $r' = r + 16\alpha$ and write $\Lambda'$ for $\Lambda_{\rm ctg} (G, o, \alpha,r)$.	

	Let $B$ be a saturated subset contained in $\Lambda_{\rm ctg} (G, o,\alpha,r)$.
	Let $V$ be an open set containing $B$.
	Observe that $(B \cap \Lambda)^+$ is contained in $\Lambda'$ (\autoref{res: equiv rel restricted to contracting})
	Fix $L > r' +16\alpha$.
	Using \autoref{res: approximation saturated contracting by shadows} with $(B \cap \Lambda)^+$, we build a subset $S \subset \mathcal T(\alpha, L)$ such that 
	\begin{equation*}
		B \cap \Lambda \subset (B \cap \Lambda)^+ \subset \bigcup_{g \in S}\mathcal O_o(go,r') \subset V.
	\end{equation*}
	According to \autoref{res: vitali}, there is a subset $S^*$ of $S$ such that 
	\begin{itemize}
		\item the collection $\left( \mathcal O_o(go,r')\right)_{g \in S^*}$ is pairwise disjoint, and
		\item $B\cap \Lambda$ is covered by $\left( \mathcal O_o(go,r'+42\alpha)\right)_{g \in S^*}$
	\end{itemize}
	Since $\nu$ gives full measure to $\Lambda$, we have $\nu_o(B) = \nu_o(B \cap \Lambda)$.
	Using \autoref{rem: shadow lemma} with the density $\nu$ we get
	\begin{equation*}
		\nu_o(B) 
		\leq \sum_{g \in S^*} \nu_o(\mathcal O_o(go,r'+42\alpha))
		\leq e^{2\omega_G(r'+42\alpha)} \sum_{g \in S^*} e^{-\omega_G \dist o{go}}.
	\end{equation*}
	Recall that for every $g \in G$, we have $\norm{\nu'_{go}} \geq e^{-\omega_N \dist o{go}}$.
	Since $\omega_G = 2 \omega_N$, we obtain
	\begin{equation*}
		\nu_o(B) 
		\leq e^{2\omega_G(r'+42\alpha)} \sum_{g \in S^*} \norm{\nu'_{go}}  e^{-\omega_N \dist o{go}}.
	\end{equation*}
	Using now the Shadow Principle with the density $\nu'$ we obtain
	\begin{equation*}
		\nu_o(B) 
		\leq C \sum_{g \in S^*} \nu'_o(\mathcal O_o(go,r'))
		\leq C \nu'_o(V),
		\quad \text{where} \quad
		C =  \frac 1\epsilon e^{2\omega_G(r'+42\alpha)}
	\end{equation*}
	does not depend on $B$.
	This inequality holds for every open subset $V$ containing $B$, hence $\nu_o(B) \leq C\nu'_o(B)$, i.e. $\mu_o(B) \leq C \mu'_o(B)$.
	This completes the proof of our claim.
	
	Denote by $f$ the Radon-Nikodym derivative $f = d\mu_o / d \mu'_o$.
	Both $\mu$ and $\mu'$ are $N$-invariant.
	Hence the set $A = \set{c \in \bar X}{f(c) > 0}$ is $N$-invariant.
	Note that $\mu'_o(A) > 0$.
	Indeed otherwise $\mu_o$ would be the zero measure.
 	Since the action of $N$ on $(\bar X, \mathfrak R, \mu'_o)$ is ergodic, we get $\mu'_o(A) = 1$.
	Hence $\mu_o$ and $\mu'_o$ are in the same class of measures.
	Since $\mu_o$ is $G$-invariant, $\mu'_o$ is $G$-quasi-invariant.
	We assumed that $\mu'_o$ is ergodic for the action of $N$.
	It follows from \autoref{res: ergodic quasi-inv induces almost-fixed} that $\mu'$ is almost fixed by $G$.
	Thus $\omega_N \geq \omega_G$ by \autoref{res: dimension almost-fixed}.
	This contradicts our assumption and completes the proof.
\end{proof}

\begin{prop}
\label{res: shadow lemma almost-fixed reduced}
	Let $H \subset G$ be a subgroup which is not virtually cyclic and contains a contracting element.
	Let $\nu = (\nu_x)$ be an $H$-invariant, $\omega_H$-conformal density and $\mu = (\mu_x)$ its restriction to the reduced horocompactification $(\bar X, \mathfrak R)$.
	Assume that the action of $H$ is divergent.
	If $\mu$ is almost fixed by $G$, then $(G, \nu)$ satisfies the Shadow Principle.
\end{prop}

\begin{proof}
	According to \autoref{res: ctg limit set with full measure}, there are $\alpha, r_0 \in \R^*_+$ such that $\nu$ gives full measure to $\Lambda_{\rm ctg}(H,o,\alpha,r_0)$.
	Proceeding as in the proof of \autoref{res: shadow lemma}, we show that for every $g \in G$ and $r \in \R_+$, 
	\begin{equation}
	\label{eqn: shadow lemma almost-fixed reduced}
		\nu_o\left(\mathcal O_o(go,r)\right)
		\geq \norm{\nu_{go}} e^{-\omega_H \dist o{go}}\nu^g_o\left(\mathcal O_{g^{-1}o}(o,r)\right).
	\end{equation}
	Choose now $r \geq r_0$ and $g \in G$.
	We denote by $\mathcal O^+_{g^{-1}o}(o,r)$ the saturation of the shadow $\mathcal O_{g^{-1}o}(o,r)$, which is mesurable by \autoref{res: saturation of closed subset}.
	According to \autoref{res: equiv rel restricted to contracting}, 
	\begin{equation*}
		\mathcal O^+_{g^{-1}o}(o,r) \cap \Lambda_{\rm ctg}(H,o,\alpha, r) \subset \mathcal O_{g^{-1}o}(o,r+20\alpha).
	\end{equation*}
	Recall that $\nu$ gives full measure to $\Lambda_{\rm ctg}(H,o,\alpha, r_0)$, thus to $\Lambda_{\rm ctg}(H,o,\alpha, r)$ as well.
	Since $\mu$ is almost fixed by $G$ we have
	\begin{align*}
		\nu^g_o\left(\mathcal O_{g^{-1}o}(o,r+20\alpha)\right)
		\geq \mu^g_o\left(\mathcal O^+_{g^{-1}o}(o,r) \right)
		& \geq \epsilon \mu_o\left(\mathcal O^+_{g^{-1}o}(o,r) \right) \\
		& \geq \epsilon \nu_o\left(\mathcal O_{g^{-1}o}(o,r) \right),
	\end{align*}
	where $\epsilon \in \R_+^*$ does not depend on $g$ and $r$.
	Combined with (\ref{eqn: shadow lemma almost-fixed reduced}) it shows that for every $r \geq r_0$, for every $g \in G$, we have 
	\begin{equation*}
		 \nu_o\left(\mathcal O_o(go,r +20\alpha)\right)
		\geq \epsilon \norm{\nu_{go}} e^{-\omega_H \dist o{go}}\nu_o\left(\mathcal O_{g^{-1}o}(o,r)\right).
	\end{equation*}
	According to our assumption, $H$ is not virtually cyclic and contains a contracting element.
	The conclusion now follows from \autoref{res: pre-shadow lemma} applied with the group $H$ and the set $\mathcal D_0 = \{ \nu\}$..
\end{proof}

\begin{theo}
\label{res: div normal sbgp}
	Let $H$ be a commensurated subgroup of $G$.
	If the action of $H$ on $X$ is divergent, then the following holds.
	\begin{enumerate}
		\item \label{res: div normal sbgp - qinv}
		Any $H$-invariant, $\omega_H$-conformal density is $G$-almost invariant when restricted to the reduced horocompactification $(\bar  X, \mathfrak R)$.
		\item \label{res: div normal sbgp - exp}
		$\omega_H = \omega_G$.
		\item \label{res: div normal sbgp - div}
		The action of $G$ on $X$ is divergent.
	\end{enumerate}
\end{theo}

\begin{proof}
	Let $\nu = (\nu_x)$ be an $H$-invariant, $\omega_H$-conformal density.
	We denote by $\mu = (\mu_x)$ its restriction to the reduced horocompactification $(\bar X, \mathfrak R)$.
	Let $g \in G$.
	By definition of commensurability the intersection $H_0 = H^g \cap H$ has finite index in $H$.
	In particular, $H_0$ is divergent and $\omega_{H_0} = \omega_H$.
	Recall that $\nu^g$ is the image of $\nu$ under the right action of $g \in G$.
	It is an $H^g$-invariant, $\omega_H$-conformal density, thus an $H_0$-invariant, $\omega_{H_0}$-conformal density.
	Similarly $\nu$ is an $H_0$-invariant, $\omega_{H_0}$-conformal density.
	Since $H_0$ is divergent, \autoref{res: quasi-conf + ergo} tells us that
	\begin{itemize}
		\item there is $C \in \R_+$ such that $\mu^g \leq C \mu$,
		\item the action of $H^g \cap H$ on $(\bar X, \mathfrak R, \mu_o)$ is ergodic.
	\end{itemize}
	Note that $C$ depends a priori on $H_o$ and thus on $g$.
	Nevertheless, it still proves that $\mu$ is $G$-quasi-invariant.
	Consequently $\mu$ is $C_0$-almost fixed by $G$, for some $C_0 \in \R_+^*$ (\autoref{res: ergodic quasi-inv induces almost-fixed}).
	We deduce from \autoref{res: shadow lemma almost-fixed reduced} that $(G, \nu)$ satisfies the Shadow Principle.
	Point \ref{res: div normal sbgp - exp} now follows from \autoref{res: dimension almost-fixed}.
	Recall that $\mathcal P_H(s) \leq \mathcal P_G(s)$, for every $s \in \R_+$.
	Since the action of $H$ on $X$ is divergent, $\mathcal P_G(s)$ diverges at $s = \omega_H = \omega_G$.
	Hence the action of $G$ on $X$ is divergent as well, which proves \ref{res: div normal sbgp - div}.
	
	We already know that $\mu$ is almost-fixed by $G$, so that the map $\chi \colon G \to \R$ sending $g$ to $\ln \norm{\mu_{go}}$ is a quasi-morphism (\autoref{res: almost-fixed gives quasi-morphism}).
	We are left to prove that $\mu$ is actually $G$-almost invariant, i.e. $\chi$ is bounded.
	Recall that $(G, \nu)$ satisfies the Shadow Principle.
	It follows from \autoref{res: series roblin} that the critical exponent of the series
	\begin{equation*}
		\sum_{g \in G} e^{\chi(g)}e^{-s\dist o{go}} = \sum_{g \in G} \norm{\nu_{go}} e^{-s\dist o{go}}
	\end{equation*}
	is exactly $\omega_H$.
	Hence $\omega_{-\chi} = \omega_\chi = \omega_H$.
	Note also that, since $\nu$ is $H$-invariant, $\chi(hg) = \chi(g)$ for every $h \in H$ and $g \in G$.
	Using \autoref{res: existence twisted ps}, with the quasi-morphism $- \chi$, we produce an $H$-invariant, $\omega_H$-conformal density $\nu^* = (\nu^*_x)$ satisfying the following additional property:
	there is $C_1 \in \R_+^*$ such that for every $g \in G$, for every $x \in X$, we have
	\begin{equation*}
		\frac 1{C_1} \nu_x \leq e^{\chi(g)} {g^{-1}}_\ast \nu_{gx} \leq C_1 \nu_x.
	\end{equation*}
	Denote by $\mu^*$ its restriction to the reduced horocompactification $(\bar X, \mathfrak R)$.
	According to \autoref{res: quasi-conf + ergo}\ref{enu: quasi-conf + ergo - uniqueness}, there is $C_2 \in \R_+$ such that $\mu \leq C_2 \mu^*$.
	Recall that $\mu$ is $C_0$-almost fixed by $G$.
	Consequently for every $g \in G$, we have
	\begin{equation*}
		e^{\chi(g)}\mu_o 
		\leq C_0 \left( {g^{-1}}_\ast \mu_{go}\right)
		 \leq C_0C_2 \left( {g^{-1}}_\ast\mu^*_{go} \right)
		 \leq C_0C_1C_2  \left(e^{-\chi(g)}\mu^*_o\right).
	\end{equation*}
	Since $\mu_o$ and $\mu^*_o$ are probability measures, $\chi$ is bounded, whence the result.
\end{proof}

\begin{rema}
\label{rem: commensurated subroups}
	As we noticed in the introduction, every finite index and every normal subgroup of $G$ is commensurated.
	More generally, consider a subgroup $H$ of $G$ and $N$ a normal subgroup of $G$.
	If $H$ and $N$ are commensurable (i.e. $H\cap N$ has finite index in both $H$ and $N$) then $H$ is commensurated.
	However there are plenty of other examples.
	Here is a construction suggested by Uri Bader \cite{Bader:2018}.
	Consider the free group $\free 2$ and morphism $\phi \colon \free 2 \to M$ where $M$ is a topological group.
	Let $K$ be an open compact subgroup of $M$.
	Then $H = \phi^{-1}(K)$ is commensurated.
	Now if $\phi$ has dense image, then $H$ is commensurable with a normal subgroup of $\free 2$ if and only if $K$ is commensurable with a normal subgroup of $M$.
	Consider now for instance a prime $p$ and a morphism $\phi \colon \free 2 \to  {\rm PSL}_2(\Q_p)$ with dense image.
	Then $\phi^{-1}({\rm PSL}_2(\Z_p))$ provides an example of a commensurated subgroup of $\free 2$ that is not commensurable with a normal subgroup of $\free 2$.
\end{rema}

\appendix

\section{Strongly positively recurrent actions}
\label{sec: spr}

\subsection{Definition}
Let $G$ be a group acting properly, by isometries on a proper, geodesic, metric space $X$.
Given a compact subset $K \subset X$, we define a subset $G_K \subset G$ as follows:
an element $g \in G$ belongs to $G_K$, if there exist $x,y \in K$ and a geodesic $\gamma$ joining $x$ to $gy$ such that the intersection $\gamma \cap G K$ is contained in $K \cup gK$.
Although $G_K$ is not a subgroup of $G$, its exponential growth rate $\omega(G_K, X)$ is defined in the same way as for the one of $G$.

\begin{defi}
\label{res: entropy at infinity}
	The \emph{entropy at infinity} of the action of $G$ on $X$ is
	\begin{equation*}
		\omega_\infty(G, X) = \inf_K \omega(G_K, X),
	\end{equation*}
	where $K$ runs over all compact subsets of $X$.
 	The action of $G$ on $X$ is \emph{strongly positively recurrent} (or \emph{statistically convex co-compact}) if $\omega_\infty(G, X) < \omega(G,X)$.
\end{defi}

We refer the reader to \cite{Schapira:2021ti,Coulon:2018va} for examples of strongly positively recurrent actions in the context of hyperbolic geometry.
Arzhantseva, Cashen and Tao \cite[Section~10]{Arzhantseva:2015cl} also observed that the work of Eskin, Mirzakani, and Rafi \cite[Theorem~1.7]{Eskin:2019uw} implies that the action of the mapping class group on the Teichmüller space endowed with the Teichmüller metric is strongly positively recurrent.

\subsection{Divergence}

\begin{prop}
\label{res: spr gives div}
	If the action of $G$ on $X$ is strongly positively recurrent, then it is divergent.
\end{prop}

The statement was proved by Yang \cite{Yang:2019wa}.
We give here an alternative approach in the spirit of Schapira-Tapie \cite{Schapira:2021ti}.
The idea is to build a $G$-invariant, $\omega_G$-conformal density which gives positive measure to the radial limit set.
Indeed, according to \autoref{res: series roblin}, this will imply that the action of $G$ on $X$ is divergent.
As we explained in \autoref{rem: series roblin} this part of \autoref{res: series roblin} does not require that $G$ contains a contracting element.

First, we give a description of the complement of the radial limit set.
To that end we introduce some notations.
Given a compact subset $K \subset X$ and $\epsilon \in \R_+^*$, we denote by $A_{K, \epsilon}$ the set of all cocycles $c \in \partial X$ with the following property: 
there is a point $x \in K$ such that for every $\epsilon$-quasi-gradient ray $\gamma \colon \R_+ \to X$ for $c$ starting at $x$, for every $u \in G$, if the intersection $\gamma \cap uK$ is non-empty, then $d(K, uK) \leq 1$.

\begin{lemm}
\label{res: inclusion radial limit set}
	The radial set of $G$ satisfies the following inclusion
	\begin{equation*}
		\partial X \setminus \Lambda_{\rm rad} (G)
		\subset \bigcap_{K \subset X} G \left( \bigcup_{\epsilon > 0} A_{K, \epsilon}\right),
	\end{equation*}
	where $K$ runs over all compact subsets of $X$.
\end{lemm}

\begin{proof}
	The proof is by contraposition.
	Consider a cocycle $c \in \partial X$ that is not in the set
	\begin{equation*}
		\bigcap_{K \subset X} G \left( \bigcup_{\epsilon > 0} A_{K, \epsilon}\right).
	\end{equation*}
	There is a compact subset $K \subset X$ such that for every $g \in G$ and $\epsilon > 0$, the cocycle $c$ does not belong to $gA_{K, \epsilon}$.
	Fix $\epsilon \in (0, 1)$ and $x_0 \in K$.
	In addition we let $g_0 = 1$.
	We are going to build by induction, a sequence of points $x_1, x_2\dots$ in $X$, a sequence of elements $g_1, g_2,\dots$ in $G$, and a sequence of rays $\gamma_1, \gamma_2,\dots$, such that for every $i \in \N\setminus\{0\}$ the following holds.
	\begin{enumerate}
		\item $x_i$ belongs to $g_iK$.
		\item $c(x_0,x_i) \geq i /2$.
		\item For every $i \in \N\setminus \{0\}$, the path $\gamma_i$ is a $2^{-i}\epsilon$-quasi-gradient ray of $c$ starting at $x_{i-1}$ and passing through $x_i$.
	\end{enumerate}
	Let $i \in \N$.
	Assume that $x_i \in X$, $g_i \in G$ have been defined.
	By assumption $c$ does not belong to the set
	\begin{equation*}
		g_iA_{K, 2^{-(i+1)}\epsilon}.
	\end{equation*}
	Hence there exists a $2^{-(i+1)}\epsilon$-quasi-gradient ray $\gamma_{i+1} \colon \R_+ \to X$ for $c$ starting at $x_i$ and an element $u_i \in G$ such that $\gamma_{i+1} \cap g_iu_iK$ is non-empty and $d(g_iK, g_iu_iK) > 1$.
	We let $g_{i+1} = g_iu_i$ and denote by $x_{i+1}$ a point in $\gamma_{i+1} \cap g_iu_iK$.
	Since $x_i \in g_iK$ and $x_{i+1} \in g_iu_iK$, we have $\dist {x_i}{x_{i+1}} > 1$.
	However $\gamma_{i+1}$ is a quasi-gradient line.
	Hence
	\begin{equation*}
		c(x_i, x_{i+1}) \geq \dist {x_i}{x_{i+1}}  - 2^{-(i+1)}\epsilon \geq 1/2.
	\end{equation*}
	Using the induction hypothesis, we get
	\begin{equation*}
		c(x_0, x_{i+1}) \geq c(x_0, x_i) + c(x_i, x_{i+1}) \geq (i+1)/2.
	\end{equation*}
	Consequently $x_{i+1}$, $g_{i+1}$, and $\gamma_{i+1}$ satisfy the announced properties.

	Note that the sequence $(x_i)$ is unbounded. 
	Indeed otherwise $c(x_0, x_i)$ should be bounded as well.
	Thus we can build an infinite path $\gamma$ by concatenating the restriction of each $\gamma_i$ between $x_{i-1}$ and $x_i$.
	It follows from the construction that $\gamma$ is an $\epsilon$-quasi-gradient line for $c$, see \autoref{rem: gradient arc}.
	Moreover $\gamma$ intersects $g_iK$ for every $i \in \N$.
	One proves using the triangle inequality that $c$ belongs to the radial limit set.
\end{proof}

Let $K \subset X$ be a compact subset and $\epsilon \in \R_+^*$.
For every compact subset $F \subset X$, we define $U_{K, \epsilon}(F)$ to be the set of cocycles $b \in \bar X$ for which there is a cocycle $c \in A_{K, \epsilon}$ satisfying $\norm[F]{b - c} < \epsilon$.
Observe that $U_{K,\epsilon}(F)$ is an open subset of $\bar X$ containing $A_{K, \epsilon}$.

\begin{lemm}
\label{res: open set around complement rad}
	Let $K \subset X$ be a compact set and $\epsilon \in \R_+^*$.
	Fix a base point $o \in K$.
	There exist $r \in \R_+$ and a finite subset $S \subset G$, such that for every $T \geq \epsilon$, if $F$ stands for the closed ball of radius $T + 2r$ centered at $o$, then 
	\begin{equation*}
		U_{K, \epsilon}(F) \cap Go \subset S \left(\bigcup_{\substack{k \in G_K \\ \dist o{ko} \geq T}} \mathcal O_o(ko,r)\right).
	\end{equation*}
\end{lemm}

\begin{proof}
	Since the action of $G$ on $X$ is proper, the set 
	\begin{equation*}
		S = \set{u \in G}{d(K, uK) \leq 1}
	\end{equation*}
	if finite.
	We fix $r > 2 \diam K + 1$.
	Let $T \geq \epsilon$ and $F$ be the closed ball of radius $R = T+2r$ centered at $o$.
	Let $g \in G$ such that $go$ belongs to $U_{K, \epsilon}(F)$.
	We write $b = \iota(go)$ for the corresponding cocycle.
	By definition, there is $c \in A_{K, \epsilon}$ such that $\norm[F]{b-c} < \epsilon$.
	Observe first that $\dist o{go} > R - \epsilon$.
	Indeed the map $x \mapsto b(x,go)$ admits a global minimum at $go$, while there exists a $c$-gradient line starting at $go$.
	We cannot have at the same time $\dist o{go} \leq R - \epsilon$ and $\norm[F]{b-c} < \epsilon$.
	In particular, $g \notin S$.
	
	Since $c \in A_{K, \epsilon}$, there exists $x \in K$, such that for every $\epsilon$-quasi-gradient ray $\gamma \colon \R_+ \to X$ for $c$, starting at $x$, if $\gamma$ intersects $uK$ for some $u \in G$, then $u \in S$.
	Consider now a geodesic $\alpha\colon \intval 0\ell \to X$ from $x$ to $go$.
	We denote by $s \in \intval 0\ell$, the largest time such that the point $y = \alpha(s)$ belongs to $SK$.
	We now denote by $t \in \intval s\ell$, the smallest time such that the point $z = \alpha(t)$ lies in $hK$ for some $h \in G\setminus S$ (such a time $t$ exists since $\alpha(\ell)$ belongs to $gK$).
	It follows from the construction that $h$ can be written $h = uk$ with $u \in S$ and $k \in G_K$.
	Moreover $y \in uK$.
	Observe that $\gro y{go}z = 0$, while $\dist y{uo} \leq r/2$ and $\dist z{uko}\leq r/2$.
	The triangle inequality yields $\gro {uo}{go}{uko} \leq r$, i.e. $go$ belongs to $u\mathcal O_o(ko, r)$.
	
	We are left to prove that $\dist o{ko} \geq T$.
	By construction $\dist o{uo} \leq r$.
	Thanks to the triangle inequality, it suffices to show that $\dist oz \geq R$.
	Assume on the contrary that  $\dist oz < R$.
	In particular, both $x$ and $z$ belong to $F$.
	Since $b$ and $c$ differ by at most $\epsilon$ on $F$, we get that $c(x,z) \geq \dist xz - \epsilon$.
	Hence any geodesic from $x$ to $z$ is an $\epsilon$-quasi-gradient arc for $c$.
	If we concatenate this path with a gradient ray for $c$ starting at $z$, we obtain an $\epsilon$-quasi-gradient ray for $c$ starting at $x$ and intersecting $hK$ with $d(K, hK) > 1$.
	This contradicts the fact that $c$ belongs to $A_{K, \epsilon}$, and completes the proof.
\end{proof}

\begin{prop}
\label{res: spr gives full measure}
	If the action of $G$ on $X$ is strongly positively recurrent, then there is a $G$-invariant, $\omega_G$-conformal density which gives full measure to the radial limit set $\Lambda_{\rm rad}(G)$.
\end{prop}

\begin{proof}
 	By definition, there is a compact subset $K \subset X$ such that $\omega_{G_K} < \omega_G$.
	We fix once and for all a base point $o \in K$.
	The argument relies on Patterson's construction recalled in the proof of \autoref{res: existence twisted ps} with $H = G$ and $\chi$ the trivial morphism.
	 In particular, $\mathcal Q(s)$ stands for the weighted Poincaré series defined in (\ref{eqn: weighted patterson series}).
	 For every $s > \omega_G$, we consider the density $\nu^s = (\nu_x^s)$ defined as in (\ref{res: patterson approx}).
	 As we explained there is a sequence $(s_n)$ converging to $\omega_G$ from above such that $\nu^{s_n}$ converges to a $G$-invariant, $\omega_G$-conformal density $\nu$ supported on $\partial X$.
	 	 
	 Let $\eta > 0$ such that $\omega_{G_K} < \omega_G - \eta$.
	 The weight $\theta$ used to construct $\nu$ is slowing increasing.
	 More precisely, according to \ref{enu: patterson - slowing growing weight} there exists $t_0$ such that for every $t \geq t_0$ and $u \in \R_+$ we have $\theta(t + u) \leq e^{\eta u}\theta(t)$.
	 Let $\epsilon > 0$.
	 Let $r \in \R_+$ and $S \subset G$ be the data provided by \autoref{res: open set around complement rad} applied with $K$ and $\epsilon$.
	For every $T \in \R_+$, we write $F_T$ for the closed ball of radius $R =T+2r$ centered at $o$.
	Let $s > \omega_G$ and $T \geq \max\{t_0, \epsilon\}$.
	In view of \autoref{res: open set around complement rad},	we have
	\begin{equation*}
		\nu^s_o\left(U_{K,\epsilon}(F_T)\right)
		\leq \card S \sum_{\substack{k \in G_K \\ \dist o{ko} \geq T}} \nu^s_o\left(\mathcal O_o(ko, r) \right).
	\end{equation*}
	Let us estimate the measures of the shadows in the sum.
	Let $k \in G_K$, such that $\dist o{ko} \geq T$.
	Any element $g \in G$ such that $go \in \mathcal O_o(ko,r)$ can be written $g = ku$ with $u \in G$ and 
	\begin{equation*}
		 \dist o{ko} + \dist o{uo} - 2r \leq \dist o{go} \leq \dist o{ko} + \dist o{uo}.
	\end{equation*}
	Unfolding the definition of $\nu^s$, we get
	\begin{equation}
	\label{eqn: spr gives full measure}
		\nu^s_o\left(\mathcal O_o(ko, r) \right)
		\leq  \frac{e^{2sr}e^{-s \dist o{ko}}}{\mathcal Q(s)}\sum_{u \in G} \theta(\dist o{go}) e^{-s \dist o{uo}}.
	\end{equation}
	Observe that if $\dist o{uo} \geq t_0$, then it follows from our choice of $t_0$ that 
	\begin{equation*}
		\theta\left(\dist o{go}\right)
		\leq \theta\left( \dist o{ko} + \dist o{uo}\right)
		\leq e^{\eta\dist o{ko}} \theta\left(\dist o{uo}\right).
	\end{equation*}
	Otherwise, since $\dist o{ko}\geq T \geq t_0$, we have
	\begin{equation*}
		\theta\left(\dist o{go}\right)
		\leq \theta\left(t_0 + \dist o{uo}+  \dist o{ko} -t_0 \right)
		\leq e^{\eta\dist o{ko}} \theta\left(t_0\right).
	\end{equation*}
	We break the sum in (\ref{eqn: spr gives full measure}) according to the length of $u$ and get 
	\begin{equation*}
		\nu^s_o\left(\mathcal O_o(ko, r) \right)
		\leq \frac{e^{2sr}e^{-(s-\eta) \dist o{ko}}}{\mathcal Q(s)} \left[\theta(t_0)\Sigma_1(s) + \Sigma_2(s)\right],
	\end{equation*}
	where
	\begin{align*}
		\Sigma_1(s) & = \sum_{\substack{u \in G \\ \dist o{uo} \leq t_0}} e^{-s \dist o{uo}}, \quad \text{and}\quad \\
		\Sigma_2(s) & =  \sum_{\substack{u \in G \\ \dist o{uo} > t_0}} \theta(\dist o{uo}) e^{-s \dist o{uo}}.
	\end{align*}
	Note that $\Sigma_1(s)$ is a finite sum that does not depend on $k$, while $\Sigma_2(s)$ is the remainder of the series $\mathcal Q(s)$.
	Hence
	\begin{equation*}
		\nu^s_o\left(\mathcal O_o(ko, r) \right)
		\leq e^{2sr}
		\left[ \frac {\theta(t_0)}{\mathcal Q(s)} \Sigma_1(s) + 1\right] e^{-(s-\eta) \dist o{ko}},
	\end{equation*}
	Summing over all long elements $k \in G_K$, we get
	\begin{equation*}
		\nu^s_o\left(U_{K,\epsilon}(F_T)\right)
		\leq \card S e^{2sr}
		\left[ \frac {\theta(t_0)}{\mathcal Q(s)} \Sigma_1(s) + 1\right] 
		\sum_{\substack{k \in G_K \\ \dist o{ko} \geq T}} e^{-(s-\eta) \dist o{ko}}.
	\end{equation*}
	Note that $\Sigma_1(s)$ is bounded, while $\mathcal Q(s)$ diverges to infinity.
	Since $U_{K,\epsilon}(F_T)$ is an open subset of $\bar X$, we can pass to the limit and get
	\begin{equation*}
		\nu_o\left(U_{K,\epsilon}(F_T)\right)
		\leq \card S e^{2\omega_Gr}
		\sum_{\substack{k \in G_K \\ \dist o{ko} \geq T}} e^{-(\omega_G-\eta) \dist o{ko}}.
	\end{equation*}
	The sum corresponds to the remainder of the Poincaré series of $G_K$ at $s = \omega_G - \eta$.
	However $\omega_G - \eta > \omega_{G_K}$.
	Hence this series converges, and its reminder tends to zero when $T$ approaches infinity.
	Consequently, for every $\epsilon > 0$,
	\begin{equation*}
		\nu_o\left(\bigcap_{T \geq 0}U_{K,\epsilon}(F_T)\right) = 0.
	\end{equation*}
	By construction the set $A_{K, \epsilon}$ is contained in $U_{K, \epsilon}(F_T)$ for every $T \in \R_+$.
	It follows from \autoref{res: inclusion radial limit set} that 
	\begin{equation*}
		\partial X \setminus \Lambda_{\rm rad} (G)
		\subset G \left( \bigcup_{\epsilon > 0} \bigcap_{T \geq 0}U_{K, \epsilon}(F_T)\right).
	\end{equation*}
	Since $G$ is countable we conclude that $\nu_o(\partial X \setminus \Lambda_{\rm rad}(G)) = 0$.
	Recall that $\nu_o$ is supported on $\partial X$, thus $\nu_o(\Lambda_{\rm rad}(G)) = 1$.
\end{proof}


\bigskip
\noindent
\emph{R\'emi Coulon} \\
Université de Bourgogne, CNRS \\
IMB - UMR 5584 \\
F-21000 Dijon, France\\
\texttt{remi.coulon@cnrs.fr} \\
\texttt{http://rcoulon.perso.math.cnrs.fr}

\end{document}